\documentclass[11pt,reqno]{amsart}

\usepackage{amscd,amsmath,amssymb,amsfonts,verbatim,graphicx,enumitem,multirow}
\usepackage[colorlinks]{hyperref}
\usepackage{diagbox}
\newcommand{\cY}{\mathcal{Y}}
\newcommand{\cC}{\mathcal{C}}
\newcommand{\C}{\mathbb{C}}

\newcommand{\ddbar}{i\partial\overline{\partial}}
\newcommand{\ddbars}{i\partial_\sigma\overline{\partial}_\sigma}
\newcommand{\p}{\partial}

\renewcommand{\epsilon}{\varepsilon}
\renewcommand{\leq}{\leqslant}
\renewcommand{\geq}{\geqslant}

\def\Xint#1{\mathchoice
{\XXint\displaystyle\textstyle{#1}}
{\XXint\textstyle\scriptstyle{#1}}
{\XXint\scriptstyle\scriptscriptstyle{#1}}
{\XXint\scriptscriptstyle\scriptscriptstyle{#1}}
\!\int}
\def\XXint#1#2#3{{\setbox0=\hbox{$#1{#2#3}{\int}$ }
\vcenter{\hbox{$#2#3$ }}\kern-.6\wd0}}

\def\dashint{\Xint-}

\newtheorem{theorem}{Theorem}[section]
\newtheorem*{mtheorem}{Main Theorem}
\newtheorem{lemma}[theorem]{Lemma}
\newtheorem{corollary}[theorem]{Corollary}
\newtheorem{proposition}[theorem]{Proposition}

\newcounter{Ass}

\theoremstyle{definition}
\newtheorem{definition}[theorem]{Definition}
\newtheorem{remark}[theorem]{Remark}
\newtheorem{convention}[theorem]{Convention}
\newtheorem{claim}[theorem]{Claim}
\newtheorem{subclaim}[theorem]{Sub-Claim}
\newtheorem{ass}[Ass]{Assumption}

\numberwithin{equation}{section}

\begin{document}
	
\title{A continuous cusp closing process\\for negative K\"ahler-Einstein metrics}
	
\begin{abstract}
We give an example of a family of smooth complex algebraic surfaces of degree $6$ in $\mathbb{CP}^3$ developing an isolated elliptic singularity. We show via a gluing construction that the unique Kähler-Einstein metrics of Ricci curvature $-1$ on these sextics develop a complex hyperbolic cusp in the limit, and that near the tip of the forming cusp a Tian-Yau gravitational instanton bubbles off.
\end{abstract}
	
\author{Xin Fu}
\address{School of Science, Institute for Theoretical Sciences, Westlake University, Hangzhou 310030, China}
\email{fuxin54@westlake.edu.cn}
	
\author{Hans-Joachim Hein}
\address{Mathematisches Institut, Universit\"at M\"unster, 48149 M\"unster, Germany}
\email{hhein@uni-muenster.de}
	
\author{Xumin Jiang}
\address{School of Sciences, Great Bay University, Dongguan 523000, China \newline \hspace*{9pt}
Great Bay Institute for Advanced Study, Dongguan 523000, China}
\email{xjiang@gbu.edu.cn}

\date{\today}
	
\maketitle
	
\markboth{A continuous cusp closing process for negative K\"ahler-Einstein metrics}{Xin Fu, Hans-Joachim Hein and Xumin Jiang}
	
\setcounter{tocdepth}{3}
\tableofcontents
	
\section{Introduction}

\subsection{Statement of the Main Theorem}

Consider the family $\{\mathcal{X}_\sigma\}_{\sigma\in\Delta}$ of degree $6$ algebraic surfaces in $\mathbb{CP}^3$ given by the projective closures of the affine sextics
\begin{equation}
\{(z_1,z_2,z_3) \in \mathbb{C}^3: (z_1^6 + z_1^3) + (z_2^6 + z_2^3) + (z_3^6 + z_3^3) = \sigma\}.
\end{equation}
Then $\mathcal{X}_\sigma$ is smooth for $0 < |\sigma| \ll 1$ and the singular set $\mathcal{X}_0^{sing}$ is exactly the origin in $\mathbb{C}^3$. Since $\mathcal{X}_\sigma$ is a smooth surface of general type for $\sigma \neq 0$, it admits a unique Kähler-Einstein metric $\omega_{KE,\sigma}$ of Ricci curvature $-1$ by the Aubin-Yau theorem \cite{Aubin,Yau}. It is also known by \cite{BG,DFS,rkoba,Song} that on $\mathcal{X}_0^{reg}$ there exists a unique complete Kähler-Einstein metric $\omega_{KE,0}$ of Ricci curvature $-1$. By \cite{Song} we have that $\omega_{KE,\sigma} \to \omega_{KE,0}$ as $\sigma \to 0$, locally smoothly on $\mathbb{C}^3 \setminus {B_\varepsilon(0)}$ for every fixed $\varepsilon>0$.  By \cite{DFS,FHJ} the end of $(\mathcal{X}_0^{reg},\omega_{KE,0})$ is asymptotically complex hyperbolic at an optimal rate. This means it is asymptotic to the end of a finite-volume quotient of the complex hyperbolic plane, or in other words, of the unit ball in $\mathbb{C}^2$ equipped with the Bergman metric. In particular, the end of $(\mathcal{X}_0^{reg},\omega_{KE,0})$ is asymptotically locally symmetric, but, of course, $(\mathcal{X}_0^{reg}, \omega_{KE,0})$ cannot be a locally symmetric space.

We view the smoothing of the cuspidal Einstein manifold $(\mathcal{X}_0^{reg},\omega_{KE,0})$ by the family $(\mathcal{X}_\sigma,\omega_{KE,\sigma})$ as a kind of cusp closing process, somewhat analogous to Thurston's hyperbolic Dehn surgery and its many generalizations \cite{And,Bamler,HaJa,Hummel,Thu}, but in another way also quite different because here we have a continuous path of metrics on a fixed smooth $4$-manifold. Thus, the picture in our case is actually closer to the usual cuspidal degenerations of hyperbolic Riemann surfaces, which do not exist in higher-dimensional hyperbolic or complex-hyperbolic geometry due to Mostow rigidity. The most interesting difference from the Riemann surface case, and also to some extent from the hyperbolic Dehn surgery situation, is that in our case there is a nontrivial amount of curvature and topology disappearing into the tip of the cusp. Our purpose in this paper is to make this observation precise.

\begin{mtheorem}
Fix any $R > 0$. For $0 < |\sigma| \ll 1$ sufficiently small relative to $R$, restrict the metric $\omega_{KE,\sigma}$ to $\mathcal{X}_\sigma \cap B_{|\sigma|^{1/3}R}(0)$, multiply it by $|{\log |\sigma|}|^{3/2}$ and push it forward under the map
\begin{align}(z_1,z_2,z_3) \mapsto \sigma^{-\frac{1}{3}}(z_1(1+z_1^3)^{\frac{1}{3}}, z_2(1+z_2^3)^{\frac{1}{3}}, z_3(1+z_3^3)^{\frac{1}{3}}).\end{align}
Then as $\sigma \to 0$, the $C^{0}$ distance of the resulting Einstein metric on $\{z_1^3 + z_2^3 + z_3^3 = 1\} \cap B_{R/2}(0)$ to the Tian-Yau gravitational instanton is $O_{R,\varepsilon}(|{\log |\sigma|}|^{-1+\varepsilon})$ for all $\varepsilon>0$. This is the only bubble, and it accounts for the total loss of $L^2$-curvature and Euler characteristic $(108-9)$ in the degeneration.
\end{mtheorem}

\begin{figure}[!ht]
\caption{$(\mathcal{X}_\sigma, \omega_{KE,\sigma})$ for $0 < |\sigma| \ll 1$. The Main Theorem describes the red part.}\label{fig:mainthm}
\vspace{3mm}
\begin{center}
\includegraphics[width=170mm]{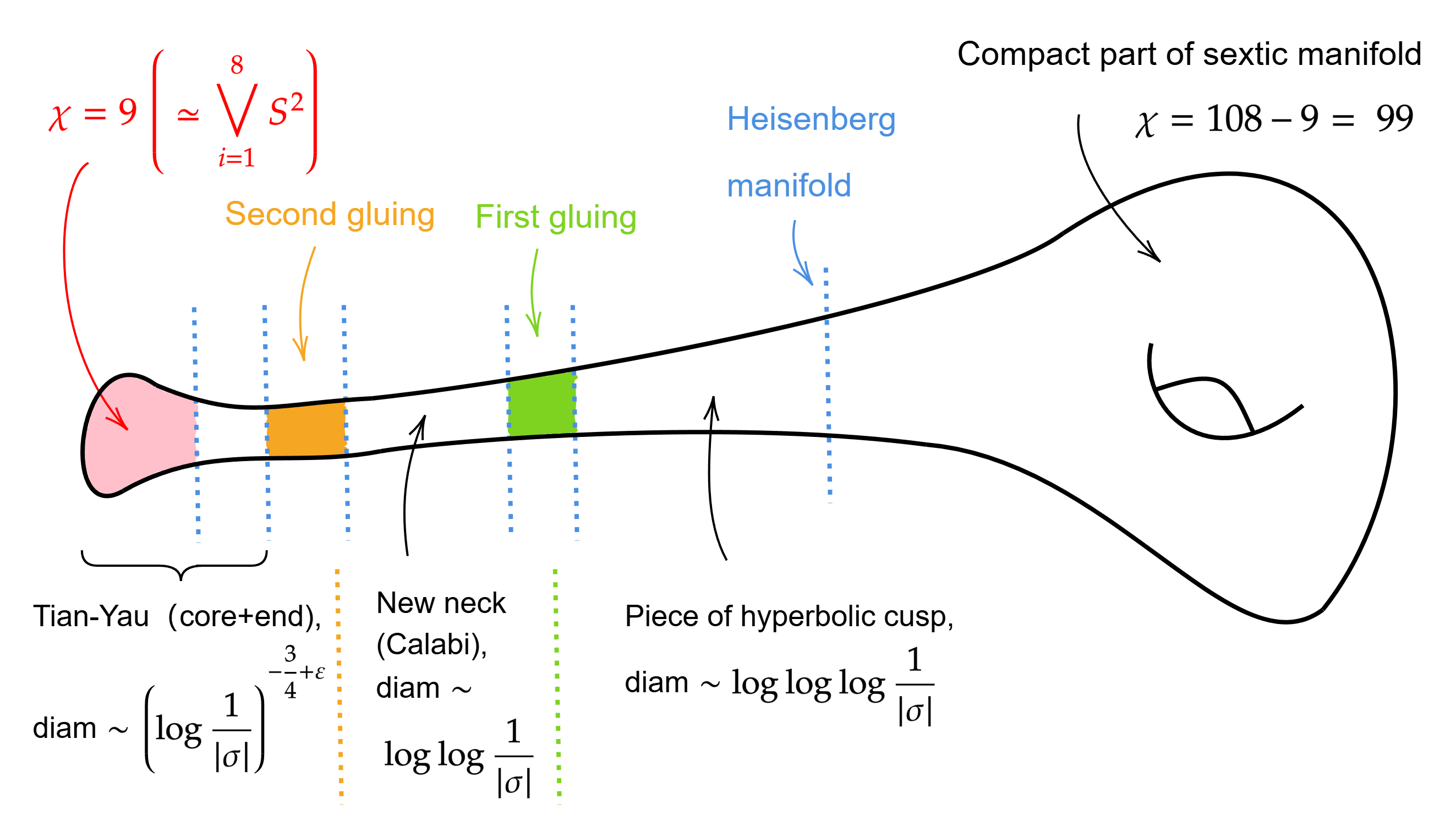}
\end{center}
\end{figure}

The Tian-Yau gravitational instanton is a complete Ricci-flat Kähler, hence hyper-Kähler metric of volume growth $O(r^{4/3})$ and curvature decay $O(r^{-2})$ on the smooth complex surface $\{z_1^3 + z_2^3 + z_3^3 = 1\}$. Its construction goes back to \cite{TY}. There exists more than one such metric even modulo automorphism and scaling. The one that appears in our theorem is characterized by being globally $\ddbar$-exact and by being asymptotic to a specific model Kähler form at infinity. We emphasize that this model form is completely determined. In particular, it does not even depend on an unknown scaling factor. 

See Figure \ref{fig:mainthm} for a rough sketch of the geometry and topology of the main regions of $(\mathcal{X}_\sigma,\omega_{KE,\sigma})$ shortly before the limit. The Main Theorem is of course proved by a gluing construction, which also yields estimates in every other region of $\mathcal{X}_\sigma$. In particular, we recover the locally smooth convergence of $\omega_{KE,\sigma}$ to $\omega_{KE,0}$ over $\mathcal{X}_0^{reg}$ from \cite{Song}. However, most likely none of our estimates are sharp.

\subsection{Possible generalizations}

We expect that all of our work in this paper goes through
\begin{itemize}
    \item[$\bullet$] for any flat family $\{\mathcal{X}_\sigma\}_{\sigma \in \Delta} \subset \mathbb{CP}^N \times \Delta$ of surfaces of general type, embedded by a fixed power of their canonical bundles, such that $\mathcal{X}_\sigma$ is smooth for $\sigma \neq 0$,
    \item[$\bullet$] all of the singularities of $\mathcal{X}_0$ are cones over elliptic curves,
    \item[$\bullet$] and the finite group ${\rm Aut}(\mathcal{X}_0)$ acts transitively on the set of singularities.
\end{itemize}
Under the first two bullet points it is known from \cite[Ch 9]{Pinkham} that the cones are of degree at most $9$. Moreover, any smoothing of such a cone is given by a del Pezzo surface of the same degree containing the elliptic curve as an anticanonical divisor. (For degrees $\geq 5$ this is also proved in \cite[Ch 9]{Pinkham}, while for degrees $\leq 4$ it follows from the realization of these cones as quasi-homogeneous complete intersections in \cite[Satz 1.9]{Saito} and from the standard deformation theory of complete intersections \cite{KS}.) We have chosen not to state our theorem in this generality because this would have forced us to introduce a lot of extraneous notation and technical arguments. The third bullet point is a much more pressing issue: as an artifact of our method, we are currently unable to deal with multiple independent singularities. Roughly speaking, each singularity contributes a $1$-dimensional obstruction space to the gluing, and we are currently not dealing with these obstructions in a systematic way---we exploit one ``accidental'' global degree of freedom, which restricts us to the case of a single singularity modulo ${\rm Aut}(\mathcal{X}_0)$.

In principle, the same kind of geometry also occurs in higher dimensions: every affine cone over a projective Calabi-Yau manifold admits Kähler-Einstein model metrics of cuspidal geometry. However, if the Calabi-Yau is flat, these cones are never smoothable in dimension $3$ and higher \cite{Kova}; in particular, they never occur as the infinity divisor in a Tian-Yau manifold that could model the smoothing near the tip of the cusp. For non-flat Calabi-Yaus, this smoothing issue disappears but the cuspidal model metrics have unbounded curvature, so it is impossible to prove (as done in \cite{DFS, FHJ} in the flat case) that the global Kähler-Einstein metric on $\mathcal{X}_0^{reg}$ is asymptotic to the cusp model. This is a cuspidal version of the well-known orbifold vs. non-orbifold dichotomy for isolated \emph{conical} singularities of Kähler-Einstein metrics. The non-orbifold conical case was solved in \cite{HS} using Donaldson-Sun theory \cite{DS2}, bypassing the $C^2$ estimate in the theory of the complex Monge-Amp\`ere equation. In the cuspidal case one would either need a cuspidal Donaldson-Sun theory, or a very strong $C^2$ estimate in unbounded curvature.

In a different direction, in dimension $2$, cones over elliptic curves are not the only singularities that can occur on canonically polarized degenerations of smooth surfaces of general type. Such singularities were classified in \cite[Thm 4.24]{KSB}. Examples include normal crossing singularities (see \cite{Mandel} for progress in this direction) as well as the cusps of Hilbert modular surfaces \cite[pp.54--57]{rkoba}. The great advantage of the $2$-dimensional case is that the natural cuspidal model metrics do have bounded curvature.

\subsection{Outline of the proof}\label{sec:outline}

There is a very extensive literature on gluing constructions in geometric analysis and more specifically in Kähler geometry, which we will not attempt to survey here. Tian-Yau spaces were used as singularity models in \cite{biqguen} and \cite{HSVZ} but the settings of these two papers are rather different from ours. In fact, our setting is in some sense a \emph{cubic analog} of the classical smoothing of Kähler-Einstein surfaces with nodal singularities via Eguchi-Hanson/Stenzel gravitational instantons. This gluing construction was carried out by Spotti \cite{spotti} and independently (in greater generality) by Biquard-Rollin \cite{biqroll}. The initial idea of our proof, dating back roughly 10 years, was that something similar can perhaps be done in the cubic situation, based on the following observation.

We identify the singularity $\{z_1^3 + z_2^3 + z_3^3 = 0\} \subset \mathbb{C}^3$ with the contraction of the zero section of the total space of a negative line bundle $L$ over the corresponding elliptic curve $E \subset \mathbb{CP}^2$. On $L$ we have a unique (up to scaling) Hermitian metric $h$ whose curvature form is minus the flat Kähler form $\omega_E$ representing the class $2\pi c_1(L')$, where $L' \to E$ denotes the line bundle dual to $L$. Thus, as a function on the cubic cone, $h = e^{-\varphi}|z|^2$, where $|z|$ is the standard Euclidean radius and $\varphi$ is a $0$-homogeneous function which, viewed as a function $\varphi: E \to \mathbb{R}$, satisfies $\ddbar\varphi =  \omega_{FS}|_E - \omega_E$. Then the asymptotic model of the Tian-Yau metric is $\ddbar (\log h)^{3/2}$ on $\{h > 1\}$ whereas the asymptotic model of the cusp metric is $-3 \ddbar \log(-{\log h})$ on $\{h < 1\}$. Define $t := \log h$. Since the complete Tian-Yau metric, i.e., the relevant gravitational instanton, lives on the smooth surface $\{z_1^3 + z_2^3 + z_3^3 = 1\}$ and since we are interested in the degeneration $\{z_1^3 + z_2^3 + z_3^3 = \sigma\}$ with $\sigma \to 0$, it is natural to pull back by the map $z \mapsto \sigma^{-1/3}z$ and thus replace $t$ by $t - T$ in the Tian-Yau potential, $T := (2/3)\log |\sigma|$. Then
\begin{align}\label{eq:naive_gluing}
    \omega_{cusp} = -\frac{3}{t}\omega_E + \frac{3}{t^2}i\partial t \wedge \overline{\partial} t\;\,(t < 0), \;\, \omega_{TY} = \frac{3}{2}(t-T)^{\frac{1}{2}}\omega_E + \frac{3}{4}(t-T)^{-\frac{1}{2}}i\partial t \wedge \overline{\partial}t\;\,(t > T).
\end{align}
Thus, both the tangential and the radial metric coefficients match up at $t = (2/3)T$ provided that we also rescale $\omega_{TY}$ by a factor of  $const \cdot |T|^{-3/2}$. This agreement is surprisingly good. In fact, it lets us write down a pre-glued metric on $\mathcal{X}_{\sigma}$ whose Ricci potential has sup norm $O(1)$ as $\sigma \to 0$.

However, to enter the gluing regime, it turns out that $O(1)$ needs to be improved to $o(1)$. Here a new idea is needed. By solving a Calabi ansatz we show that $\omega_{cusp}$ belongs to a $1$-parameter family $\omega_b$ ($b \in \mathbb{R})$ of radial Kähler-Einstein metrics on the cubic cone such that $\omega_{0} = \omega_{cusp}$ and $\omega_b$ undergoes a ``geometric transition'' at $b = 0$. For $b > 0$, $\omega_b$ extends to the total space of $L$ with an edge singularity along the zero section of $L$. These edge metrics were introduced in \cite{biqguen} and \cite{FHJ}. However, in this paper the case $b < 0$ is more relevant. For $b < 0$, $\omega_b$ is only defined on a subset $t > T$ of the cubic cone, $T \sim - const \cdot |b|^{-1/3}$, and it has a \emph{horn} singularity: as $t \to T$, $\omega_b \sim const \cdot |b|^{1/2} \cdot \ddbar (t-T)^{3/2}$.
This is now a perfect match for the end of the Tian-Yau space, leading to a pre-glued metric whose Ricci potential can be $O(|T|^{-(3/2)+\varepsilon})$ for any $\varepsilon>0$ if the gluing is done sufficiently close to $t = T$.

Since this decay of the Ricci potential almost matches the scaling factor of the Tian-Yau metric, which is the smallest scale in the construction, it follows from the maximum principle and from Savin's small perturbation theorem \cite{Savin} that the difference between $\omega_{KE,\sigma}$ and the pre-glued metric goes to zero everywhere except on the Tian-Yau cap. To get control on the cap we need to develop a weighted Hölder space theory as in \cite{biqroll,spotti}, but unlike in \cite{biqroll,spotti} the gluing is obstructed. This is because we are now dealing with three neck regions rather than one: the end of the Tian-Yau space, the cusp of $\mathcal{X}_0^{reg}$, and the new neck coming from the Calabi ansatz $\omega_b$ that interpolates between these two. There is a solution to the linearized PDE on the new neck that approaches $1$ on the Tian-Yau side and $0$ on the cusp side, and we have been unable to rule out this solution using any choice of weight. We use an ad hoc trick to get around this issue (which however prevents us from dealing with multiple independent singularities): the Ricci potential is only defined modulo a constant, and while changing this constant is the same as adding a constant to the solution of the Monge-Amp\`ere equation, the ``Einstein modulo obstructions'' metric in the sense of \cite{OzuchThesis} reacts in a nontrivial way to this change.

\subsection{Acknowledgments}

We thank O. Biquard and H. Guenancia for some very helpful discussions about the Calabi ansatz, B. Ammann, G. Tian and V. Tosatti for pointing out references \cite{Hummel}, \cite{CYcusp,TY0} and \cite{Kova}, respectively, and an anonymous referee for suggestions that greatly improved our exposition. XF was supported by National Key R\&D Program of China 2024YFA1014800
 and NSFC No. 12401073. HJH was partially supported by the German Research Foundation (DFG) under Germany's Excellence Strategy EXC 2044-390685587 ``Mathematics M\"unster:~Dynamics-Geometry-Structure" as well as by the CRC 1442 ``Geometry:~Deformations and Rigidity'' of the DFG.

\section{Building blocks of the gluing}

\subsection{Degenerations of projective surfaces of general type}

In this subsection, we give a family of canonically polarized surfaces $\mathcal X$ over the disk $\Delta$. Let $\mathcal X$ be a family of degree $6$ surfaces (hence canonically polarized) in $\mathbb{CP}^3$ as follows:
\begin{align}\label{eq:sextics}
\mathcal X_\sigma:=\{(Z_1^6+Z^3_1Z^3_0)+(Z_2^6+Z^3_2Z^3_0)+(Z_3^6+Z^3_3Z^3_0)=\sigma Z_0^6\}.
\end{align}
When $Z_0\neq 0$, in the affine coordinates $z_1=\frac{Z_1}{Z_0},z_2=\frac{Z_2}{Z_0},z_3=\frac{Z_3}{Z_0}$, $\mathcal X_\sigma$ is defined by \begin{equation}\label{SingModel}
(z_1^6+z_1^3)+(z_2^6+z^3_2)+(z_3^6+z_3^3)=\sigma.
\end{equation}
It is easy to see that this affine surface is smooth for $0 < |\sigma| \ll 1$, that the origin is its unique singular point for $\sigma = 0$, and that this singularity is locally analytically isomorphic (via $z_i \mapsto z_i(1 + z_i^3)^{1/3}$ for $i = 1,2,3$) to the singularity defined by the equation 
\begin{align}z_1^3+z_2^3+z_3^3=0.\end{align}
Also, when $Z_0=0$, $\mathcal X_\sigma$  is smooth for every $\sigma$. For instance, by symmetry, let us assume $Z_1\neq 0$. Then we may take affine coordinates $w_1=\frac{Z_0}{Z_1}, w_2=\frac{Z_2}{Z_1},w_3=\frac{Z_3}{Z_1}$. Then the singular locus of $\mathcal X_\sigma$ will be the common zeros of the following system of equations:
\begin{align}\label{eq:singularlocus}
\begin{split}
   1+w_1^3+w_2^6+w_2^3w_1^3+w_3^6+w_3^3w_1^3-\sigma w_1^6=0,\\
   3w_1^2 + 3w_2^3 w_1^2 + 3w_3^3 w_1^2 - 6\sigma w_1^5 = 0,\\
    6w_2^5+3w_2^2w_1^3=0,\\
6w_3^5+3w_3^2w_1^3=0.
\end{split}
\end{align}
Now observe that if  $w_1=0$, then it is easy to see that the above system has no common zeros for all values of $\sigma$. Hence when $Z_0=0$, $\mathcal X_\sigma$ is smooth for all $\sigma$.

Obviously $K_{\mathcal X_\sigma}\cong \mathcal O(2)|_{\mathcal X_\sigma}$. So let us denote the zero locus of the linear homogeneous polynomial $Z_0$ on $\mathcal{X}_\sigma$ by $\mathcal{D}_\sigma$. Also set $\mathcal Y_\sigma:=\mathcal X_\sigma\setminus \mathcal{D}_\sigma$. Then 
\begin{align}\label{eq:sextics_affine}
\mathcal Y_\sigma=\{(z_1^6+z_1^3)+(z_2^6+z^3_2)+(z_3^6+z_3^3)=\sigma\}\subset \mathbb C^3.
\end{align}
From what we said it is clear that there is a section $\Omega$ of the relative canonical bundle $K_{\mathcal X_\sigma/\Delta}$ whose restriction to $\mathcal Y_\sigma$ is a nonvanishing holomorphic volume form $\Omega_\sigma$ with a zero of order $2$ at $\mathcal{D}_\sigma$.

\begin{remark}\label{nonVHVunique}
We remark that up to scaling $\Omega_\sigma$ is the unique holomorphic volume form on the affine manifold $\mathcal{Y}_\sigma$ with zeros or poles along $\mathcal{D}_\sigma$.  Indeed, if $\hat \Omega_\sigma$ is any other such form, then $H_\sigma:=\Omega_\sigma/\hat \Omega_\sigma$ is a nowhere vanishing holomorphic function on $\mathcal Y_\sigma$ with zeros or poles along $\mathcal{D}_\sigma$.  Thus, either $H_\sigma$ or $1/H_\sigma$ extends to a holomorphic function on $\mathcal{X}_\sigma$ and is therefore constant. More generally, in order to conclude $H_\sigma = const$ it suffices to assume that $\hat\Omega_\sigma$ has at most polynomial growth near $\mathcal{D}_\sigma$. 
\end{remark}

By equation \eqref{SingModel}, we know that near the singularity, the family $\mathcal X_\sigma$ is locally analytically isomorphic to the following family of Tian-Yau spaces (this terminology will become clear in Section \ref{TYConstruction})
\begin{equation}\label{TYspace}
TY_\sigma:=\{z_1^3+z_2^3+z_3^3=\sigma\}.
\end{equation}
We remark here that $TY_0$ is obviously the affine cone over a smooth elliptic curve $E \subset \mathbb{CP}^2$. Thus, for the specific singularity $z=0$ in $TY_0$, there is a log resolution of singularities \begin{equation}\label{resolution}\pi:\hat L\to  TY_0\end{equation} with exceptional divisor $E$. Here $L \to E$ denotes a negative line bundle, $\hat{L}$ denotes the total space of $L$ and $\pi$ is the contraction of the zero section of $\hat{L}$. We also remark that $TY_\sigma$ is a family of affine varieties in $\mathbb C^3\times \mathbb C\subset\mathbb {CP}^3 \times \mathbb C$  and it can be  compactified in $\mathbb{CP}^3\times\mathbb C$ by adding the  divisor
\begin{equation}\label{infdiv}
E = \{Z_1^3+Z_3^3+Z_3^3=0\}
\end{equation}
independently of $\sigma$,  where $Z_i$ are the homogeneous coordinates on $\mathbb {CP}^3$.  

Next, we fix an identification of $TY_\sigma$ and $\mathcal {Y}_\sigma$ locally near the singularity.

\begin{definition}
We define a local analytic isomorphism of the  family $\Psi_\sigma: \mathcal Y_\sigma\to TY_\sigma $ as follows:
\begin{equation}\label{Psi}
    \Psi_\sigma(z_1,z_2,z_3):=(z_1(1+z_1^3)^{\frac{1}{3}},z_2(1+z_2^3)^{\frac{1}{3}},z_3(1+z_3^3)^{\frac{1}{3}}).
\end{equation}
From the implicit function theorem, it is clear that $\Psi_\sigma$ is invertible near $(z_1,z_2, z_3) =(0,0,0).$
\end{definition}

We will also need to identify $TY_0$ with $TY_1$ in some fixed neighborhood of infinity of the two spaces.  Of course, this can only be done diffeomorphically, not biholomorphically.
 
\begin{lemma}[{\cite[Lemma 5.5]{CH}}]\label{l:normproj}
There exist $R>0$ and a smooth function
$\nu:\C^{3}\setminus
\overline{B}_R\to\C$  such that 
\begin{align}\sum_{i=1}^{3}(z_{i}+\nu(z)\overline{z}_{i}^{2})^{3}=1\end{align}
for all $z\in TY_0 \setminus \overline{B}_R$.
Moreover, 
$\partial^k \nu(z) = O(|z|^{-4-k})$ as $|z|\to\infty$ for all $k \geq 0$.
\end{lemma}

\begin{proof}
Taking complex coordinates $z=(z_1,z_2,z_3)$ on $S^5 \subset \C^{3}$, define a function $f$ by
\begin{equation}
f:S^5\times[0,\infty)\times\C\to\C,
\;\, f(z,r,y)=
3y\left(\sum_{i=1}^3|z_{i}|^{4}\right)+3y^{2}\left(\sum_{i=1}^3|z_{i}|^{2}\overline{z}^{3}_{i}\right)+
y^{3}
\left(\sum_{i=1}^3\overline{z}_{i}^{6}\right)-r^{3},
\end{equation}
and fix any point $p=(p_{1},p_2,p_3)\in S^{5}$. Since 
\begin{align}f(p,0,0)=0,\;\,
\frac{\p f}{\p y}(p,0,0)=3\sum_{i=1}^{3}|p_{i}|^{4}\neq 0,
\end{align}
the implicit function theorem asserts the existence of a unique
smooth function $g_{p}$, defined in some open neighborhood
$U_{p} \times [0,\epsilon_p)$ of $(p,0),$ such that
$g_{p}(p,0)=0$ and $f(z,r,g_{p}(z,r))=0$.
An obvious covering argument on $S^{5}$ then yields a smooth function $g: S^5 \times [0,\epsilon) \to \C$ satisfying
\begin{align}
g(z,0)=0, \quad f(z,r,g(z,r))=0.
\end{align}
We now set $R=\epsilon^{-1}$ and define $\nu:\C^{3}\setminus \overline{B}_R\to\C$ by
\begin{align}
\label{gz}\nu(z) = \frac{1}{|z|}g\left(\frac{z}{|z|},\frac{1}{|z|}\right).
\end{align}
The fact that $\nu$ satisfies $\sum_{i=1}^3(z_{i}+\nu(z)\overline{z}_{i}^{2})^{3}=1$
for each $z\in TY_0 \setminus
\overline{B}_{R}$ is straightforward to verify. 

Now we show that $\nu(z) = O(|z|^{-4})$. To see this, observe that $\nu(z)|z|^4P(z)=1$ on $\C^{3}\setminus\overline{B}_{R}$, where
\begin{equation}\label{PZ}
P(z)=3\left(\sum_{i=1}^3\frac{|z_{i}|^{4}}{|z|^{4}}\right)+3(\nu(z)|z|)\left(\sum_{i=1}^3\frac{|z_{i}|^{2}\overline{z}^{3}_{i}}{|z|^{5}}\right)+(\nu(z)|z|)^{2}
\left(\sum_{i=1}^3\frac{\overline{z}_{i}^{6}}{|z|^{6}}\right).
\end{equation}
It then suffices to note that $\nu(z)|z| = g(\frac{z}{|z|},\frac{1}{|z|}) \to 0$ uniformly as   $|z| \to \infty$.

In order to complete the proof of the lemma, we show that $\partial^k \nu(z)= O( |z|^{-4-k})$ for all $k \geq 1$. We only do this for $k = 1$. To this end, note that $\nu(z)=1/(|z|^4P(z))$. Thus, using the expression of $P(z)$ in \eqref{PZ}, it suffices to show that $\partial(\nu(z)|z|) = O(|z|^{-1})$.  But this follows from \eqref{gz} and from the fact that $g$ is a smooth function on $S^5\times [0,\epsilon)$. The lemma is proved.
\end{proof}

Now we define a diffeomorphism $\Phi$ from $ TY_0\setminus \overline{B}_{R}$ onto its image contained in $ TY_1$ as follows:
\begin{equation}\label{eq-Phi}\Phi(z_1,z_2,z_3):=(z_1+\nu(z)\overline{z}_1^2,z_2+\nu(z)\overline{z}_2^2,z_3+\nu(z)\overline{z}_3^2).\end{equation}
With this, we can trivialize the family $TY_\sigma$ diffeomorphically at infinity.

\begin{definition}\label{def:scaling_maps}
Let $\Phi$ be defined by \eqref{eq-Phi}, $m'_\sigma:TY_0\to TY_0$ by sending $z$ to $\sigma^{-1/3}z$ and $m_\sigma:TY_\sigma\to TY_{1}$ by sending $z$ to $\sigma^{-1/3}z$. Then we define a family of maps $\Phi_\sigma: TY_0\setminus B_{|\sigma|^{1/3}R}\to TY_\sigma$ by
\begin{equation}\label{eq-Phis}
\Phi_\sigma:=m^{-1}_\sigma\circ\Phi\circ m'_\sigma.
\end{equation} 
\end{definition}
  
\begin{lemma}\label{Phis}Let $z\in TY_0$ with $|z|\geq |\sigma|^{1/3}R$ and also set 
 $\nu_\sigma(z):=\sigma^{-1/3}\nu(\sigma^{-1/3}z)$.  Then
 \begin{equation}
 \label{eq-Phi_sigma} 
   \Phi_\sigma(z_1,z_2,z_3)=(z_1+\nu_{\sigma}(z)\overline{z}_1^2,z_2+\nu_{\sigma}(z)\overline{z}_2^2,z_3+\nu_{\sigma}(z)\overline{z}_3^2)
   \end{equation}  
with
$\partial^k \nu_{\sigma}(z) = O(|\sigma||z|^{-4-k})$ as $|\sigma|^{-1/3} |z| \to \infty$ for all $k \geq 0$.
\end{lemma}

\begin{proof}
      By \eqref{eq-Phis}, for $z\in TY_0$,
      \begin{align}
    \Phi_\sigma(z_1,z_2, z_3) = (z_1+\sigma^{-\frac{1}{3}}\nu(\sigma^{-\frac{1}{3}}z) \overline{z}_1^2, z_2+\sigma^{-\frac{1}{3}}\nu(\sigma^{-\frac{1}{3}}z) \overline{z}_2^2, z_3+\sigma^{-\frac{1}{3}}\nu(\sigma^{-\frac{1}{3}}z) \overline{z}_3^2).
      \end{align}
       Recall that $\partial^k \nu(z) = O( |z|^{-4-k})$ as $|z|\to\infty$ for all $k \geq 0$.
       Thus,
       \begin{align}
           \partial^k \nu_\sigma(z) = \sigma^{-\frac{1+k}{3}}(\partial^k \nu)(\sigma^{-\frac{1}{3}}z) = O(|\sigma|^{-\frac{1+k}{3}}(|\sigma|^{-\frac{1}{3}}|z|)^{-4-k}) = O(|\sigma| |z|^{-4-k})
       \end{align}
       as $|\sigma|^{-1/3}|z| \to \infty$.
\end{proof}
  
\begin{remark}
By Lemma \ref{Phis}, $\Phi_\sigma$ is well-defined if $|z|\geq |\sigma|^{1/3}R$ for some sufficiently large but fixed $R \gg 1$. However, the gluing will later be carried out in a region where $|z|$ is much larger than this. More precisely, the gluing region is $|z| = |\sigma|^{1/3} e^{\zeta \,\cdot\, |{\log |\sigma|}|^\alpha}$ for some fixed $\alpha < \frac{1}{2}$ and for $\zeta \in [1,2]$. 
\end{remark}

Lastly, we use the above identifications to trivialize the family $\{\mathcal{X}_\sigma\}_{\sigma \in \Delta}$ diffeomorphically away from smaller and smaller neighborhoods of the singularity $(0,0,0) \in \mathcal{X}_0 \subset \mathbb{C}^3$.

\begin{lemma}\label{outdiff}
There exists a $C^5$ diffeomorphism
$G_\sigma$ from $\mathcal X_0\setminus \{|z| \leq 3|\sigma|^{1/3}R\}$ onto its image, which is contained in $\mathcal X_\sigma\setminus \{|z|\leq 2|\sigma|^{\frac{1}{3}}R\}$ and contains $\mathcal{X}_\sigma \setminus \{|z| \leq 4|\sigma|^{1/3}R\}$, such that the following hold:
\begin{itemize}
\item[$(1)$] $G_\sigma$ is smooth away from the hyperplane $\{Z_0=0\}$ and is the identity on this hyperplane.
\item[$(2)$] Fix $R_0 \ll 1$ so that $\Psi$ is defined on $|z|\leq 3R_0$. Then $G_\sigma = \Psi_\sigma^{-1}\circ\Phi_\sigma\circ\Psi_0$  for $|z|\leq R_0$.
\item[$(3)$] If we denote the complex structure of $\mathcal X_\sigma$ by $J_\sigma$, then  for any $r_0>0$ and any smooth metric $\omega_0$ on $\mathcal X_0\setminus \{|z| < r_0\}$ we have that
\begin{align}|\nabla^k_{\omega_0}
 (G^*_\sigma
J_\sigma-J_0)|_{\omega_0} =O_{r_0,\omega_0}(|\sigma|)\;\,\text{on}\;\,\mathcal X_0\setminus \{|z| < r_0\}
\end{align}
for all $\sigma$ sufficiently small depending on $r_0$ and for all $0 \leq k \leq 4$.
\end{itemize}
\end{lemma}  

\begin{proof}
By construction, $\mathcal{X} = \{(Z,\sigma) \in \mathbb{CP}^3 \times \Delta: p(Z) = \sigma Z_0^6\}$ for a certain homogeneous sextic $p$. By \eqref{eq:singularlocus}, all fibers $\mathcal{X}_\sigma = \{Z \in \mathbb{CP}^3: (Z,\sigma) \in \mathcal{X}\}$ intersect the plane $\{Z_0 = 0\}$ transversely in the smooth curve $\{Z_0=0, Z_1^6 + Z_2^6 + Z_3^6 = 0\}$. Write $\mathbb{A}^3$ to denote the affine chart $\{Z_0 \neq 0\}$ and define
\begin{equation}
    q: \mathbb{A}^3 \to \mathbb{C}, \;\, q(Z) := p(Z)/Z_0^6.
\end{equation}
Let $\nabla q$ denote the type $(1,0)$ gradient of $q$ with respect to the standard Euclidean Kähler metric on $\mathbb{A}^3$ and let $|\nabla q|$ denote the Euclidean length of $\nabla q$. Then the smooth $(1,0)$-vector field
\begin{equation}
    V := \left(\frac{\nabla q}{|\nabla q|^2},1 \right) \in T^{1,0}(\mathbb{A}^3 \times \Delta)
\end{equation}
is defined at all points of the submanifold $\mathcal{X}' := (\mathcal{X}\setminus \mathcal{X}_0^{sing}) \cap (\mathbb{A}^3 \times \Delta)$ and is  tangent to $\mathcal{X}'$. In fact, the time-$\sigma$ flow of $V$ maps the fiber $\mathcal{X}'_\tau$ into the fiber $\mathcal{X}'_{\tau+\sigma}$ for every $\tau$.

\begin{claim}\label{claim:extend}
$V$ extends from $\mathcal{X}'$ to $\mathcal{X} \setminus \mathcal{X}_0^{sing}$ as a $C^5$ vector field vanishing along the infinity divisor.
\end{claim}

\begin{proof}[Proof of Claim \ref{claim:extend}] It suffices to check this in the chart $\{Z_1 \neq 0\}$. Define affine coordinates $w_1 = Z_0/Z_1, w_2 = Z_2/Z_1, w_3 = Z_3/Z_1$ on this chart. Then $d(1/w_1), d(w_2/w_1), d(w_3/w_1)$ are an orthonormal $(1,0)$-coframe with respect to the standard Euclidean metric on $\mathbb{A}^3 = \{w_1 \neq 0\}$. The following $(0,1)$-frame (written as vectors with respect to $\partial_{\overline{w}_1},\partial_{\overline{w}_2},\partial_{\overline{w}_3})$ is metrically dual to this coframe:
\begin{equation}
V_1 := (-w_1^2,-w_1w_2,-w_1w_3),\;\,V_2 := (0,w_1,0),\;\,V_3 := (0,0,w_1).
\end{equation}
With $\psi(t) := t^6 + t^3$ we then have that $q = \psi(1/w_1) + \psi(w_2/w_1) + \psi(w_3/w_1)$ and so
\begin{equation}
\overline{\nabla q} = \psi'\left(\frac{1}{w_1}\right)V_1 + \psi'\left(\frac{w_2}{w_1}\right)V_2 + \psi'\left(\frac{w_3}{w_1}\right)V_3.
\end{equation}
It follows that $|w_1|^{10}|\nabla q|^2$ is smooth and $\geq 36$ at $\{w_1 = 0\}$. Using this fact and the transversality of $\mathcal{X}$ to the hyperplane $\{Z_0 = 0\}\times \Delta$, we get that the $T^{1,0}\mathbb{CP}^3$ component of $V$ vanishes to order $6$ at the infinity divisor in $\mathcal{X}$. Thus, $V$ extends as a $C^5$ vector field.
\end{proof}

We may assume that $C^{-1}\leq |V|_{\omega_0} \leq C$ for some positive constant $C$ on the set $|z|\geq \frac{R_0}{2}$. Consider the flow $H_t$ induced by $V$ on $\mathcal X'\setminus \{|z| < \frac{R_0}{2}\}$. For $\sigma$ sufficiently small, $H_\sigma$ induces a diffeomorphism from $\mathcal X'_0\setminus \{|z|< \frac{3R_0}{4}\}$ onto its image in $\mathcal X'_\sigma$. Here we use the fact
that $|V|_{\omega_0}$ is uniformly bounded on $\mathcal X'\setminus \{|z|< \frac{R_0}{2}\}$, so the incompleteness of $\mathcal X'$ does not affect that $H_\sigma$ is a diffeomorphism  when $\sigma$ is sufficiently small and $|z|\geq\frac{3R_0}{4}$. It follows that on the strip $S :=\{ {R_0 \leq |z| \leq 2R_0}\}\cap \mathcal X_0$ we have two diffeomorphisms
$\Psi_\sigma^{-1} \circ \Phi_\sigma \circ \Psi_0$ and $H_\sigma$ onto some region of $\mathcal X_\sigma$. Use the map $\Psi_\sigma^{-1}\circ\Phi_\sigma\circ\Psi_0$ to
identify $S$ with its image inside $\mathcal X_\sigma$. Under this identification, $H_\sigma$ can be seen as a family of embeddings
\begin{align}
\tilde H_\sigma: S'\to S
\end{align}
from $S' := \{1.1R_0 \leq |z| \leq 1.9R_0\} \cap \mathcal{X}$ into $S$. Since 
$\tilde H_\sigma\to{\rm Id}_{S'}$
as $\sigma\to 0$,
$\tilde H_\sigma$ is smoothly isotopic to the identity map on $S'$ when $\sigma$ is sufficiently small. In particular, $\tilde H_\sigma$ is given by the flow of a time-dependent vector field $\tilde{V}_\sigma$ defined on
$\tilde{H}_\sigma(S')$. Let $\tau : S \to [0, 1]$ be a smooth cut-off function given by
$\tau(z) = \chi(|z|),$
where $\chi$ is a smooth increasing function equal to zero for $|z|\leq 1.3 R_0$ and equal to 1 for $|z| \geq 1.7R_0.$
Let $\tilde{I}_\sigma$ be the diffeomorphisms generated by the vector field $\tau \tilde{V}_\sigma$ for
$\sigma$ sufficiently small. By construction, $\tilde{I}_\sigma$ is equal to the identity for $|z| \leq 1.2R_0$ and equal to $ \tilde H_\sigma$ for $|z| \geq 1.8R_0.$
Thus, the diffeomorphism
\begin{equation}
G_\sigma:=\begin{cases}
\Psi_\sigma^{-1} \circ \Phi_\sigma \circ \Psi_0 \circ \tilde{I}_\sigma & \text{for}\;\, |z|\leq 1.8R_0,\\
H_\sigma & \text{for}\;\, |z| \geq 1.8R_0,\end{cases}
\end{equation}
satisfies the desired properties.
\end{proof}
        
\subsection{Singular Kähler-Einstein metrics with hyperbolic cusps}
         
Let $\mathcal X = \{\mathcal{X}_\sigma\}_{\sigma\in\Delta}$ be the family of sextic surfaces in $\mathbb{CP}^3$ discussed above. Recall that we have removed a family of hyperplane sections $\mathcal{D} = \{\mathcal{D}_\sigma\}_{\sigma\in\Delta}$ with $[2\mathcal{D}_\sigma] = K_{\mathcal X_\sigma}$ for all $\sigma\in\Delta$, thus defining a family of affine surfaces $\cY_\sigma=\mathcal X_\sigma\setminus \mathcal{D}_\sigma$, and that we have also fixed a family of holomorphic volume forms $\Omega_{\sigma}$ (unique up to scaling) on $\mathcal Y_\sigma$ that vanish to order $2$ along the infinity divisors $\mathcal{D}_\sigma$.

It is a well-known fact that the regular part of $\mathcal{X}_0$ admits a unique complete Kähler-Einstein metric $\omega_{KE,0}$ with Ricci curvature $-1$. This fact can nowadays be viewed as a very special case of a general theory of complete Kähler-Einstein metrics on log-canonical models; see \cite[Thm A]{BG} and \cite[Section 3]{Song}. However, in the $2$-dimensional case---particularly for surfaces of general type with elliptic singularities such as our $\mathcal{X}_0$---the required existence result actually goes back to work of R. Kobayashi \cite[Thm 1]{rkoba} and of Cheng-Tian-Yau \cite{CYcusp, TY0}. We now briefly review the construction of $\omega_{KE,0}$ in our example.

In our example, recall that $K_{\mathcal X_\sigma}\cong\mathcal O(2)|_{\mathcal X_\sigma}$. Fix the embedding $\mathbb {CP}^3\to\mathbb{CP}^9$ induced by $\mathcal O(2)$ as
\begin{align}
[Z_0,Z_1,Z_2,Z_3]\longrightarrow [Z_0^2,Z_0Z_1,Z_0Z_2,Z_0Z_3,Z_1^2,Z_1Z_2,Z_1Z_3,Z_2^2,Z_2Z_3,Z_3^2].
\end{align}
Restricting the Fubini-Study metric on $\mathbb{CP}^9$  to the image of $\mathbb{CP}^3$, we have
\begin{align}\label{eq:FS}
\begin{split}
\omega_{FS}&:=\ddbar\log(|Z_0^2|^2+|Z_0Z_1|^2+|Z_0Z_2|^2+|Z_0Z_3|^2+|Z_1^2|^2+|Z_1Z_2|^2+|Z_1Z_3|^2\\
&+|Z_2^2|^2+|Z_2Z_3|^2+|Z_3^2|^2).
\end{split}
\end{align}
On  the affine piece  $Z_0\neq 0$, we set 
\begin{equation}\label{FS}
\psi_{FS}:=\log(1+|z_1|^2+|z_2|^2+|z_3|^2+|z_1|^4+|z_1z_2|^2+|z_1z_3|^2+|z_2|^4+|z_2z_3|^2+|z_3|^4).
\end{equation}
For better emphasis, we also introduce the notation
\begin{align}\label{FS_sigma}
\omega_{FS,\sigma} := \omega_{FS}|_{\mathcal{X}_\sigma},\;\,\psi_{FS,\sigma} := \psi_{FS}|_{\mathcal{Y}_\sigma}\;\,\text{for all}\;\,\sigma \in \Delta.
\end{align}
Now let us fix a volume form $V$ on $\mathcal X_0$, which is induced by the Fubini-Study metric of the line bundle $\mathcal O(1)$ on $\mathbb{CP}^9$, such that
\begin{equation}\label{bbbbbb}
\omega_{FS,0} =\ddbar\log V.
\end{equation}
  Then the complex Monge-Amp\`ere equation of interest on $\mathcal X_0$ is given by
\begin{equation}\label{KEeqn}
	(\omega_{FS,0}+\ddbar\psi_0)^2 = e
	^{\psi_0} V.
 \end{equation}
We have the following existence result from \cite[Thm 1]{rkoba}, \cite[Thm C]{BG}, \cite[Lemma 3.6]{Song}, together with an asymptotic estimate of the solution $\psi_0$ from \cite[Prop D]{DiNeGueGue}, \cite[Prop 3.1]{SSW2}, \cite[Thm 1.1]{DFS}.
 
\begin{theorem}
\label{lem:exloglogsoln}
The equation \eqref{KEeqn} admits a solution $\psi_0 \in C^\infty(\mathcal{X}_0^{reg})$ satisfying
\begin{align}\label{eq:rough_est_ke_pot}-(3+\epsilon)\log (-{\log |\sigma_E|_{h_E}})-C_\epsilon\leq\pi^*(\Psi_0^{-1})^*\psi_0\leq C\end{align}
for every $\varepsilon>0$,  where $\pi$ is the resolution \eqref{resolution},  $\sigma_E$ is a defining section of $E$ inside the total space $\hat{L}$ and $h_E$ is any Hermitian metric on $L\to E$.
\end{theorem}

It will be convenient to rewrite the Monge-Amp\`ere equation \eqref{KEeqn} in the following manner.

\begin{lemma}\label{NormalizeKE}
Define $\psi_{KE,0}:=\psi_0+\psi_{FS,0}$, where $\psi_0$ is the solution provided by Theorem \ref{lem:exloglogsoln}. Then, after adding a constant to $\psi_{KE,0}$, the K\"ahler-Einstein metric  $\omega_{KE,0}:=\ddbar\psi_{KE,0}$ satisfies
\begin{equation}\label{eq:gfeorp}
	\log\frac{\omega_{KE,0}^
		2}{\Omega_0\wedge\overline{\Omega}_0}=\psi_{KE,0}.
\end{equation}
\end{lemma}

\begin{proof}
By construction, \eqref{KEeqn}, we have that
\begin{equation}\label{eq:fgeorp}
    \log\frac{\omega_{KE,0}^2}{V} = \psi_{KE,0} - \psi_{FS,0}.
\end{equation}
Recall that by definition, $V$ is the unique real volume form on the regular part of $\mathcal{X}_0$ that, viewed as a Hermitian metric on $K_{\mathcal{X}_0}$, maps to the Fubini-Study metric $h_{FS}$ under the adjunction isomorphisms $K_{\mathcal{X}_0} \cong \mathcal{O}_{\mathbb{CP}^3}(2)|_{\mathcal{X}_0} \cong \mathcal{O}_{\mathbb{CP}^9}(1)|_{\mathcal{X}_0}$. On the other hand, also by construction, $\Omega_0$ vanishes to order $2$ along the hyperplane section $\mathcal{D}_0 = \mathcal{X}_0 \setminus \mathcal{Y}_0$ and vanishes nowhere else. Thus, under the above adjunctions, $\Omega_0$ maps to some nonzero multiple of the section $S$ of $\mathcal{O}_{\mathbb{CP}^9}(1)|_{\mathcal{X}_0}$ that cuts out the hyperplane $\mathcal{D}_0$. The square of the Hermitian norm of $\Omega_0$ with respect to $V$ is $(\Omega_0 \wedge \overline{\Omega}_0)/V$, and $|S|_{h_{FS}}^2 = -\psi_{FS,0}$. Thus, we get \eqref{eq:gfeorp} from \eqref{eq:fgeorp} by adding $\psi_{FS,0} + const$ to both sides.
\end{proof}

We now turn to a more precise asymptotic description of $\psi_{KE,0}$ near the singularity $p \in \mathcal{X}_0$. After introducing the relevant radial model potential $\psi_{cusp}$ and proving some technical lemmas, we will give the precise asymptotics of $\psi_{KE,0}$ in Proposition \ref{KE-CUSP}, which is the final result of this section.

Working on the line bundle $L$ whose total space $\hat{L}$ resolves the singularity, let $h$ denote the Hermitian metric on $L$ (unique up to scale) whose curvature form is minus the flat Kähler form $\omega_E$ representing $2\pi c_1(L') \in H^{1,1}(E,\mathbb{R})$. Here $L'$ denotes the dual of $L$. We also denote the positively curved Hermitian metric dual to $h$ by $h'$. Abusing notation, we view $h$ as a function on the total space of $L$ via
\begin{equation}
    h: \hat{L} \to \mathbb{R}, \;\, \xi \mapsto |\xi|_h^2.
\end{equation}
We consider radial Kähler metrics $\omega = i\partial \overline{\partial} \psi(t)$ on $\{0<h\leq \delta\}$,  where $\psi: (-\infty, \log \delta] \rightarrow \mathbb{R}$ and 
\begin{align}
 t := \log h.
\end{align}
Since the discussion of these model Kähler metrics (here as well as in all subsequent sections) works more or less the same way in all dimensions $n$ with $(L',h') \to (E,\omega_E)$ a polarized compact Calabi-Yau $(n-1)$-fold,  we prefer to \emph{not} specialize our computations of these metrics to the case $n = 2$. 

Continuing to work in the general $n$-dimensional setting, we fix a nowhere vanishing holomorphic $(n-1)$-form $\Omega_{E}$ on $E$ with
\begin{equation}
\omega_E^{n-1}=i^{(n-1)^2} \Omega_E\wedge\overline{\Omega}_E. 
\end{equation}
Note that $\Omega_E$ is unique up to multiplication by a unit complex number. Using this, we now construct a holomorphic volume form $\Omega_{\mathcal C}$ on the total space of the $\mathbb{C}^*$-bundle associated with $L$ or $L'$, i.e., on the complement of the zero section in the total space of either one of these two line bundles. This form $\Omega_{\mathcal{C}}$ will be unique up to sign, and will have first-order poles along both zero sections such that its residues at these poles are equal to $\pm \Omega_E$. Precisely, let $p$ denote the bundle projection from the total space of the $\mathbb{C}^*$-bundle onto $E$ and let $\pm Z$ denote the $(1,0)$-vector field on the total space that generates the natural $\C^*$-action coming from either $L$ or $L'$. Then $\Omega_{\mathcal C}$ is determined by the equation
 \begin{equation}
 \pm Z \lrcorner\ \Omega_{\mathcal C}=p^*\Omega_E.
 \end{equation}

A radial $(1,1)$-form $\omega = i\partial\overline\partial \psi(t)$ as above is positive definite if and only if 
\begin{align}\label{eq-psi-1}
\psi'(t), \psi''(t)>0.
\end{align}	
Then $\omega$ is a Kähler-Einstein metric if \eqref{eq-psi-1} holds and
\begin{align}\label{eq-psi-2}
(\psi')^{n-1} \psi'' =e^{\psi +a}
\end{align}
for some $a \in \mathbb{R}$. 
If $\psi$ satisfies 
\eqref{eq-psi-2}, then given any $k \in \mathbb{R}^+$, the function $\Psi(t) := \psi(kt)$ satisfies
\begin{align}
(\Psi')^{n-1} \Psi'' =e^{\Psi +a +(n+1)\log k},
\end{align}
which is the same as equation \eqref{eq-psi-2} except having a different $a$. The solutions related to the cuspidal model Kähler-Einstein metrics in our previous paper \cite{FHJ} are
\begin{align}\label{eq-psi-cusp}
\psi_{cusp}(t) = -(n+1) \log (-t),\;\,
e^a = (n+1)^n.
\end{align}
Given an arbitrary $a$, the model solution is defined as
\begin{align}\label{eq-psi-cusp-2}
\psi_{cusp}(t) := -(n+1) \log (-t )+n\log (n+1) -a, \;\, \omega_{cusp} := i\partial\overline{\partial}\psi_{cusp}.
\end{align}

\begin{lemma}\label{Normpsicusp}
Adding a constant to $\psi_{cusp}$, i.e., choosing the constant $a$ in \eqref{eq-psi-cusp-2} correctly, one has
\begin{equation}\label{normalizingpotential}
e^{-\psi_{cusp}}\omega_{cusp}^n=i^{n^2}\Omega_{\mathcal C}\wedge\overline{\Omega}_{\mathcal C}.
\end{equation}
\end{lemma}

\begin{proof}
It follows from the symmetry of $\omega_{cusp},\Omega_{\cC}$ and from the K\"ahler-Einstein condition that
\begin{align}g(t):=\frac{e^{-\psi_{cusp}}\omega_{cusp}^n}{i^{n^2}\Omega_{\cC}\wedge\overline{\Omega}_{\cC}}\end{align} is a pluriharmonic function depending only on $t$.  Now
\begin{align}\label{eq:radial_ph}
\ddbar g(t)=g'(t)\ddbar t+ig''(t)\partial t\wedge\overline{\partial} t=0.
\end{align}
So $g(t)$ must be a constant and \eqref{normalizingpotential} can be achieved by adding a constant to $\psi_{cusp}$.
\end{proof}

To compare $\psi_{KE,0}$ to $\psi_{cusp}$ in our main example, it is helpful to rescale the holomorphic volume forms $\Omega_\sigma$ by a constant so that $\Omega_0$ becomes asymptotic to $\Psi_0^*\Omega_{\mathcal{C}}$ at the singularity $p \in \mathcal{X}_0$.

\begin{lemma}\label{OmegaOC}
For fixed $\Omega_0$ and $\Omega_\cC$, there is a nonzero complex constant $C$ such that 
\begin{equation}\label{Omega0-OmegaC}
    (\Psi_0^{-1})^*(\Omega_0\wedge\overline{\Omega}_0)= |C|^2 e^{-v} {\Omega_{\mathcal C}}\wedge\overline{\Omega}_{\mathcal C},
\end{equation}
where $v$ is a pluriharmonic function satisfying for all $\epsilon>0$ and $j \geq 1$ and for $t \to -\infty$ that
\begin{equation}\label{ggg}|v|=O(e^{\frac{t}{2}}),\;\,|\nabla_{\omega_{cusp}}^j v|_{\omega_{cusp}}=O_{\varepsilon,j}(e^{\frac{(1-\epsilon)t}{2}}).\end{equation}
\end{lemma}

\begin{proof}
The function $H:=[(\Psi_0^{-1})^*\Omega_0]/\Omega_\cC$ is a nowhere vanishing holomorphic function on $U \setminus \{0\}$ for some small open neighborhood $U$ of the singularity $0 \in TY_0$. Since this is a normal isolated surface singularity in $\mathbb{C}^3$, the Hartogs principle \cite[Thm 1.1]{Ruppenthal} says that $H$ extends to a holomorphic function on $U$. We must have that $H(0) \neq 0$ because otherwise $\lim_{z \to 0} 1/H(z) = \infty$, contradicting the Hartogs principle for $1/H$. Let $C := H(0)$. Then \eqref{Omega0-OmegaC} holds with $v(z) := \log |H(0)/H(z)|^2$, which is obviously pluriharmonic and obviously satisfies $|v(z)| = O(|z|) = O(e^{t/2})$ as $|z| \to 0$ resp. $t \to -\infty$.

Now we further estimate $|\nabla_{\omega_{cusp}}^j v|_{\omega_{cusp}}$ for all $j \geq 1$. Since $v$ is in particular harmonic with respect to $\omega_{cusp}$ and since $\omega_{cusp}$ is a hyperbolic metric for $n = 2$, these derivatives can be estimated using Schauder theory on the universal covers of $\omega_{cusp}$-geodesic balls of some fixed radius $r_0 \in (0,1]$. Coordinates on these local universal covers that make $\omega_{cusp}$ uniformly smoothly bounded can be found e.g. in \cite[proof of Lemma 3.5]{FHJ}. The $\varepsilon$-loss in \eqref{ggg} is due to the fact that the function $t$ (which coincides with $-1/x$ in the notation of \cite{FHJ}) varies by a factor of $1 + O(r_0)$ over any $\omega_{cusp}$-ball of radius $r_0 \ll 1$.
\end{proof}

By a common rescaling of the forms $\Omega_\sigma$, which was our only freedom in choosing these forms, we can now arrange that $C = 1$ in \eqref{Omega0-OmegaC}, i.e., that $\Omega_0$ is asymptotic to $\Psi_0^*\Omega_{\mathcal{C}}$ near the singularity $p \in \mathcal{X}_0$. According to Lemma \ref{NormalizeKE}, this leads to an additive normalization of $\psi_{KE,0}$ such that \eqref{eq:gfeorp} holds, and this equation then matches the equation \eqref{normalizingpotential} satisfied by $\psi_{cusp}$ to leading order at the singularity.

We require one additional technical lemma before stating our final result.

\begin{lemma}\label{PLHest}
Let $U$ be a neighborhood of the singularity $p=(0,0,0)$ in $TY_0$. Let $\varphi$ be a function on $U\setminus \{p\}$  satisfying $\ddbar\varphi=0$ and $|\varphi| = o(-{\log |z|}) = o(-t)$ as $|z|\to 0$ resp. $t \to -\infty$. Then $\varphi$ extends continuously as a PSH function on $U$. Moreover, $\varphi$ is the real part of a holomorphic function near the origin of $\mathbb C^3$, so, after subtracting a constant, $|\varphi| = O(|z|) = O(e^{t/2})$ as $|z|\to 0$ resp. $t \to -\infty$.
\end{lemma}

\begin{proof}
Consider the resolution of singularities $\pi:\hat L \to TY_0$ from \eqref{resolution}. The exceptional divisor is an elliptic curve $E$ and $\hat L$ is the total space of a negative holomorphic line bundle $L$ over $E$. For any point $q\in E$, there is a local holomorphic trivialization of the line bundle $\hat L|_V\cong V\times \mathbb C$ for some neighborhood $V \ni q$. Let $w_1$ be the coordinate of $V$ and $w_2$ be the coordinate of $\mathbb C$. Then by assumption, $\varphi$ restricted to any fiber of $L$ is a smooth harmonic function on the punctured disk $\Delta^*$ with a sub-logarithmic pole at $0$. Hence for each fixed $w_1$, $\varphi$ can be extended continuously across $0$. So for fixed $w_1$, $\varphi$ is the real part of a holomorphic function on $\Delta$. By the mean value theorem,
\begin{equation}
\varphi(w_1,0)=\dashint_{\gamma}\varphi(w_1,w_2)\,dw_2,   
\end{equation}
where $\gamma$ is a fixed circle centered at $0$. Taking $\ddbar$ with respect to the variable $w_1$, we have that $\varphi(w_1,0)$ is a harmonic function of $w_1$. Notice that $E$ is compact, so $\varphi(w_1,0)$ is a constant. Hence $\varphi$ can be extended continuously across $\mathcal{X}_0^{sing}$ as a PSH function on $\mathcal{X}_0$. 

Now we show that $\varphi$ is the real part of a holomorphic function. Note that because the line bundle $\hat L$ is homotopy equivalent to its zero section $E$, the de Rham cohomology $H^1(\hat L,\mathbb R)$ is naturally isomorphic to $H^1(E,\mathbb R)$. Since $\ddbar\pi^*\varphi=0$, $d^c\pi^*\varphi$ is a closed real $1$-form. So we have 
\begin{align}d^c\pi^*\varphi=dg+a\alpha+b\beta,\end{align} where $a,b$ are constants, $g$ is a function on an open neighborhood of the zero section of $\hat L$ and $\alpha,\beta$ are closed $1$-forms on $E$ generating $H^1(E,\mathbb{R})$. Let $\gamma_\alpha,\gamma_\beta$ be loops in $E$ whose homology classes are Poincar\'e dual to $\alpha,\beta$.  Note that $\varphi$ is constant on $E$ and $dg$ is exact, so the integrals of the $1$-form $d^c\pi^*\varphi-dg$ along the loops $\gamma_\alpha,\gamma_\beta$ are zero. So $a,b$ are zero. So $f:=\pi^*\varphi-ig$ is a holomorphic function with real part $\pi^*\varphi$, necessarily constant along the zero section $E$. After pushing down to the singularity, we can locally extend $f$ to a holomorphic function on an open set of $\mathbb C^3$. By subtracting a constant, we may assume that $f(p) = 0$. By Taylor expansion, we deduce that $f=O(|z|)$.
\end{proof}

We are now able to prove the following precise asymptotic comparison of $\psi_{KE,0}$ and $\psi_{cusp}$. Modulo the above lemmas this is the statement of \cite[Main Theorem]{FHJ}, whose proof relies on an earlier partial result from \cite[Thm 1.4]{DFS}. The results of \cite{DFS,FHJ} can be generalized to all dimensions $n \geq 2$ assuming that the polarized Calabi-Yau $(n-1)$-fold $(L',h')\to (E,\omega_E)$ in the above discussion of $\psi_{cusp}$ is flat.

\begin{proposition}\label{KE-CUSP}
Let $\Psi_\sigma$ be the local holomorphic isomorphism defined in equation \eqref{Psi}. Then
\begin{align}\label{eq:KE-CUSP}
(\Psi^{-1}_0)^*\psi_{KE,0}-\psi_{cusp}=u+v,
\end{align}
where for some constants $\mathfrak{s} \in \mathbb{R}$, $\delta_0 \in \mathbb{R}^+$, for all $j \geq 0$ and for $t \to -\infty$,
\begin{equation}\label{eq:KE-CUSP-2}
    \left|\nabla_{\omega_{cusp}}^j\left(u+3 \log \left(1-\frac{\mathfrak{s}}{t}\right)\right)\right|_{\omega_{cusp}}= O_j(e^{-\delta_0\sqrt {-t}}),
\end{equation}
and where $v$ is a pluriharmonic function satisfying for all $\epsilon>0$ and $j \geq 1$ and for $t \to -\infty$ that
\begin{equation}\label{bbbb}
|v|=O(e^{\frac{t}{2}}),\;\,|\nabla_{\omega_{cusp}}^j v|_{\omega_{cusp}}=O_{\varepsilon,j}(e^{\frac{(1-\epsilon)t}{2}}).
\end{equation}
\end{proposition}

\begin{remark}\label{scalelambda}
Composing both sides of \eqref{eq:KE-CUSP} with the automorphism $scale_{\lambda}(z) = \lambda z$ of $\mathbb{C}^3$ yields a new decomposition
$(\Psi_0^{-1} \circ scale_{\lambda})^*\psi_{KE,0} - \psi_{cusp} = u_\lambda + v_\lambda$,
where $u_\lambda, v_\lambda$ satisfy the same properties as before except that the constant $\mathfrak{s}$ in \eqref{eq:KE-CUSP-2} gets replaced by $\mathfrak{s}-\log\lambda $. Thus, by choosing $\lambda = e^{\mathfrak{s}}$ we obtain that $u_\lambda$ is purely exponentially decaying and contains no powers of $1/t$ in its expansion as $t \to -\infty$. This was already pointed out in the introduction to our previous paper \cite{FHJ}.

This suggests making the following modification to our setup: For $\lambda = e^{\mathfrak{s}}$, replace $\Psi_\sigma: \mathcal{Y}_\sigma \to TY_\sigma$ by ${scale}_{\lambda^{-1}} \circ \Psi_\sigma: \mathcal{Y}_\sigma \to TY_{\lambda^{-3}\sigma}$ for all $\sigma \in \Delta$. (The domains of these maps are actually just small open neighborhoods of the origin in $\mathbb{C}^3$ intersected with $\mathcal{Y}_\sigma$.) In this way, we can assume without loss of generality that $\mathfrak{s} = 0$ in \eqref{eq:KE-CUSP-2}, i.e., that $u$ is exponentially decaying. Note that, so far, we have never actually used the map $\Psi_\sigma$ with $\sigma \neq 0$ except in Lemma \ref{outdiff}. The statement and proof of that lemma remain unchanged if we also replace $\Phi_\sigma$ by $\Phi_{\lambda^{-3}\sigma}$, which again takes values in $TY_{\lambda^{-3}\sigma}$.

However, $\Psi_\sigma,\Phi_\sigma$ will also be used as gluing maps in Section \ref{sec:glued_approx_KE}. The replacements $\Psi_\sigma \leadsto {scale}_{\lambda^{-1}}\circ \Psi_\sigma$ and $\Phi_\sigma \leadsto \Phi_{\lambda^{-3}\sigma}$ do not worsen any of the estimates in Section \ref{sec:glued_approx_KE} (or later), but being able to assume that $\mathfrak{s} = 0$ in \eqref{eq:KE-CUSP-2} drastically reduces the gluing error. To exploit this improvement without having to introduce even more notation, we will from now on simply assume that $\mathfrak{s} = 0$ and $\lambda = 1$.
\end{remark}

\begin{proof}[Proof of Proposition \ref{KE-CUSP}]
First of all, if we define
\begin{align}
u:=\log\left(\frac{(\Psi_0^{-1})^*\omega^2_{KE,0}}{\omega^2_{cusp}}\right),
\end{align}
then by the Kähler-Einstein equation, we have that
\begin{align}
(\Psi_0^{-1})^*\omega_{KE,0}-\omega_{cusp}=\ddbar u.
\end{align}
By \cite[Main Theorem]{FHJ}, there exist $\mathfrak{s} \in \mathbb{R}$, $\delta_0 \in \mathbb{R}^+$ such that for all $j \geq 0$ and for $t \to -\infty$,
\begin{align}
\left|\nabla_{\omega_{cusp}}^j\left(u +3 \log \left(1-\frac{\mathfrak{s}}{t}\right)\right)\right|_{\omega_{cusp}} = O_j(e^{-\delta_0\sqrt {-t}}).
\end{align}
Here we remark again that the function $-1/t$ in this paper is the same as the function $x$ used in \cite{FHJ}.
So at the K\"ahler potential level, we have that
\begin{equation}
	(\Psi_0^{-1})^*\psi_{KE,0}-\psi_{cusp}=\varphi+u,
\end{equation}
where $\varphi$ is a pluriharmonic function. By the estimate \eqref{eq:rough_est_ke_pot} of $\psi_{KE,0}$ near the singularity and by the definition of $\psi_{cusp}$, we have $|\varphi|=O(\log(-t)) = o(-t)$. Thus, by Lemma \ref{PLHest}, $\varphi =\mathbf{c}+\mathbf{v},$ where $\mathbf{c}$ is a constant and $\mathbf{v}$ is a pluriharmonic function satisfying $|\mathbf{v}| = O (|z|) = O(e^{t/2})$.

It is now easy to see that the constant $\mathbf{c}$ is zero. Indeed,
\begin{equation}
e^{-\psi_{cusp}}\omega_{cusp}^2=\Omega_{\mathcal C}\wedge\overline{\Omega}_{\mathcal C}=e^{v}(\Psi_0^{-1})^*(\Omega_0\wedge\overline{\Omega}_0)=e^{v}(\Psi_0^{-1})^*(e^{-\psi_{KE,0}}\omega^2_{KE,0})
\end{equation}
by \eqref{normalizingpotential}, by \eqref{Omega0-OmegaC} with our chosen normalization $C = 1$, and by \eqref{eq:gfeorp}. Therefore,
\begin{equation}
{\mathbf{c}+u+\mathbf{v}}=(\Psi_0^{-1})^*\psi_{KE,0}-\psi_{cusp} = \log\left(\frac{(\Psi_0^{-1})^*\omega^2_{KE,0}}{\omega^2_{cusp}}\right) + v = u + O(e^{\frac{t}{2}}).
\end{equation}
Since $\mathbf{v}$ goes to zero at the singularity, we get $\mathbf{c}=0$. Moreover, $\mathbf{v} = v$, so \eqref{ggg} implies \eqref{bbbb}.
\end{proof}

\subsection{Tian-Yau metrics}\label{TYConstruction}
         
In this subsection we review the Tian-Yau construction \cite[Thm 4.2]{TY} of a complete Ricci-flat K\"ahler metric on the complement of a smooth anti-canonical divisor in a smooth Fano manifold. Expositions of this construction can be found in \cite[Section 3]{HSVZ} and \cite[Section 3]{HSVZ2}. We will borrow freely from these two references and show how to match their notation to ours.

As in our discussion of cusp metrics above, let $E$ be an $(n-1)$-dimensional compact K\"ahler manifold
with trivial canonical bundle $K_E$ and let $L' \to E$ be an ample line bundle, the dual of a negative line bundle $L \to E$. We fix a nowhere vanishing
holomorphic $(n-1)$-form $\Omega_{E}$ on $E$ with
\begin{equation}\label{CMA}
\int_{E} i^{(n-1)^2}\Omega_E\wedge\overline{\Omega}_E=(2\pi c_1(L'))^{n-1}.
\end{equation}
By Yau’s resolution of the Calabi conjecture \cite{Yau}, there exists a unique Ricci-flat K\"ahler metric $\omega_E$ on $E$ representing the Kähler class $2\pi c_1(L') \in H^{1,1}(E,\mathbb{R})$ and satisfying the equation
\begin{equation}
\omega_E^{n-1}=i^{(n-1)^2} \Omega_E\wedge\overline{\Omega}_E. 
\end{equation}
Up to scaling, there exists a unique Hermitian metric $h'$ on $L'$ whose curvature form is $\omega_E$, and $h'$ is the dual of a negatively curved Hermitian metric $h$ on $L$. We
now fix a choice of $h'$. Then the \emph{Calabi model space} is the subset $\mathcal C$ of the total space of $L'$ consisting
of all elements $\xi$ with $0< |\xi|_{h'}< 1$, endowed with a nowhere vanishing holomorphic volume form
$\Omega_{\mathcal C}$ and a Ricci-flat K\"ahler metric $\omega_{\mathcal C}$ which is incomplete as $|\xi|_{h'}\to  1$ and complete as $|\xi|_{h'}\to 0.$
Again as in our discussion of cusp metrics above, the holomorphic volume form $\Omega_{\mathcal C}$ is uniquely determined by the equation
 \begin{equation}\label{cone}
 Z \lrcorner\ \Omega_{\mathcal C}=p^*\Omega_E,
 \end{equation}
where $p : \mathcal C \to E$ is the bundle projection and $Z$ is the holomorphic vector field generating the natural $\mathbb C^*$-action on the fibers of $p$ (i.e., the one coming from the line bundle structure of $L'$). The metric $\omega_{\mathcal C}$ is given by the Calabi ansatz
\begin{equation}\label{calabiansatz}
\omega_{\mathcal C}:=\ddbar \psi_{\mathcal C}, \;\, \psi_{\mathcal{C}} := \frac{n}{n+1}(-{\log |\xi|_{h'}^2})^{\frac{n+1}{n}},
\end{equation}
and satisfies the Monge-Amp\`ere equation
 \begin{equation}\label{calabiyau}
 \omega_{\mathcal C}^n=i^{n^2} \Omega_{\mathcal C}\wedge\overline\Omega_{\mathcal C},
 \end{equation}
hence is Ricci-flat. Also define the momentum coordinate
\begin{equation}
    \mathbf{z} := (-{\log|\xi|^2_{h'}})^\frac{1}{n}.
\end{equation}
Then the $\omega_{\mathcal C}$-distance to a fixed point in $\mathcal{C}$ is uniformly comparable to $\mathbf{z}^{(n+1)/2}$ for $z \gg 1$.

We now explain the Tian-Yau construction \cite{TY} of complete Ricci-flat K\"ahler metrics asymptotic to a Calabi ansatz at infinity. Let $M$ be a smooth Fano manifold of complex dimension $n$, let $E\in |K_M^{-1}|$ be a smooth divisor, and let $L'$ denote the holomorphic normal bundle to $E$ in $M$. Then $E$ has a trivial canonical bundle and $L'$ is ample, so in particular we can choose a holomorphic volume form $\Omega_E$ on $E$ which satisfies \eqref{CMA}. We fix a defining section $S$ of $E$, so that $S^{-1}$ can be viewed as a holomorphic $n$-form $\Omega_X$ on $X=M\setminus E$ with a simple pole along $E$. After scaling $S$ by a nonzero complex constant, we may assume that $\Omega_E$ is the residue of $\Omega_X$ along $E$. (In practice this means that $\Omega_X$ is asymptotic to $\Omega_{\mathcal{C}}$ with respect to a suitable diffeomorphism $\Phi$ between tubular neighborhoods of $E$ in $M$ and of the zero section in $L'$.) Lastly, we fix a Hermitian metric $h_M$ on $K_M^{-1}$ whose curvature form is strictly positive on $M$ and restricts to the unique Ricci-flat K\"ahler form $\omega_E \in 2\pi c_1(L')$ on $E$. Then
 \begin{equation}\label{hM}
 \omega_X^\circ := \frac{n}{n+1}\ddbar(-{\log |S|_{h_M}^2})^{\frac{n+1}{n}}
 \end{equation}
defines a K\"ahler form on a neighborhood of infinity in $X$. This $\omega_X^\circ$ is then complete towards $E$ and is asymptotic to $\omega_{\mathcal{C}}$, where the Hermitian metric $h'$ used in \eqref{calabiansatz} is the restriction of $h_M$ to $K_M^{-1}|_E = L'$. In particular, $\omega_X^\circ$ is asymptotically Ricci-flat. Moreover, since $X$ is an affine variety, it is reasonably straightforward to extend $\omega_X^\circ$ as a globally defined $i\partial\overline\partial$-exact Kähler form on the whole manifold $X$. Then the following is proved in \cite{TY} by solving a Monge-Amp\`ere equation with reference metric $\omega_X^\circ$. The exponential decay statement follows from \cite[Prop 2.9]{Hein}.

\begin{theorem}\label{t:hein}
There is a complete Ricci-flat K\"ahler metric $\omega_{X}$ on $X$ solving the equation
 \begin{equation}
 \omega_{X}^n=i^{n^2}\Omega_{X}\wedge\overline\Omega_{X}.
 \end{equation}
Moreover, there is a unique choice of the scaling factor of $h'$ resp.~of $h_M$ such that $\omega_{X} = \omega_X^\circ + i\partial\overline\partial \phi$, where for some $\delta_0 >0$ and for all $k\geq 0$, as $z \to \infty$,
 \begin{equation}\label{lalilu}
 |\nabla_{\omega_{\mathcal{C}}}^k (\phi \circ \Phi)|_{\omega_{\mathcal{C}}} = O_k(e^{-\delta_0 \mathbf{z}^{n/2}}).
\end{equation}
Here we have implicitly fixed a smooth identification $\Phi$ of $M$ and of the total space of $L'$ near $E$.
\end{theorem}

We will now apply this construction to our space $X := TY_1 = \{z_1^3+z_2^3+z_3^3=1\}$, a smooth cubic in $\mathbb{C}^3$. We can compactify $X$ to $M$, a smooth cubic in $\mathbb{CP}^3$, by adding an elliptic curve $E$ at infinity. Note that $M$ is an anticanonically embedded del Pezzo surface and $L'=-K_{M}^{-1}|_{E} = \mathcal O_{\mathbb{CP}^3}(1)|_{E}$. The total space of the dual bundle $L$ resolves the singularity of the cubic cone $TY_0$ at the origin, and $\mathcal{C}$ is identified with a neighborhood of infinity in $TY_0$. Then our diffeomorphism $\Phi = \Phi_1$ from \eqref{eq-Phi}, \eqref{eq-Phis} plays the role of the smooth identification $\Phi$ in Theorem \ref{t:hein}, and $\mathbf{z}^2 = t$ for $t > 0$.

The following lemma records some more detailed estimates from \cite[Prop 3.4]{HSVZ} in our setting.

\begin{lemma}\label{p:TY-asym}
For all $\epsilon > 0$ and $k \geq 0$ and for $t \to +\infty$,
\begin{equation}\label{eq:moron}
|\nabla_{g_{\mathcal C}}^k(\Phi_1^*J_{TY_1}-J_{\cC})|_{g_\mathcal C} + |\nabla_{g_{\mathcal{C}}}^k(\Phi_1^*\Omega_{TY_1}-\Omega_{\mathcal C})|_{g_{\mathcal C}} =O_{\varepsilon,k}(e^{-(\frac{1}{2}-\epsilon)t}).
\end{equation}
Moreover, there exists a global potential $\psi_{TY_1}$ of $\omega_{TY_1}$, i.e., $\omega_{TY_1} = i\partial\overline\partial\psi_{TY_1}$ globally on $TY_1$, such that there exists $\delta_0>0$ such that for all $k \geq 0$ and for $t \to +\infty$,
\begin{align}\label{halilu}
|\nabla_{g_\mathcal C}^k(\Phi_1^*\psi_{TY_1}-\psi_{\cC})|_{g_{\mathcal C}} + |\nabla_{g_\mathcal C}^k(\Phi_1^*\omega_{TY_1}-\omega_{\mathcal C})|_{g_{\mathcal C}} =O_k(e^{-\delta_0 \sqrt{t}}).
\end{align} 
\end{lemma}

In our gluing construction, we will be working with a scaled copy of $\omega_{TY_1}$, pulled back from $TY_1$ to $TY_\sigma$ via the maps of Definition \ref{def:scaling_maps}. Let us first recall these maps for convenience:
\begin{equation}\label{eq:review_maps}
\begin{split}
\Phi: TY_0\setminus \overline{B}_R \to TY_1, \;\,\Phi(z) = z + O(|z|^{-2})\;\,\text{as}\;\;|z|\to\infty,\\
m_\sigma': TY_0 \to TY_0, \;\, m_\sigma'(z) = \sigma^{-1/3}z,\\
m_\sigma: TY_\sigma \to TY_1, \;\,m_\sigma(z) = \sigma^{-1/3}z,\\
\Phi_\sigma: TY_0 \setminus \overline{B}_{|\sigma|^{1/3}R}\to TY_\sigma, \;\, \Phi_\sigma = m^{-1}_\sigma \circ \Phi \circ m'_\sigma.
\end{split}
\end{equation}
Next, we introduce the following pullback objects:
\begin{equation}
\begin{split}
\Omega_{TY_\sigma} := m_\sigma^*\Omega_{TY_1}, \;\,\omega_{TY_\sigma} := m_\sigma^*\omega_{TY_1},\;\,\psi_{TY_\sigma} := m_\sigma^*\psi_{TY_1} 
\;\,\text{and}\;\,
g_{\cC,\sigma}:=(m'_{\sigma})^*g_{\cC}.
\end{split}
\end{equation}
Then Lemma \ref{p:TY-asym} implies the following estimates, including an additional rescaling. The strange form of the scaling factor,  $|b|^{1/2}$ with $b < 0$, is due to some conventions in Section \ref{ss:new_neck}.

\begin{lemma}\label{TYerror}
For all $b < 0$ and $\sigma \in \Delta^*$ the following estimates hold with all of the implied constants independent of $b$ and $\sigma$. First, for all $\varepsilon > 0$ and $k \geq 0$ and for $t - \frac{2}{3}\log|\sigma|\to+\infty$,  
\begin{equation}\label{e:n-form-asympt}
\begin{split}
|\nabla_{|b|^{\frac{1}{2}}g_{\mathcal C,\sigma}}^k(\Phi^*_\sigma J_{TY_\sigma}-J_{\cC})|_{|b|^{\frac{1}{2}}g_{\mathcal C,\sigma}}
&=O_{\varepsilon,k}(|b|^{-\frac{k}{4}}e^{-(\frac{1}{2}-\epsilon)(t-\frac{2}{3}\log|\sigma|)}),\\
|\nabla_{|b|^{\frac{1}{2}}g_{\mathcal{C},\sigma}}^k(\Phi_\sigma^*\Omega_{TY_\sigma}-\Omega_{\mathcal C})|_{|b|^{\frac{1}{2}}g_{\mathcal C, \sigma}}&=O_{\varepsilon,k}(|b|^{-\frac{k}{4}}e^{-(\frac{1}{2}-\epsilon)(t-\frac{2}{3}\log|\sigma|)}).
\end{split}
\end{equation}
Moreover, there exists $\delta_0 > 0$ such that for all $k \geq 0$ and for $t - \frac{2}{3}\log|\sigma|\to+\infty$,
\begin{equation}\label{TYError}
\begin{split}
|\nabla_{|b|^{\frac{1}{2}}g_{\mathcal C,\sigma}}^k(|b|^{\frac{1}{2}}\Phi_\sigma^*\psi_{TY_\sigma}-|b|^{\frac{1}{2}}\psi_{\mathcal C,\sigma})|_{{|b|^{\frac{1}{2}}g_{\mathcal C,\sigma}}}&=O_k(|b|^{\frac{2-k}{4}}e^{-{\delta}_0 (t-\frac{2}{3}\log|\sigma|)^{1/2}}),\\
|\nabla_{|b|^{\frac{1}{2}}g_{\mathcal C,\sigma}}^k(|b|^{\frac{1}{2}}\Phi_\sigma^*\omega_{TY_\sigma}-|b|^{\frac{1}{2}}\omega_{\mathcal C,\sigma})|_{{|b|^{\frac{1}{2}}g_{\mathcal C,\sigma}}}&=O_k(|b|^{-\frac{k}{4}}e^{-{\delta}_0 (t-\frac{2}{3}\log|\sigma|)^{1/2}}).
\end{split}
\end{equation} 
\end{lemma}

\begin{proof}
This is trivial except for the following two observations. First, $t - \frac{2}{3}\log |\sigma| = (m_\sigma')^*t$. Second, $J_{\mathcal{C},\sigma} := (m_\sigma')^*J_{\mathcal{C}}$ and $\Omega_{\mathcal{C},\sigma} := (m_\sigma')^*\Omega_{\mathcal{C}}$ are actually equal to $J_\mathcal{C}$ resp. $\Omega_{\mathcal{C}}$. 
\end{proof}

Lastly, we need to compare the holomorphic volume form $\Omega_{TY_\sigma}$ on $TY_\sigma$ to the holomorphic volume form $\Omega_\sigma$ on $\mathcal{Y}_\sigma \subset \mathcal{X}_\sigma$ on a uniform neighborhood of the origin in $\mathbb{C}^3$.

\begin{lemma}\label{Hvolumeratio}
Recall the local biholomorphism $\Psi_\sigma: \mathcal{Y}_\sigma \to TY_\sigma$ from \eqref{Psi}. Then for all $\sigma\in\Delta^*$, 
\begin{equation}
(\Psi_\sigma^{-1})^*\Omega_\sigma = (1 + O(|z|))\Omega_{TY_\sigma}
\end{equation}
as $|z| \to 0$, and the implied constant is independent of $\sigma$.
\end{lemma}

\begin{proof}
There exists an $\varepsilon>0$ such that for every fixed $\sigma \in \Delta^*$, the function $H_\sigma := [(\Psi_\sigma^{-1})^*\Omega_\sigma]/\Omega_{TY_\sigma}$ is holomorphic on $TY_\sigma \cap B_\varepsilon(0)$. Since all data depend holomorphically on $\sigma$, it follows that \begin{equation}
H: B_\varepsilon(0)\setminus TY_0 \to \C, \;\,z \mapsto H_{\sigma}(z)\;\,\text{with}\;\,\sigma = z_1^3 + z_2^3 + z_3^3 \neq 0,
\end{equation}
is holomorphic.

\begin{claim}\label{claim:hartogs}
$H$ extends holomorphically to $B_\varepsilon(0)$.
\end{claim}

\begin{proof}[Proof of Claim \ref{claim:hartogs}]
It is enough to show that $H$ extends holomorphically to $B_\varepsilon(0) \setminus \{0\}$ because then the claim follows from Hartogs' theorem. Obviously the numerator, $(\Psi_\sigma^{-1})^*\Omega_\sigma$, is holomorphic even at the points of $TY_0 \setminus \{0\}$. For the denominator, $\Omega_{TY_\sigma}$, cover $\mathbb{C}^3 \setminus \{0\}$ by the open sets $\{z_i \neq 0\}$ $(i = 1,2,3)$. By symmetry we only need to consider the case $i = 1$. By adjunction, there exists $\gamma \in \mathbb{C}^*$ such that $\Omega_{TY_1}$ is the restriction to the tangent bundle of $TY_1$ of the $2$-form $\gamma(dz_2 \wedge dz_3)/z_1^2$ on $\{z_1 \neq 0\}$. The latter is invariant under $m_\sigma$. Thus, $\Omega_{TY_\sigma}$ is given by the same formula, so it extends as a holomorphic and nowhere vanishing section of the relative canonical bundle across $TY_0 \setminus \{0\}$.
\end{proof}

It remains to show that $H(0) = 1$. To see that this is true, we approach the origin along a sequence of points in $TY_0\setminus\{0\}$. Along this sequence, $[(\Psi_0^{-1})^*\Omega_0]/\Omega_{\mathcal{C}} \to 1$ by the normalization chosen after the proof of Lemma \ref{OmegaOC}, so we need to show that $\Omega_{TY_0}/\Omega_{\mathcal{C}} \to 1$ as well. In fact, $\Omega_{TY_0} = \Omega_{\mathcal{C}}$ on $TY_0 \setminus \{0\}$. The reason is that these forms are scale-invariant on the cone $TY_0 \setminus \{0\}$ and have the same residue at the infinity divisor $E$; compare the general construction of $\Omega_X$ before Theorem \ref{t:hein}.
\end{proof}

\subsection{A new neck region between a Tian-Yau end and a hyperbolic cusp}\label{ss:new_neck}

We now return to the radial Kähler-Einstein equation \eqref{eq-psi-2} on a general $n$-dimensional cusp singularity, which already gave us the asymptotic model $\psi_{cusp}$ for the unique complete Kähler-Einstein potential on the regular part of our $2$-dimensional example $\mathcal{X}_0$.
In general, \eqref{eq-psi-2} is equivalent to
\begin{align}\label{eq-psi-3}
\frac{1}{n+1}(\psi')^{n+1} = e^{\psi+a}+b
\end{align}
for some constant $b$.
Then we have, for any $t<t_0,$
\begin{align}\label{eq-psi-4}
\int_{\psi(t)}^{\psi(t_0)} (e^{\xi+a}+b )^{-\frac{1}{n+1}}d\xi = (n+1)^\frac{1}{n+1}(t_0-t).
\end{align}	
The case $b = 0$ yields the cusp solution $\psi_{cusp}$. The case $b > 0$ yields solutions that are still defined on an entire negative half-line but are not metrically complete as $t \to -\infty$. In \cite[Sect 2]{BG} and \cite[Ex 2.7]{FHJ}, it was shown that the completions of these metrics may be viewed as Kähler-Einstein edge metrics on the total space of $L$, with conical singularities along the zero section that get pushed off to infinity as $b \to 0$ while the diameter of the zero section shrinks to zero. Since we are now considering smoothings rather than resolutions of singularities, it is natural to try to use the case $b < 0$.

We showed in \cite{FHJ} that when $b<0,$ there are solutions to \eqref{eq-psi-4} which correspond to incomplete Kähler-Einstein metrics on $\{\delta^-<h\leq \delta^+\}$, where $\delta^-, \delta^+>0$ are constants that depend on $a, b$. In fact,  \eqref{eq-psi-1} and \eqref{eq-psi-3} imply that $e^{\psi+a}+b>0$. When $b<0,$ the integral on the left-hand side of \eqref{eq-psi-4} is bounded and the bound is independent of $t, t_0$. The purpose of this section is to study these solutions in detail. They give rise to ``horn metrics'' $\omega_T$ approximating $\omega_{cusp}$ as in Figure \ref{fig:cusp_horn}.

\begin{convention}
    We will introduce geometric parameters $T < T + 2T_0 < 2\tau < 0$, and $b,T,T_0,\tau$ will eventually be made to depend on $\sigma$. The standing assumption is that as $\sigma \to 0$ we have that 
\begin{align}\label{eq-b-tau}
T \to -\infty, \;\, T_0 \to +\infty, \;\, T_0/T \to 0, \;\, \tau \to -\infty, \;\, \tau/T \to 0, \;\, b \to 0, \;\, b|\tau|^{n+1} \to 0.
\end{align}
We will be proving many $O(\ldots)$ type estimates for various different quantities, and the understanding will always be that the $O(\ldots)$ holds as $\sigma \to 0$, so that \eqref{eq-b-tau} is in force.
\end{convention}

We assume $\psi_T(t)$ is the solution of \eqref{eq-psi-2} such that
\begin{align}\label{eq-psi-10}
e^{\psi_T(T) +a} +b=0,\quad
\psi_T(\tau) = \psi_{cusp} (\tau).
\end{align}
Then
\begin{align}\label{eq-psi-5}
\psi_T(T) = \log |b|-a, \quad \psi_T(\tau) = -(n+1)\log( - \tau)+n\log (n+1) -a.
\end{align}
We note here that with $\omega_T := \ddbar \psi_T$ it holds that
\begin{equation}\label{normalizingpotential2}e^{-\psi_T}\omega_T^n=i^{n^2}\Omega_{\mathcal C}\wedge\overline{\Omega}_{\mathcal C}.\end{equation}
Indeed, radial pluriharmonic functions are constant (cf. \eqref{eq:radial_ph}), so there is a constant $C_T$ such that \begin{equation}e^{-\psi_T+C_T}\omega_T^n=i^{n^2}\Omega_{\cC}\wedge\overline{\Omega}_{\cC}.\end{equation}
Now by taking $t=\tau$ and noting that $\psi_T(\tau)=\psi_{cusp}(\tau)$ (cf. \eqref{eq-psi-10}), one has
\begin{equation}
i^{n^2}\Omega_{\cC}\wedge\overline{\Omega}_{\cC}=e^{-\psi_T(\tau)+C_T}\omega_T^n=e^{-\psi_{cusp}(\tau)+C_T}\omega_T^n=e^{C_T}i^{n^2}\Omega_{\cC}\wedge\overline{\Omega}_{\cC}
\end{equation}
as before. So indeed $C_T=0$. 

\begin{figure}[!ht]
\caption{Cusp metric $\omega_{cusp}$, horn metric $\omega_T$, gluing regions.}\label{fig:cusp_horn}
\vspace{-1mm}
\begin{center}
\includegraphics[width=160mm]{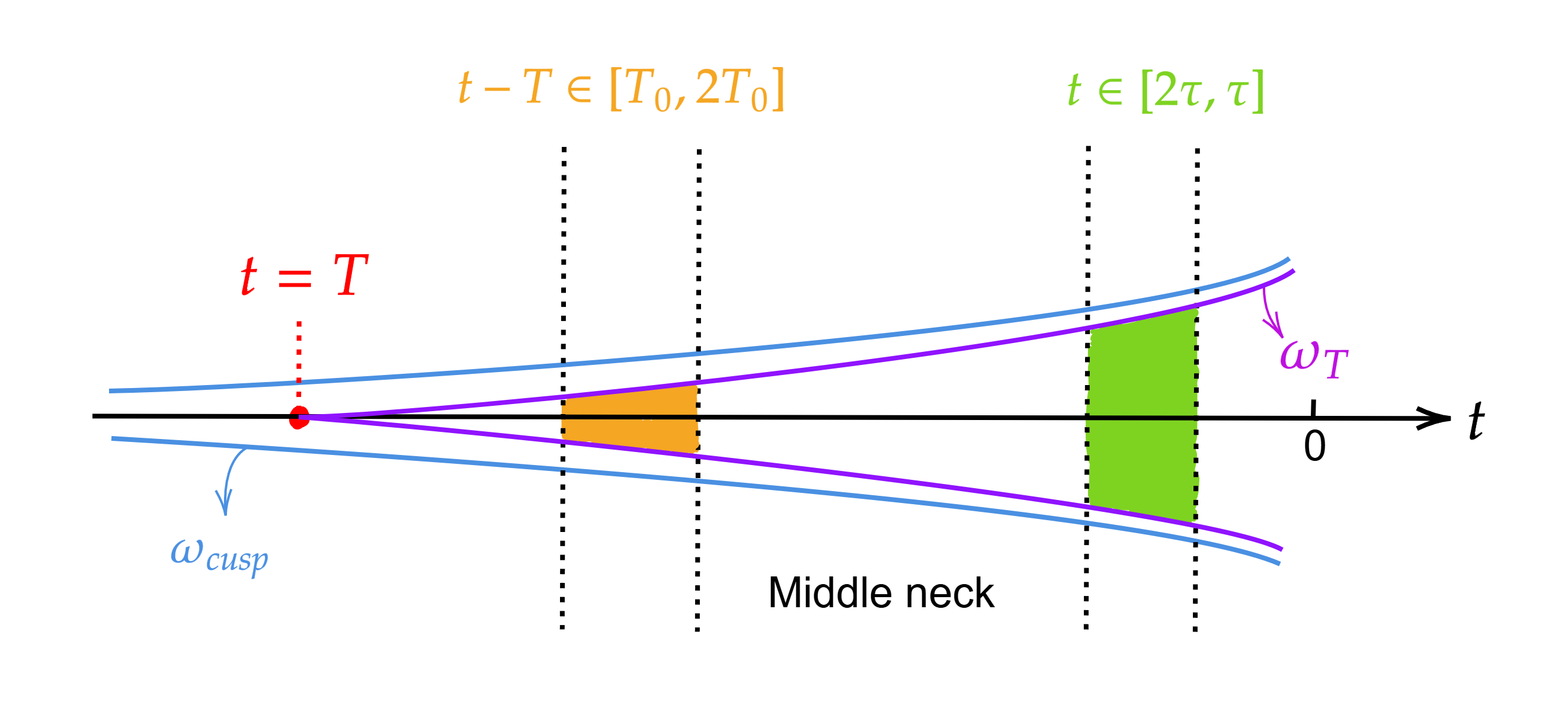}
\end{center}
\vspace{-7mm}
\end{figure}

\begin{remark}\label{rem:colored_regions}
We show the relationship between $\omega_T$ and $\omega_{cusp}$ in Figure \ref{fig:cusp_horn}. The metric $\omega_T$ has a horn singularity at $t = T$, which evolves into the cusp singularity as $T \to -\infty$. Later we will see that for a fixed $T$, the horn of $\omega_T$ is asymptotic to a scaled copy of the Calabi model space, which allows us to carry out a gluing construction in the region $t \in [T+T_0,T+2T_0]$ that we shaded orange in the figure ($T_0 \to +\infty$ but $T_0/T \to 0$). Similarly, in the green region, $t \in [2\tau,\tau]$ ($\tau\to-\infty$ but $\tau/T \to 0$), we can glue $\omega_T$ with $\omega_{cusp}$. We often refer to the region $T + 2T_0 < t < 2\tau$ as the ``new'' or ``middle'' neck.
\end{remark}

By \eqref{eq-psi-4},\begin{align}\label{eq-psi-6}
\int_{\psi_T(T)}^{\psi_T(\tau)} (e^{\xi+a}+b )^{-\frac{1}{n+1}}d\xi = (n+1)^\frac{1}{n+1}(\tau-T).
\end{align}	
\eqref{eq-psi-4} and \eqref{eq-psi-5} show that, by setting $\xi=s+\log |b|-a$,
\begin{align}\label{eq-bT}
\int_0^{-(n+1)\log |\tau| -\log|b|+n\log (n+1) }(e^s-1)^{-\frac{1}{n+1}}ds = (n+1)^\frac{1}{n+1}|b|^\frac{1}{n+1}(\tau -T).
\end{align}
Hence we obtain the crucial relation between $T$ and $b$:
\begin{align}\label{eq-T-b}
T =- (n+1)^{-\frac{1}{n+1}}  c(n) |b|^{-\frac{1}{n+1}} +O(\tau),
\end{align}
where
\begin{align}\label{eq:def_c(n)}
c(n) := \int_0^\infty (e^s-1)^{-\frac{1}{n+1}} ds.
\end{align}
We conclude that
\begin{align}
\psi_{cusp}(T) = \log |b| -(n+1)\log c(n)  +O(|b|^\frac{1}{n+1}\tau).
\end{align}
 Hence,
\begin{align}\label{eq-psi-T}
\psi_{cusp}(T) - \psi_T(T)=   -(n+1)\log c(n) +a +O(|b|^\frac{1}{n+1}\tau),
\end{align}
which tends to a constant as $\sigma\to 0$ by \eqref{eq-b-tau}.

\subsubsection{Estimates asymptotically close to the cusp \textup{(}green region\textup{)}}

Similarly to \eqref{eq-psi-4}, one has
\begin{align}
\int_{\psi_T(t)}^{\psi_T(\tau)} (e^{\xi+a}+b )^{-\frac{1}{n+1}}d\xi = (n+1)^\frac{1}{n+1}(\tau-t).
\end{align}
Substitute $\xi = s+\psi_T(\tau)$. 
By \eqref{eq-psi-5},
\begin{align}\label{eq-psi-7-c}
\int_{\psi_T(t)-\psi_T(\tau)}^{0} (e^{s}+(n+1)^{-n} b|\tau|^{n+1})^{-\frac{1}{n+1}}ds =- (n+1)(1-\tau^{-1}t).
\end{align}
Let $c_1>0$ be the unique constant such that
\begin{align}\label{eq-es-c}
\int_{-c_1}^0 (e^s)^{-\frac{1}{n+1}}ds = n+1.
\end{align}	
When $t = 2 \tau$, the right-hand side of \eqref{eq-psi-7-c} equals $n+1$. As $b<0$, we have
\begin{align}
(e^{s}+e^{-a} b|\tau|^{n+1})^{-\frac{1}{n+1}} > (e^s)^{-\frac{1}{n+1}}>0.
\end{align}
By comparing  \eqref{eq-psi-7-c} and \eqref{eq-es-c},
 we derive that $\psi_T(2\tau)-\psi_T(\tau)> - c_1$. By taking $b\rightarrow 0^-$ in \eqref{eq-psi-7-c}, we see that $\psi_T(t)-\psi_T(\tau)$ must decrease to a
constant. On the other hand,
\begin{align}
\int_{\psi_{cusp}(2\tau)-\psi_{cusp}(\tau)}^{0} (e^{s})^{-\frac{1}{n+1}}ds =n+1.
\end{align}
 By \eqref{eq-b-tau} and all the discussions above,
$\psi_T(2\tau)-\psi_T(\tau)$ converges to
$
\psi_{cusp}(2\tau)-\psi_{cusp}(\tau).
$
By \eqref{eq-psi-10}, $\psi_T(2\tau) \rightarrow \psi_{cusp}(2\tau)$.
Replacing $2\tau$ by any constant in $(2\tau, \tau)$, the same argument implies that,
\begin{align}
\psi_T(t) - \psi_{cusp}(t) =o(1)
\end{align}
as $b\rightarrow 0^-$ if we fix $\frac{t}{\tau}$ as a constant in $ [1, 2]$. 

When $t= c\tau$ for some constant $c\in [1, 2]$,
 $s\geq \psi_T(t)-\psi_T(\tau) \geq -c_1$ in \eqref{eq-psi-7-c}. Using that
\begin{align}
\begin{split}
\int_{\psi_T(t)-\psi_T(\tau)}^{0} (e^{s}+(n+1)^{-n} b|\tau|^{n+1})^{-\frac{1}{n+1}}ds &= (n+1)(c-1)\\
&= \int_{\psi_{cusp}(t)-\psi_{cusp}(\tau)}^{0} (e^{s})^{-\frac{1}{n+1}}ds,
\end{split}
\end{align}
we derive
\begin{align}
\int_{\psi_T(t)-\psi_T(\tau)}^{0}\left( (e^{s}+(n+1)^{-n} b|\tau|^{n+1})^{-\frac{1}{n+1}} - (e^{s})^{-\frac{1}{n+1}} \right) ds = \int_{\psi_{cusp}(t)-\psi_{T}(t)}^{0} (e^{s})^{-\frac{1}{n+1}}ds,
\end{align}
which shows that, for $t \in [2\tau, \tau]$,
\begin{align}
\label{eq-psi-8-c}
\psi_T(t) - \psi_{cusp}(t) =O(|b| |\tau|^{n+1}).
\end{align}

Let $\rho :=\sqrt{(n+1)/2} \log (-t)$, which is a Busemann function of $\omega_{cusp}$.
For $t \in [2\tau, \tau]$, $\rho$ varies by an additive constant independent of $b$. Set the normalized Busemann function \begin{align}
\tilde \rho(t) := \rho(t) - \rho(\tau) =\sqrt{\frac{n+1}{2}} \log \left( \frac{t}{\tau} \right),
\end{align}
whose range is
$I = 
[0, \sqrt{(n+1)/2}\log 2].
$ Regard $ t \psi_T',  t^2 \psi_T''$ as functions of $\tilde \rho$.
As $b \rightarrow 0^-,$
\begin{align}
\begin{split}
 t \psi_T' &=  t (n+1)^\frac{1}{n+1} (e^{\psi_T+a}+b)^\frac{1}{n+1}\\
&= - e^{\sqrt{\frac{2}{n+1}}\tilde \rho} (n+1)^\frac{1}{n+1} (e^{\psi_T-\psi_T(\tau)+a}+b |\tau|^{n+1} )^\frac{1}{n+1}\\
&= -(n+1)+O(|b| |\tau|^{n+1}),
\end{split}
\end{align}
uniformly for $\tilde \rho \in I$,
where we have applied the fact that
\begin{align}
    \psi_T-\psi_T(\tau)=\psi_{cusp}-\psi_{cusp}(\tau)+O(|b||\tau|^{n+1}) =-(n+1)\log |t/\tau| +O(|b||\tau|^{n+1}).
\end{align}
  A similar computation shows that
\begin{align}
t^2 \psi_T'' 
= t^2 e^{\psi_T +a}(\psi_T')^{1-n} = e^{\psi_T -\psi_{cusp} +n\log (n+1)}(-t\psi_T')^{1-n} = n+1+O(|b||\tau|^{n+1})
\end{align}
as $b\rightarrow 0^-$. By differentiating
\begin{align}\label{eq-diff-psi}
t^2 \psi_T'' 
= (n+1)^n e^{\psi_T -\psi_{cusp}}(-t\psi_T')^{1-n}
\end{align}
multiple times
and multiplying by $t$ each time, we derive the following lemma.

\begin{lemma}\label{lem-cusp}
For $k\geq 1,$
\begin{align}
t^k\psi_T^{(k)}=(-1)^k(n+1)(k-1)! + O_k(|b||\tau|^{n+1})
\end{align}
on $\{\tilde \rho \in I\}$ as $b\rightarrow 0^-$. Here $(-1)^k(n+1)(k-1)! = t^k\psi_{cusp}^{(k)}$.
\end{lemma}

Introduce a local bundle chart $(z_1,\ldots,z_n)$ on the total space of $L$ such that $z_n = 0$ cuts out the zero section. Write the Hermitian metric as $h = e^{-\varphi}|z_n|^2$, where $\varphi$ depends only on $z_1,\ldots,z_{n-1}$. Then for any radial potential $\psi = \psi(t)$, $t = \log h$, the associated $(1,1)$-form can be written as
\begin{align}
\omega = i\partial \overline{\partial} \psi= -\psi'  i\partial \overline{\partial} \varphi + \psi'' i(\partial \varphi - \partial \log z_n) \wedge (\overline{\partial} \varphi - \overline{\partial} \log \overline{z}_n).
\end{align}
We now assume that the Calabi-Yau $(E,\omega_E)$, $\omega_E = -\ddbar\varphi$, is flat and $z_1,\ldots,z_{n-1}$ are the standard linear coordinates on the universal cover $\mathbb{C}^{n-1} \to E$. We also write $z_\alpha = x_\alpha + iy_\alpha$ for $1 \leq \alpha \leq n-1$, $\theta = {\rm arg}\,z_n$ and $x = -1/t$. Then, on the universal cover of $\{2\tau\leq t\leq \frac{1}{2}\tau\}$, we define a new chart via
\begin{align}\label{eq-quasi}
(\check x_\alpha, \check y_\alpha, \check x, \check \theta) := (|\tau|^{-\frac{1}{2}} x_\alpha, |\tau|^{-\frac{1}{2}} y_\alpha, |\tau| x, |\tau|^{-1}\theta ) 
\end{align}
which takes the value $p = (0,\ldots,0,1,0)$ at a point $q$ such that $t(q) =\tau$. Under this chart,  the metric $\omega_{cusp}$ is equivalent to $\delta_{ij}$ and in addition, for all $k,\alpha$,
\begin{align}
\psi'_{cusp}  i\partial \overline{\partial} \varphi, \quad \psi''_{cusp}  i(\partial \varphi - \partial \log z_n ) \wedge (\overline{\partial} \varphi - \overline{\partial} \log \overline{z}_n)
\end{align}
have bounded $C^{k}$ norms on a fixed size ball centered at $p$, with bounds that are independent of $\tau$. Also, $\psi_T, \psi_{cusp}$ are independent of $x_\alpha, y_\alpha, \theta.$
Applying this chart, Lemma \ref{lem-cusp} implies that, for 
$t \in [2\tau, \tau],$ 
\begin{align}\label{eq-cusp-T-1}
    |\nabla_{\omega_{cusp}}^k (\psi_T -\psi_{cusp})|_{\omega_{cusp}}=O_k(|b||\tau|^{n+1}).
\end{align}
Meanwhile, to prove higher order derivative estimates of  $\omega_T$ under the coordinates \eqref{eq-quasi}, we only have to check that 
\begin{align}\label{eq-psi-prime-doubleprime}
\frac{\psi'_T}{\psi'_{cusp}}, \;\,\frac{\psi''_T}{\psi''_{cusp}}
\end{align}
have uniformly bounded $C^{k}$ norm with respect to \eqref{eq-quasi}. For example,
\begin{align}\label{eq-psi-11-c}
\frac{\partial}{\partial \check x} \frac{\psi''_T}{\psi''_{cusp}} = -|\tau|^{-1} \frac{\partial t}{\partial x}\cdot \frac{\psi'''_T \psi_{cusp}'' - \psi''_T \psi_{cusp}'''}{(\psi_{cusp}'')^2},
\end{align}
where 
\begin{align}
\begin{split}
t^5(\psi'''_T \psi_{cusp}'' - \psi''_T \psi_{cusp}''')=  (t^3\psi'''_T -t^3 \psi_{cusp}''')t^2\psi_{cusp}'' - (t^2\psi''_T -t^2 \psi_{cusp}'')t^3\psi_{cusp}''' 
,
\end{split}
\end{align}
which is $O(|b| |\tau|^{n+1})$ by Lemma \ref{lem-cusp}. This shows that
\begin{align}\label{eq-ddxhat-psu}
\frac{\partial}{\partial \check x} \frac{\psi''_T}{\psi''_{cusp}} = O(|b||\tau|^{n+1}).
\end{align}
We can estimate the $C^k$ norm of \eqref{eq-psi-prime-doubleprime} with $k\geq 1$ by induction, using the formula
\begin{align}\label{eq-high-quotient}
 \left(\frac{f}{g}\right)^{(k)} = \frac{1}{g}\left(f^{(k)} -k!  \sum_{j=1}^k \frac{g^{(k+1-j)}}{(k+1-j)!} \frac{\left(\frac{f}{g}\right)^{(j-1)}}{(j-1)!}\right)
\end{align}
with
\begin{align}
f=\psi'_T - \psi'_{cusp},\;\, g = \psi'_{cusp}
\end{align}
or with
\begin{align}\label{eq-fg}
f=\psi''_T - \psi''_{cusp},\;\, g = \psi''_{cusp}.
\end{align}
By Lemma \ref{lem-cusp} and \eqref{eq-psi-cusp}, we have for any $k\geq 1$ that $f^{(k)} = O(|b||\tau|^{n+1})$.

Thus, in conclusion, we have proved the following lemma.

\begin{lemma}\label{CusptoNeckerror}
Let $\psi_T$ be the solution of \eqref{eq-psi-2} satisfying \eqref{eq-bT} and assume that \eqref{eq-b-tau} is in force.
 Then for any $k\in \mathbb{N}$ and for all $t \in [2\tau, \tau]$,
\begin{align}\label{eq-Ck-T-c}
|\nabla_{\omega_{cusp}}^k (\omega_T - \omega_{cusp})|_{\omega_{cusp}}= O(|b||\tau|^{n+1}).
\end{align}
\end{lemma}

\subsubsection{Estimates on the middle neck}

\begin{proposition}\label{prop-psi-diff}
Assume that $b\in (-\frac{1}{2}, 0)$. Let $\psi_T$ be the solution of \eqref{eq-psi-2} satisfying \eqref{eq-bT}. 
Then:
\begin{enumerate}
\item[$(1)$] For $ t\in [T, \tau],$
\begin{align}\label{eq-psi-T-max}
0\leq \psi_T(t) - \psi_{cusp}(t) < C,
\end{align}
where $C$ does not depend on $b$. 

\item[$(2)$] $\psi_T(t) - \psi_{cusp}(t)$ is a  decreasing function on $[T, \tau]$.

\item[$(3)$] $\psi_T'(t) < \psi_{cusp}'(t)$ for $t \in [T, \tau]$.
\end{enumerate}
\end{proposition}

\begin{proof}
By \eqref{eq-psi-10}, \eqref{eq-psi-T},
when $t =\tau$ or $T$, $0\leq \psi_T- \psi_{cusp}<C$, where $C$ is independent of $b$.  Now assume that $\psi_T - \psi_{cusp}$ has a positive local maximum at $p \in (T, \tau)$. Then
\begin{align}
&\psi_T(p) > \psi_{cusp}(p),\\
&0<\psi_T'(p) = \psi_{cusp}'(p),\\
&0<\psi_T''(p) \leq \psi_{cusp}''(p),
\end{align}
which contradicts the fact that  $\psi_T$ and $\psi_{cusp}$ both satisfy equation \eqref{eq-psi-2}.
Similarly, $\psi_T - \psi_{cusp}$ does not have a negative local minimum in $(T, \log |b|)$. Part (1) of the lemma is proved.

Notice that
$\psi_T - \psi_{cusp} = 0$ at $\tau$. If $\psi_T - \psi_{cusp}$ is not monotone, then it has either a positive local maximum or a negative local minimum in $(T, \tau)$, which contradicts the above discussion. At $\tau$, $\psi_T' < \psi_{cusp}'$ by \eqref{eq-psi-3}. So $\psi_T - \psi_{cusp}$ is a decreasing function.

For part (3), we assume that it is not true and set
\begin{align}
t_0 := \sup \{ t \in [T, \tau]: \psi_T'(t) \geq \psi'_{cusp}(t) \}.
\end{align}
Then $t_0 < \tau$. At $t_0$, we have
\begin{align}
&0<\psi_T'(t_0) = \psi_{cusp}'(t_0),\\
&0<\psi_T''(t_0) \leq  \psi_{cusp}''(t_0),
\end{align}
By part (2), $\psi_T(t_0) > \psi_{cusp}(t_0)$. This contradicts equation \eqref{eq-psi-2}.
\end{proof}

\begin{proposition}\label{prop-T-c}
Fix $\delta \in (0,1)$. Then for all $\eta \in (0,\delta]$ we have that
\begin{align}\label{eq-psi-T-eta-0}
\psi_T (\eta T) = \psi_{cusp} (\eta T) + O(|b|^\frac{1}{n+1}\tau) +O_\delta(\eta^{n+1}).
\end{align}
 In addition,
\begin{align}
\psi'_T(\eta T) &= \psi_{cusp}' (\eta T)e^{O_\delta(\eta^{n+1}) } (1+O(|b|^\frac{1}{n+1}\tau)),\label{eq-psi-T-eta-1}\\
\psi''_T(\eta T) &=\psi_{cusp}''(\eta T)e^{O_\delta(\eta^{n+1}) } (1+O(|b|^\frac{1}{n+1}\tau) ).\label{eq-psi-T-eta-2}
\end{align}
\end{proposition}

\begin{proof}
Taking $t_0 =\eta T$ and $t=T$ and substituting $\xi = s+ \psi_T(T)$ in \eqref{eq-psi-4}, we get
\begin{align}\label{eq-etaT-T}
\int_{0}^{\psi_T(\eta T)-\psi_T(T)} (e^{s}-1 )^{-\frac{1}{n+1}}ds = (n+1)^\frac{1}{n+1}|b|^{\frac{1}{n+1}}(\eta T-T).
\end{align}
By part (2) of Proposition \ref{prop-psi-diff}, as $\eta \to 0^+$, 
\begin{align}
    \psi_T (\eta T) - \psi_T (T) > \psi_{cusp} (\eta T)   -\psi_{cusp} (T) -C \rightarrow \infty.
\end{align}
Taking  $\eta \rightarrow 0^+$ in \eqref{eq-etaT-T},
\begin{align}\label{bandT}
\int_{0}^{\infty} (e^{s}-1 )^{-\frac{1}{n+1}}ds = (n+1)^\frac{1}{n+1}|b|^{\frac{1}{n+1}}|T|.
\end{align}
Subtracting it from \eqref{eq-etaT-T},
\begin{align}
\begin{split}
\int_{\psi_T (\eta T) - \psi_T (T)}^{\infty} (e^{s}-1 )^{-\frac{1}{n+1}}ds &= (n+1)^\frac{1}{n+1}\eta|b|^{\frac{1}{n+1}}|T|\\
&=\eta\left(\int_0^\infty (e^{s}-1 )^{-\frac{1}{n+1}}ds  +O( |b|^\frac{1}{n+1}\tau)\right),
\end{split}
\end{align}
where we applied \eqref{eq-T-b} for the last step.
When  $s>0$ is bounded away from 0,
\begin{align}
(e^{s}-1 )^{-\frac{1}{n+1}} = e^{-\frac{s}{n+1}} (1 - e^{-s})^{-\frac{1}{n+1}} = e^{-\frac{s}{n+1}} + O(e^{-\frac{(n+2)s}{n+1}}).
\end{align}
Thus,
\begin{align}
\begin{split}
\psi_T (\eta T) - \psi_T (T) &=(n+1)\log (n+1) -(n+1)\log \eta\\
&-(n+1)\log \left(\int_0^\infty (e^{s}-1 )^{-\frac{1}{n+1}}ds \right)+O(|b|^\frac{1}{n+1}\tau) +O(\eta^{n+1}).
\end{split}
\end{align}
Comparing this to
\begin{align}
\begin{split}
\psi_{cusp} (\eta T) - \psi_T (T) &= -(n+1)\log (-\eta T) - \log |b| +n\log(n+1) \\
&=-(n+1) \log\eta +(n+1)\log (n+1) \\
&-(n+1)\log \left(\int_0^\infty (e^{s}-1 )^{-\frac{1}{n+1}}ds \right)+O(|b|^\frac{1}{n+1}\tau),
\end{split}
\end{align}
we obtain that
\begin{align}
\psi_T (\eta T) - \psi_{cusp} (\eta T) = O(|b|^\frac{1}{n+1}\tau) +O(\eta^{n+1}).
\end{align}
The rest of the proposition follows by analyzing \eqref{eq-psi-3} and \eqref{eq-psi-2}.
\end{proof}

\begin{definition}[Perturbation of differential operators]\label{def-pert}
Given a function $\epsilon(\mathbf{x})>0$ on a local chart $\mathbf{x}$, we say that $f(\mathbf{x}) = O(\epsilon(\mathbf{x}))$ if $|f(\mathbf{x})| \leq C\epsilon(\mathbf{x})$ for some $C>0$ independent of $\mathbf{x}$ and of $b, T$.
Let
\begin{align}
    L &= \sum_{i,j=1}^m a_{ij}(\mathbf{x}) \partial^2_{ij} + \sum_{i=1}^m b_i(\mathbf{x}) \partial_i,\\
    \tilde L &= \sum_{i,j=1}^m \tilde a_{ij}(\mathbf{x}) \partial^2_{ij} + \sum_{i=1}^m\tilde b_i(\mathbf{x}) \partial_i,
\end{align}
be two differential operators. We say that $\tilde L = L + O(\epsilon(\mathbf{x}))$ if 
\begin{align}\label{eq:O(eps)terms}
\tilde a_{ij} = (1+O(\epsilon(\mathbf{x}))) a_{ij}, \quad   \tilde b_i =(1+O(\epsilon(\mathbf{x}))) b_{i}, 
\end{align}
for all $1\leq i, j\leq m.$ Similarly, we say that $\tilde L = L \cdot e^{O(\epsilon(\mathbf{x}))}$ if
\begin{align}
\tilde a_{ij} = e^{O(\epsilon(\mathbf{x}))} a_{ij}, \quad   \tilde b_i =e^{O(\epsilon(\mathbf{x}))} b_{i}.
\end{align}
\end{definition}

\begin{lemma}\label{lem-g-inv}
Recall the local holomorphic chart $(z_1,\ldots,z_n)$ introduced after Lemma \ref{lem-cusp}. Assume that in this chart a Kähler metric $g$ is given by
\begin{align}
g_{i \bar j} = -F\varphi_{i \bar j} + G\left(\varphi_i -\frac{\delta_{i n }}{z_n}\right)\left(\varphi_{\bar j} - \frac{\delta_{j n }}{\overline{z}_n}\right),
\end{align}
where $F, G$ are positive functions and $\varphi$ is a function of $z_1, \ldots, z_{n-1}$. Then its inverse $g^{\bar j k}$ is given by
\begin{align}
g^{\bar \beta \alpha} = -\frac{\varphi^{\bar \beta \alpha}}{F}\;\,(1 \leq \alpha,\beta \leq n-1), \;\, g^{\bar \alpha n} = -\frac{\varphi^{\bar \alpha \gamma}\varphi_{ \gamma}z_n}{F}, \;\, g^{\bar n n} = \frac{|z_n|^2 (F- G \varphi^{\bar \beta \alpha} \varphi_\alpha \varphi_{\bar \beta})}{F G}.
\end{align}
 \end{lemma}
 
The next corollary follows from Proposition \ref{prop-T-c}.

\begin{corollary}\label{cor-control-on-neck}
Fix $\delta \in (0,1)$ and let $\eta \in (0, \delta]$.  
Then at $t= \eta T$,
    \begin{align}\label{eq-g-T-eta}
        \Delta_{\omega_T} = (L_h +O(|b|^\frac{1}{n+1}\tau)) \cdot e^{O_\delta(\eta^{n+1})},
    \end{align}
where the cuspidal Laplacian $L_h$ was introduced in \cite[Lemma 2.5]{FHJ}. Using the notation $x = -1/t$ and $\theta = {\rm arg}\,z_n$ from \cite{FHJ}, an explicit formula for $L_h$ is
\begin{align}
\begin{split}
L_hu =  \frac{1}{n+1}(x^2 u_{xx}+(n+1) xu_x - (n+1)u -x^{-1}\varphi^{\bar \beta \alpha}u_{\alpha \bar \beta} +(2x)^{-1}u_{\theta\theta})+F,\\
F = F(x_\alpha,y_\alpha, x^{-1}u_{\theta\alpha}, x^{-1}u_{\theta\bar\alpha}, x^{-1}u_{\theta\theta}),
\end{split}
\end{align}
with $F$ smooth in $x_\alpha,y_\alpha$ and linear homogeneous in the other arguments. In addition, for all $k \geq 0$,
\begin{align}\label{eq-T-c-Dk-2}
|\nabla_{\omega_{cusp}}^k (\omega_T - \omega_{cusp})|_{\omega_{cusp}} = O_{\delta,k}(|b|^\frac{1}{n+1}|\tau| + \eta^{n+1}).
\end{align}
\end{corollary}

\begin{proof}
    \eqref{eq-g-T-eta} follows from Proposition \ref{prop-T-c}, Lemma \ref{lem-g-inv}, and the proof of Lemma 2.5 in \cite{FHJ}.

    To prove \eqref{eq-T-c-Dk-2}, we first rewrite \eqref{eq-diff-psi} as
\begin{align}\label{eq-diff-psi-2}
 \psi_T'' 
= (n+1)^n e^{\psi_T -\psi_{cusp}} (-t)^{-n-1}(\psi_T')^{1-n}
\end{align}
and differentiate it
    multiple times to inductively show that
    \begin{align}\label{eq-Psi-T-k-eta}
    \psi^{(k)}_T(\eta T) &=\psi_{cusp}^{(k)}(\eta T) (1+O(|b|^\frac{1}{n+1}\tau)+O(\eta^{n+1})).
    \end{align}
In fact, cases $k=1,2$ follow directly from Proposition \ref{prop-T-c}. 
For $k\geq 3$, when we differentiate \eqref{eq-diff-psi-2} $k-2$ times using the product rule,  every term on the right-hand side is of the form
\begin{align}
    (-t)^{-k} (C+O(|b|^\frac{1}{n+1}\tau)+O(\eta^{n+1})).
\end{align}
By formally writing $1 = e^{\psi_{cusp} - \psi_{cusp}}$ and differentiating
\begin{align}\label{eq-Psi-cusp-eta}
\psi_{cusp}'' 
= (n+1)^n e^{\psi_{cusp} -\psi_{cusp}} (-t)^{-n-1}(\psi_{cusp}')^{1-n}
\end{align}
$k-2$ times, we know that the constants $C$ add up to $(n+1)(k-1)!$ as
\begin{align}
 \psi_{cusp}^{(k)}(t) =(n+1)(k-1)!(-t)^{-k}.  
\end{align}
Here we use the fact that the $(k-2)$-th derivatives of \eqref{eq-diff-psi-2} and $(k-2)$-th derivatives of \eqref{eq-Psi-cusp-eta} are in similar forms.
With this, we derive \eqref{eq-Psi-T-k-eta}. Near a point $q$ with $t_* := t(q) \in [\delta T, \tau]$, we then consider quasi-coordinates for $\omega_{cusp}$ as in \eqref{eq-quasi}:
\begin{align}\label{eq-quasi-1}
(\check x_\alpha, \check y_\alpha, \check x, \check \theta) := ((-t_*)^{-\frac{1}{2}} x_\alpha, (-t_*)^{-\frac{1}{2}} y_\alpha, -t_* x, (-t_*)^{-1}\theta ).
\end{align}
Similarly as in the proof of Lemma \ref{lem-cusp}, we need to show that
the $C^k$ norm of 
\begin{align}
\frac{\psi'_T - \psi'_{cusp}}{\psi'_{cusp}}, \;\,\frac{\psi''_T - \psi''_{cusp}}{\psi''_{cusp}}
\end{align}
with respect to \eqref{eq-quasi-1} on $\{2t_*<t<\frac{1}{2}t_*\}$ is bounded by $C_k (|b|^\frac{1}{n+1}|\tau|+\eta^{n+1})$. 
For $k=0$, this follows directly from \eqref{eq-Psi-T-k-eta}. For $k\geq 1$, we prove it by induction by applying \eqref{eq-high-quotient}--\eqref{eq-fg}.
\end{proof}

\begin{proposition} \label{prop-L-s}
Let $\delta \in (0,1)$ be a small dimensional constant. Define the coordinate $s:= 1-t/T$ and restrict it to $s\in [2|T|^{\alpha-1},\delta]$ for some $\alpha \in (0,1)$. Then $\Delta_{\omega_T}$  is of the form
\begin{align}\label{eq:666}
\Delta_{\omega_T}  = L +O(|b|^{-\frac{1}{n+1}} s^{-\frac{1}{n}})\cdot \partial^2_{\theta \theta}+O(|b|^\frac{1}{n+1}\tau)+O(s^\frac{n+1}{n}).
\end{align}
Here the $O$ notation is understood in the sense of Definition \ref{def-pert} with respect to the local chart $s, \theta, z_\alpha$ and the model operator $L$ takes the following form:
\begin{align}
\begin{split}
    L u &= \left(\frac{n+1}{n}\right)^\frac{1}{n} c(n)^{-\frac{n+1}{n}}(n s^\frac{n-1}{n}u_{ss} + (n-1)s^{-\frac{1}{n}}u_s) + |b|^{-\frac{1}{n+1}}s^{-\frac{1}{n}} \mathcal{T} u\\
    &+ C|b|^{-\frac{2}{n+1}} s^\frac{n-1}{n} u_{\theta\theta}+s^{-\frac{1}{n}}\mathcal{H}u,
    \end{split}
\end{align}
where  $C>0$ is a constant, $\mathcal{T}$ is an elliptic operator on $E$ and $\mathcal{H}$ is linear homogeneous in 
\begin{align}\label{eq:form_of_H}
    u_{s \alpha}, u_{s \bar \beta},  u_{s\theta},  |b|^{-\frac{1}{n+1}} u_{\alpha \theta},  |b|^{-\frac{1}{n+1}} u_{\bar \beta \theta}
\end{align}
with coefficients that are smooth in $x_\alpha, y_\alpha$ and independent of $b,T$.
\end{proposition}

The term $O(s^\frac{n+1}{n})$ in \eqref{eq:666} stands for a comparison of differential operators as in \eqref{eq:O(eps)terms}, so for applications it is important to know that this term is $> -1$. This is ensured by our assumption $s \leq \delta$. Similarly, given the condition $s \geq 2|T|^{\alpha-1}$, we can observe that the term with $\partial^2_{\theta \theta}$ in  \eqref{eq:666} is
\begin{align}
O(|b|^{-\frac{1}{n+1}} s^{-\frac{1}{n}})\cdot \partial^2_{\theta \theta} = O(|b|^\frac{\alpha}{n+1}) \cdot |b|^{-\frac{2}{n+1}} s^\frac{n-1}{n} u_{\theta\theta}.
\end{align}
Hence,  this term represents a higher-order contribution compared to the operator  $L.$

\begin{proof}[Proof of Proposition \ref{prop-L-s}]
By \eqref{eq-etaT-T} and \eqref{eq-T-b},
\begin{align}\begin{split}\label{eq-psi-T-s}
\int_{0}^{\psi_T(\eta T)-\psi_T(T)} (e^{\xi}-1 )^{-\frac{1}{n+1}} d\xi &= (n+1)^\frac{1}{n+1}|b|^{\frac{1}{n+1}}(1-\eta)|T|\\
&=(c(n) +O(|b|^\frac{1}{n+1}\tau))(1-\eta),
\end{split}
\end{align}
where $1-\eta =s$ and the $O(|b|^\frac{1}{n+1}\tau)$ term is independent of $s.$
When $\xi>0$ is bounded,
\begin{align}\label{eq-xi-exp}
    (e^\xi-1)^{-\frac{1}{n+1}} = \xi^{-\frac{1}{n+1}}(1+O(\xi)) = \xi^{-\frac{1}{n+1}} + O(\xi^\frac{n}{n+1}).
\end{align}
Thus, we have that
\begin{align}
    \int_{0}^{\psi_T(\eta T)-\psi_T(T)} (e^{\xi}-1 )^{-\frac{1}{n+1}} d\xi = \frac{n+1}{n}({\psi_T(\eta T)-\psi_T(T)})^\frac{n}{n+1}(1+O({\psi_T(\eta T)-\psi_T(T)})).
\end{align}
By \eqref{eq-psi-T-s},
\begin{align}
    \psi_T(\eta T)-\psi_T(T) =\left(\frac{n}{n+1}\right)^\frac{n+1}{n} \left(c(n) +O(|b|^\frac{1}{n+1}\tau)\right)^\frac{n+1}{n}s^\frac{n+1}{n} (1+ O(\psi_T(\eta T)-\psi_T(T))),
\end{align}
which further implies that
\begin{align}\label{eq-psi-T-s-2}
    \psi_T(\eta T)-\psi_T(T) = \left(\frac{n}{n+1}\right)^\frac{n+1}{n} \left(c(n)+O(|b|^\frac{1}{n+1}\tau) \right)^\frac{n+1}{n} s^\frac{n+1}{n}(1+O(s^\frac{n+1}{n})).
\end{align}
The metric $\omega_T$ is given by
\begin{align}
    \omega_T = -\psi_T'i\partial \overline{\partial} \varphi +\psi_T''i(-\partial \varphi +\partial \log z_n) \wedge (-\overline{\partial} \varphi +\overline{\partial} \log z_n).
\end{align}
Thanks to \eqref{eq-psi-T-s-2}, we can evaluate the coefficients as follows:
\begin{align}
\begin{split}\label{eq-Psi-T-11}
    \psi_T'(\eta T)& = (n+1)^\frac{1}{n+1}(e^{\psi_T(\eta T) +a}+b)^\frac{1}{n+1}\\
    &= (n+1)^\frac{1}{n+1}|b|^\frac{1}{n+1}(e^{\psi_T(\eta T) -\psi_T(T)}-1)^\frac{1}{n+1}\\
    &= (n+1)^\frac{1}{n+1}\left(\frac{n}{n+1}\right)^\frac{1}{n}|b|^\frac{1}{n+1}s^\frac{1}{n}\left(c(n)^\frac{1}{n} +O(|b|^\frac{1}{n+1}\tau) +O(s^\frac{n+1}{n})\right),
    \end{split}\\
\begin{split}\label{eq-Psi-T-12}
    \psi_T''(\eta T) & = \frac{e^{\psi_T(\eta T)+a}}{(\psi_T'(\eta T))^{n-1}} =  \frac{|b|e^{\psi_T(\eta T)- \psi_T(T)}}{(\psi_T'(\eta T))^{n-1}}\\
    &= (n+1)^{\frac{1-n}{n+1}} \left(\frac{n}{n+1}\right)^{\frac{1-n}{n}}|b|^\frac{2}{n+1}s^{\frac{1-n}{n}}\left(c(n)^{\frac{1-n}{n}} +O(|b|^\frac{1}{n+1}\tau) +O(s^\frac{n+1}{n})\right).
    \end{split}
\end{align}

Setting $x= -1/t$, we have $x= (1-s)^{-1}|T|^{-1}$ and
\begin{align}
    \partial_x &= |T|(1-s)^2 \partial_s, \quad
    \partial_{xx}^2 = |T|^2(1-s)^4\partial_s^2 - 2|T|^2(1-s)^3 \partial_s.
\end{align}
Given this conversion formula, it will be enough to express $\Delta_{\omega_T}$ in terms of $\partial_x, \partial^2_{xx}$. This can be done by a computation similar to the proof of \cite[Lemma 2.5]{FHJ}. Set 
\begin{align}
Q := \varphi^{\bar \beta \alpha}\varphi_\alpha \varphi_{\bar \beta},    
\end{align}
which is uniformly bounded. Applying Lemma \ref{lem-g-inv} with $F= \psi'(\eta T)$, $G = \psi''(\eta T)$, we have that
\begin{align}
\begin{split}\label{eq-Del1}
    g^{\bar n n} \partial_{n} \partial_{\bar n} &= \frac{r^2(F - GQ)}{FG} \left(\frac{x^2}{z_n}\partial_x - \frac{i}{2z_n}\partial_\theta\right)
    \left(\frac{x^2}{\overline{z}_n}\partial_x + \frac{i}{2\overline{z}_n}\partial_\theta\right)\\
    &= \left(\frac{1}{G} -\frac{Q}{F}\right)
    \left(x^4 \partial^2_{xx} + 2 x^3 \partial_x + \frac{1}{4} \partial^2_{\theta\theta} \right)\\
    &= (n+1)^{\frac{n-1}{n+1}} \left(\frac{n}{n+1}\right)^{\frac{n-1}{n}}c(n)^{\frac{n-1}{n}}|b|^{-\frac{2}{n+1}} s^{\frac{n-1}{n}} \left(x^4 \partial^2_{xx} + 2 x^3 \partial_x + \frac{1}{4} \partial^2_{\theta\theta} \right) \\
    & -\frac{Q}{F}(x^4 \partial^2_{xx} + 2 x^3 \partial_x)+O(|b|^{-\frac{1}{n+1}} s^{-\frac{1}{n}}) \partial^2_{\theta \theta} +O(|b|^\frac{1}{n+1}\tau) +O(s^\frac{n+1}{n})\\
     &= (n+1) \left(\frac{n}{n+1}\right)^{\frac{n-1}{n}} c(n)^{-\frac{n+1}{n}} s^{\frac{n-1}{n}} \partial^2_{ss} + C|b|^{-\frac{2}{n+1}} s^\frac{n-1}{n} \partial^2_{\theta\theta}  \\
    &  -\frac{Q}{F}(x^4 \partial^2_{xx} + 2 x^3 \partial_x)+O(|b|^{-\frac{1}{n+1}} s^{-\frac{1}{n}}) \partial^2_{\theta \theta} +O(|b|^\frac{1}{n+1}\tau) +O(s^\frac{n+1}{n}),
        \end{split}\\
\begin{split}\label{eq-Del2}
    g^{\bar \alpha n}\partial_n\partial_{\bar \alpha} &=-\frac{\varphi^{\bar \alpha \gamma}\varphi_{ \gamma}z_n}{F} \left(\frac{x^2}{z_n}\partial_x -\frac{i}{2z_n}\partial_\theta\right)(\partial_{\bar \alpha} - x^2\varphi_{\bar \alpha} \partial_x)\\
    &= -\frac{x^2}{F}\varphi^{\bar \alpha \gamma}\varphi_{ \gamma}\partial^2_{x\bar \alpha}+\frac{Q}{F}(x^4\partial^2_{xx} +2x^3 \partial_x)-\frac{ix^2Q\partial^2_{x\theta} }{2F}+\frac{i\varphi^{\bar \alpha \gamma}
    \varphi_\gamma \partial^2_{
    \bar \alpha \theta}}{F}\\
    &= s^{-\frac{1}{n}} H_1(x_\alpha, y_\alpha,   \partial^2_{s\bar \alpha}, |b|^\frac{1}{n+1} \partial^2_{ss},|b|^\frac{1}{n+1} \partial_{s}, \partial^2_{s\theta},  |b|^{-\frac{1}{n+1}} \partial^2_{\bar \alpha \theta})+O(b^{n+1}\tau)+O(s^\frac{n+1}{n}),
        \end{split}\\
\begin{split}\label{eq-Del3}
    g^{\bar \beta \alpha}\partial_\alpha \partial_{\bar \beta} &= -\frac{\varphi^{\bar \beta \alpha}}{F}(\partial_{z_\alpha} - x^2 \varphi_\alpha \partial_x)(\partial^2_{\bar \beta} - x^2 \varphi_{\bar \beta} \partial_x)\\
    &=-\frac{\varphi^{\bar \beta \alpha}}{F}\partial^2_{\alpha\bar \beta} +(n-1)\frac{x^2\partial_x}{F} +\frac{x^2\varphi^{\bar \beta \alpha}(\varphi_\alpha \partial^2_{x\bar \beta} + \varphi_{\bar \beta} \partial^2_{x\alpha})}{F} -\frac{Q (x^4\partial^2_{xx} +2x^3 \partial_x)}{F}\\
    &= |b|^{-\frac{1}{n+1}}s^{-\frac{1}{n}}T + (n-1)\left(\frac{n}{n+1}\right)^{-\frac{1}{n}} c(n)^{-\frac{n+1}{n}} s^{-\frac{1}{n}}\partial_s  \\
    &+ s^{-\frac{1}{n}} H_2(x_\alpha, y_\alpha,  \partial^2_{s \alpha}, \partial^2_{s \bar \beta},  |b|^{\frac{1}{n+1}} \partial^2_{ss}, |b|^{\frac{1}{n+1}} \partial_{s}) +O(b^{n+1}\tau)+O(s^\frac{n+1}{n}),
    \end{split}
\end{align}
where $H_1, H_2$ are smooth in $x_\alpha,y_\alpha$ and linear homogeneous with respect to their differential operator arguments.
As $\Delta_{\omega_T}$ is a real operator, all terms with an $i$ must cancel out. In addition, the terms
\begin{align}
\frac{Q}{F}(x^4 \partial^2_{xx} + 2 x^3 \partial_x)    
\end{align}
in \eqref{eq-Del1}--\eqref{eq-Del3} cancel out. So the terms $|b|^{\frac{1}{n+1}} \partial^2_{ss}, |b|^{\frac{1}{n+1}} \partial_{s}$ in $H_1, H_2$ cancel out as well.
\end{proof}

Let us also note the following formulas for later use. By repeatedly differentiating \eqref{eq-Psi-T-11}--\eqref{eq-Psi-T-12}, we may prove inductively that, when $k>2$, there are some constants $C_{j,l}$ such that
\begin{align}\label{horderpsiT}
\begin{split}\psi^{(k)}_{T}(\eta T)=&\sum_{j=-(k-2)}^{n(k-2)}\sum_{l=1}^{k-1}\frac{C_{j,l}e^{l(\psi_T(\eta T)+a)}}{(\psi'_T(\eta T))^{n-1+j}}\\
=&\sum_{j=-(k-2)}^{n(k-2)}\sum_{l=1}^{k-1}\frac{C_{j,l}|b|^le^{l(\psi_T(\eta T)-\psi_T(T))}}{(\psi'_T(\eta T))^{n-1+j}}\\
=&\sum_{j=-(k-2)}^{n(k-2)}\sum_{l=1}^{k-1}C_{n,j,l}|b|^{l-\frac{n-1+j}{n+1}}s^{\frac{1-n-j}{n}}\left(c(n)^{\frac{1-n}{n}} +O(|b|^\frac{1}{n+1}\tau) +O(s^\frac{n+1}{n})\right).
\end{split}
\end{align}

The following lemma about the shape of the volume form of $\omega_T$ is also an easy consequence of the computations in this section.

\begin{lemma}\label{l:volumeform}
Write $s = 1- \frac{t}{T}$ as before. Then for every $\eta \in (0,1)$ we have that
\begin{align}\label{eq:volumeform_T}
\omega_T^n|_{t=\eta T} = |T|^{-n} \mu_T(1-\eta)\, ds \wedge d\theta \wedge d{\rm Vol}_E,
\end{align}
where the radial volume density $\mu_T: (0,1) \to \mathbb{R}^+$ satisfies the following properties:
\begin{itemize}
    \item[$(1)$] There are constants $e(n), e'(n)>0$ with $e(n) \leq \mu_T(1-\eta) \leq e'(n)\eta^{-(n+1)}$ for all $\eta,T$.
    \item[$(2)$] $\mu_\infty := \lim_{T \to -\infty} \mu_T$ exists pointwise and is smooth.
    \item[$(3)$] $\mu_\infty(s) \to const>0$ as $s \to 0^+$ and $\eta^{n+1}\mu_\infty(1-\eta) \to const>0$ as $\eta \to 0^+$.
\end{itemize}
\end{lemma}

\begin{proof}
From the Kähler-Einstein equation, it is straightforward to see that
\begin{align}
\omega_T^n|_{t=\eta T} = e''(n) |T|^{-n} e^{\psi_T(\eta T) - \psi_T(T)} ds \wedge d\theta \wedge d{\rm Vol}_E
\end{align}
for some constant $e''(n)>0$. For item (1), we use \eqref{eq-psi-T-max} to estimate
\begin{align}\begin{split}
0 \leq \psi_T(\eta T) - \psi_T(T) &\leq \psi_{cusp}(\eta T) + C - \psi_T(T) \\
&\leq -(n+1)\log\eta - (n+1)\log(-T) + C - \psi_T(T)\\
&\leq -(n+1)\log \eta + C.
\end{split}\end{align}
For items (2) and (3), we use the more precise formula \eqref{eq-psi-T-eta-0}, which is valid for all $\eta \in (0,1)$. This tells us that $\psi_T(\eta T) - \psi_T(T)$ has a pointwise limit as $T \to -\infty$, which is a smooth function of $\eta$ and which moreover differs from $-(n+1)\log \eta + C$ by $O(\eta^{n+1})$ as $\eta \to 0^+$. This also establishes the limit as $\eta \to 0^+$ in item (3). For the limit as $s \to 0^+$ in item (3), we can instead look at \eqref{eq-psi-T-s-2}.
\end{proof}

\subsubsection{Estimates asymptotically close to the Tian-Yau end \textup{(}orange region\textup{)}}

Recall from Remark \ref{rem:colored_regions} that the orange gluing region takes the form $t \in [T+T_0,T+2T_0]$ for some positive $T_0$ depending on $\sigma$ such that $T_0 \to +\infty$ but $T_0/T \rightarrow 0^-$ as $b \rightarrow 0^-$. From now on we fix an $\alpha \in (0,1)$ and set
\begin{align}\label{eq:def_T0}
T_0 := |T|^\alpha.
\end{align}

In this section we will prove decay estimates for the difference of the Tian-Yau metric and the neck metric $\omega_T$ in the orange region. We slightly abuse notation by writing
\begin{align}\label{eq-psi-C}
\psi_{\mathcal{C}}(t) = \frac{n}{n+1}(t-T)^{\frac{n+1}{n}}
\end{align}
for the Calabi model potential shifted by $T$. We also set
\begin{align}\label{cccccc}
\mathfrak{c} := n^{\frac{1}{n}}.
\end{align}
We wish to glue $\omega_T$ with $\mathfrak{c} |b|^\frac{1}{n} \omega_{\mathcal{C}}$. To this end, we introduce the error term
\begin{align}\label{ErrorE}
E(t):= \psi_T(t) - \psi_T(T) - \mathfrak{c} |b|^\frac{1}{n} \psi_{\mathcal{C}}.
\end{align}
Notice that
\begin{align}
\int_{\psi_T(T)}^{\psi_T(t)} (e^{\xi+a}+b )^{-\frac{1}{n+1}}d\xi = (n+1)^\frac{1}{n+1}(t-T).
\end{align}
Taking $s = \xi -\psi_T(T)$, we have that for $t-T \in [T_0, 2T_0]$, 
\begin{align}\label{eq-psi-t-T}
\int_{0}^{\psi_T(t)-\psi_T(T)} (e^{s}-1 )^{-\frac{1}{n+1}}ds = (n+1)^\frac{1}{n+1}|b|^{\frac{1}{n+1}}(t-T) = O(T_0|T|^{-1}).
\end{align}
Then $\psi_T(t)-\psi_T(T) =o(1)$ as $T\rightarrow -\infty.$ When $s$ is small, 
\begin{align}
(e^s -1)^{-\frac{1}{n+1}} = s^{-\frac{1}{n+1}}  +O(s^{\frac{n}{n+1}}).
\end{align}
We have that
\begin{equation}\label{psiDiff}
\frac{n+1}{n}(\psi_T(t)-\psi_T(T))^\frac{n}{n+1}(1 +O(\psi_T(t)-\psi_T(T))) = (n+1)^\frac{1}{n+1}|b|^{\frac{1}{n+1}}(t-T).
\end{equation}
In sum, for $t-T \in [T_0, 2T_0]$, 
\begin{align}\label{eq-Psi-t-T-T0}
\psi_T(t)-\psi_T(T) = \mathfrak{c}|b|^{\frac{1}{n}}\psi_\mathcal{C} + O(|b|^\frac{2}{n}T_0^\frac{2(n+1)}{n}),
\end{align}
which shows that
\begin{align}
    E(t) = O(|b|^\frac{2}{n}T_0^\frac{2(n+1)}{n}).
\end{align}
We can actually expand $E(t)$ in terms of powers of $|b|^\frac{1}{n}\psi_{\mathcal{C}}$ up to any finite order.
For example,
\begin{align}\label{eq:i_need_this_too}
E(t) = \frac{\mathfrak{c}^2}{4n+2} |b|^\frac{2}{n} \psi_{\mathcal{C}}^2 + O(|b|^\frac{3}{n}\psi_{\mathcal{C}}^3)\;\,\text{as}\;\,|b|^{\frac{1}{n}}\psi_{\mathcal{C}} \to 0.
\end{align}
More precisely, we have the following statement.

\begin{lemma}\label{ana}
 $E(t)$ is an analytic function of $|b|^\frac{1}{n} (t-T)^\frac{n+1}{n}$ at $|b|^\frac{1}{n} (t-T)^\frac{n+1}{n}=0.$   
\end{lemma}
\begin{proof}
Applying the Maclaurin series of $e^s$,   \begin{align}
    (e^s -1)^{-\frac{1}{n+1}}= s^{-\frac{1}{n+1}} \left(\sum_{i=0}^\infty \frac{s^i}{(i+1)!}\right)^{-\frac{1}{n+1}}=  s^{-\frac{1}{n+1}} \sum_{i=0}^\infty a_i s^i,
\end{align}
where $a_0=1$ and $C^i a_i \rightarrow 0$ for some $C>0$. Indeed, $\left(\frac{e^s-1}{s}\right)^{-\frac{1}{n+1}}$ is analytic at $s=0.$
By \eqref{eq-psi-t-T},
\begin{align}
\sum_{i=0}^\infty \tilde{a}_i (\psi_T(t)-\psi_T(T))^{i+\frac{n}{n+1}} = (n+1)^\frac{1}{n+1}|b|^\frac{1}{n+1}(t-T),
\end{align}
where again $\tilde{a}_0>0$ and $C^i \tilde{a}_i \rightarrow 0$ for some $C>0$. Taking the $\frac{n+1}{n}$-th power of both sides, we get
\begin{align}
(\psi_T(t) - \psi_T(T))\left(\sum_{i=0}^\infty \tilde{a}_i (\psi_T(t)-\psi_T(T))^{i}\right)^\frac{n+1}{n} = (n+1)^\frac{1}{n}|b|^\frac{1}{n}(t-T)^\frac{n+1}{n}.
\end{align}
The left-hand side is obviously analytic in $\psi_T(t) - \psi_T(T)$ around $\psi_T(t) - \psi_T(T) =0$. 
Then the lemma follows from the Lagrange inversion theorem.
\end{proof}

Lastly, we estimate the derivatives of $\ddbar E$ under the scaled Calabi metric $|b|^{\frac{1}{n}}\omega_{\mathcal{C}}$. For this we need one more lemma establishing quasi-coordinates for this metric. Following our work after Lemma \ref{lem-cusp}, we define real coordinates on the universal cover of the annulus $\{y := t-T \in [T_0,2T_0]\}$ via
\begin{align}\label{eq-TY-coor}
 (\check x_\alpha, \check y_\alpha, \check y, \check \theta) := (T_0^{\frac{1}{2n}}x_\alpha, T_0^{\frac{1}{2n}} y_\alpha, T_0^{\frac{1-n}{2n}}    y, T_0^{\frac{1-n}{2n}}    \theta ).
 \end{align}
We also introduce the corresponding holomorphic coordinates, with $w = \log z_n$:
\begin{align}\label{eq-TY-coor-2}
(\check z_\alpha, \check w) :=  (T_0^{\frac{1}{2n}}  z_\alpha, T_0^{\frac{1-n}{2n}}    w).
\end{align}

\begin{lemma}\label{lem-Ca-coor}
The un-scaled Calabi metric $\omega_{\mathcal{C}}$ is uniformly equivalent to $(\delta_{jk})$ under \eqref{eq-TY-coor}--\eqref{eq-TY-coor-2}. For all $k \geq 1$ its entries under \eqref{eq-TY-coor}--\eqref{eq-TY-coor-2} satisfy a uniform $C^k$ bound independent of $T,T_0$. 
\end{lemma}

\begin{proof}
Because $w =\log z_n$, we locally have that
\begin{align}\label{eq-t-w}
t =\log h = -\varphi + w + \overline{w}.    
\end{align}
This implies that 
\begin{align}
\begin{split}
 \omega_{\mathcal{C}} = \frac{n}{n+1} i\partial\overline{\partial} (t- T)^\frac{n+1}{n} &= \frac{1}{n} (t-T)^{\frac{1-n}{n}} i\partial t \wedge \overline{\partial} t +(t-T)^\frac{1}{n}i\partial\overline{\partial} t\\
 &=  \frac{1}{n} (t-T)^{\frac{1-n}{n}} i\partial (-\varphi +w) \wedge \overline{\partial} (-\varphi +\overline{w}) -(t-T)^\frac{1}{n}i\partial\overline{\partial} \varphi,
 \end{split}
\end{align}
where $\varphi$ is a quadratic polynomial in $z_\alpha, \overline{z}_\alpha$ for $\alpha=1,\ldots, n-1$. 
Under the coordinates
$(\check z_\alpha, \check w)$ given in \eqref{eq-TY-coor-2}, it follows that
\begin{align}\begin{split}\label{eq:formula666}
 \omega_{\mathcal{C}}
 &=  \frac{1}{n} ((t-T)T_0^{-1})^{\frac{1-n}{n}} i(-T_0^{-\frac{1}{2}}\varphi_{z_\alpha} d\check z_\alpha +d \check w) \wedge
(-T_0^{-\frac{1}{2}}\varphi_{\overline{z}_\alpha}d \overline{{\check z}}_{\alpha}+d \overline {\check w})\\&  
-((t-T)T_0^{-1})^\frac{1}{n}\varphi_{z_\alpha\overline{z}_\beta}id\check z_\alpha \wedge d\overline{\check z}_{\beta}.
\end{split}
\end{align}
This is uniformly equivalent to the Euclidean metric in the chart \eqref{eq-TY-coor-2} for $y = t-T \in [T_0,2T_0]$. To check the desired $C^k$ bound for $k = 1$, note that $\varphi_{z_\alpha\overline{z}_\beta}$ is a constant, that $\varphi_{\overline{z}_\alpha}$ satisfies
\begin{align}
\frac{\partial\varphi_{\overline{z}_\alpha}}{\partial\check z_\beta} = T_0^{-\frac{1}{2n}}\frac{\partial\varphi_{\overline{z}_\alpha}}{\partial z_\beta} =T_0^{-\frac{1}{2n}}\varphi_{z_\beta\overline{z}_\alpha},
\end{align}
and that, by \eqref{eq-t-w},
\begin{align}
\frac{\partial ((t-T)T_0^{-1})}{\partial \check w} = T_0^\frac{n-1}{2n} \frac{\partial ( (t-T)T_0^{-1})}{\partial w} =T_0 ^{-\frac{n+1}{2n}}.
\end{align}
Thus, the first coordinate derivatives of $\omega_{\mathcal{C}}$ are uniformly decaying. This pattern persists for all $k$.
\end{proof}

\begin{proposition}\label{prop:estimate_of_E}
     For $t\in [T+T_0, T+ 2T_0]$ and $j=1,2,$
\begin{equation}\label{TY-neckError}
\left| \nabla^j_{|b|^{\frac{1}{n}}\omega_{\cC}} (i\partial   \overline{\partial}
 E)\right|^2_{|b|^{\frac{1}{n}}\omega_{\cC}} \leq  C|b|^{\frac{2-j}{n}}T_0^{\frac{(2-j)(n+1)}{n}}.
\end{equation} 
\end{proposition}

\begin{proof}
Define $y := t-T$. By deriving the Taylor expansion of $\psi_T(t) - \psi_T(T)$ with respect to $|b|^\frac{1}{n}y^\frac{n+1}{n}$ using \eqref{psiDiff},
we obtain that for $y \in [T_0, 2T_0]$ and $j\geq 0$,
\begin{align}\label{eq-E-j}
|(y\partial_t)^j E(t)|\leq C_j |b|^\frac{2}{n}T_0^\frac{2(n+1)}{n},
\end{align}
where the constants $C_j$ are independent of $b$.

Recall the real chart $(\check x_\alpha, \check y_\alpha, \check y, \check \theta)$ and the holomorphic chart $(\check z_\alpha, \check w)$ on the universal cover of the annulus $\{y \in [T_0, 2T_0]\}$ defined in \eqref{eq-TY-coor}--\eqref{eq-TY-coor-2} above. By Lemma \ref{lem-Ca-coor}, these are quasi-coordinates for $\omega_{\mathcal{C}}$, i.e., the pullback of $\omega_{\mathcal{C}}$ to the universal cover is uniformly smoothly comparable to the Euclidean metric in these coordinates and its $C^{k}$ norms are bounded independently of $T_0$.

Now we are ready to estimate the derivatives of $i\partial   \overline{\partial} E$. In fact, 
 \begin{align}
  \begin{split}
 i\partial   \overline{\partial} E = E' i\partial \overline{\partial} t +E'' i\partial t \wedge \overline{\partial} t =-E' i\partial \overline{\partial} \varphi + E''i (\partial\varphi - \partial w) \wedge (\overline{\partial}\varphi - \overline{\partial} \overline{w}),
  \end{split}
 \end{align}
 where 
 \begin{align}
 T_0^\frac{1}{n}i\partial \overline{\partial} \varphi, \quad T_0^\frac{1-n}{n}i(\partial\varphi - \partial w)\wedge (\overline{\partial}\varphi - \overline{\partial} \overline{w})
 \end{align}
are uniformly smoothly bounded under \eqref{eq-TY-coor-2}. So we need to check the regularity of 
\begin{align}\label{eq-f''-f'}
E'/T_0^\frac{1}{n}, \quad  E''/T_0^\frac{1-n}{n}.
\end{align}
Because these are radial functions, we just need to check the derivatives of \eqref{eq-f''-f'} with respect to $\check y$. Applying \eqref{eq-E-j}, we obtain that
\begin{align}
\left|\frac{\partial (E'/T_0^\frac{1}{n})}{\partial \check y}\right|+\left|
\frac{\partial (E''/T_0^\frac{1-n}{n})}{\partial \check y}\right|\leq C|b|^\frac{2}{n}T_0^\frac{n+1}{2n},\\
\left|\frac{\partial^2 (E'/T_0^\frac{1}{n})}{\partial \check y^2}\right| +
\left| \frac{\partial^2 (E''/T_0^\frac{1-n}{n})}{\partial \check y^2}\right| \leq C|b|^\frac{2}{n}.
\end{align}
We conclude that the 2-form $i\partial   \overline{\partial}
 E$ satisfies  for $j=1,2$ that
\begin{align}\label{eq-E-TY}
\left| \nabla^j_{\omega_{\cC}} (i\partial   \overline{\partial}
 E)\right|_{\omega_{\cC}} \leq  C|b|^\frac{2}{n}T_0^{\frac{n+1}{2n}(2-j)}.
\end{align} 
 Now the metric $|b|^\frac{1}{n}\omega_{\cC}$ in the statement of the proposition is simply a rescaling of $\omega_{\cC}$.  If $\tilde \omega = A \omega$ for some constant $A > 0$, then for any 2-form $\mathbf{T}$,
 \begin{align}\label{eq-T-scale}
 |\nabla^j_{\tilde \omega} \mathbf{T} |^2_{\tilde \omega} = A^{-2-j} |\nabla^j_{ \omega} \mathbf{T} |^2_\omega.
 \end{align}
The proposition follows from \eqref{eq-E-TY}--\eqref{eq-T-scale}.
\end{proof}

\section{The glued approximate Kähler-Einstein metric }\label{sec:glued_approx_KE}
         
\subsection{Setting up the glued metric and the Monge-Amp\`ere equation}

 Before defining the glued metric, we briefly review some material from previous sections and fix some parameters.
 
 First, we have a family of Tian-Yau spaces $TY_\sigma$ as our singularity models, see \eqref{TYspace}. We have smooth embeddings $\Phi_\sigma: TY_0\setminus\overline{B}_{|\sigma|^{1/3}R} \to TY_\sigma$ onto a neighborhood of infinity in $TY_\sigma$ for $\sigma \neq 0$, see \eqref{eq:review_maps}. In Section \ref{TYConstruction} we review the Tian-Yau construction $(\psi_{TY_1},\omega_{TY_1},\Omega_{TY_1})$ on $TY_1$ and its decay towards the Calabi model data $(\psi_{\mathcal{C}},\omega_\cC,\Omega_\cC)$ on $TY_0$ via $\Phi_1$. This gives rise to a family $(\psi_{TY_\sigma},\omega_{TY_\sigma},\Omega_{TY_\sigma})$ via the biholomorphism $m_\sigma: TY_\sigma \to TY_1$ of \eqref{eq:review_maps}. A tricky point here is that, in the Tian-Yau construction, there is a unique choice of a Hermitian metric $h$ on the line bundle such that the Calabi model potential $\psi_{\mathcal{C}}$ defined using $h$ (see \eqref{calabiansatz}) differs from $\psi_{TY_1}$ by an exponentially decaying term. We shall fix this choice of $h$, i.e., no rescalings of $h$ are allowed. Another point that we hope will aid clarity is that we are not allowing any rescalings of $\Omega_{\mathcal{C}}$. This also fixes the scale of $\omega_{\mathcal{C}}$ because $\omega^2_\cC=\Omega_\cC\wedge\overline{\Omega}_\cC$.
 
 Secondly,  we have a family of canonically polarized surfaces $\mathcal  X_\sigma \subset \mathbb{CP}^3$ and hyperplane sections $\mathcal{D}_\sigma$ with $[2\mathcal{D}_\sigma] = K_{\mathcal  X_\sigma}$, see \eqref{eq:sextics} and \eqref{eq:sextics_affine}. We have a family of algebraic holomorphic volume forms $\Omega_{\sigma}$ on the affine surfaces $\cY_\sigma=\mathcal X_\sigma\setminus \mathcal{D}_\sigma$, which are unique up to scaling (Remark \ref{nonVHVunique}). On the regular part of $ \cY_0$, i.e., on the complement of the origin, we have a K\"ahler-Einstein metric $\omega_{KE,0}=\ddbar\psi_{KE,0}$ with $e^{-\psi_{KE,0}}\omega^2_{KE,0}=\Omega_0\wedge\overline{\Omega}_0$ (Lemma \ref{NormalizeKE}). We remark that the additive normalization of $\psi_{KE,0}$ depends on the choice of a scale of $\Omega_0$, which we will fix in a moment. 
 
 Third, we fix a local holomorphic identification $\Psi_\sigma$ of $TY_\sigma$ and $\mathcal X_\sigma$, see \eqref{Psi}. By Lemma \ref{OmegaOC},  by a suitable (unique) rescaling of $\Omega_0$ we can achieve that $ (\Psi_0^{-1})^*(\Omega_0\wedge\overline{\Omega}_0)=(1+O(e^{\frac{t}{2}}))({\Omega_{\mathcal C}}\wedge\overline{\Omega}_{\mathcal C})$ as $t = \log h \to -\infty$. We remark that there is no scaling ambiguity for $\Omega_\sigma$ any more since $\Omega_0$ is fixed. To get rid of the term $-3 \log (1-\mathfrak{s}/t)$ in
 the expansion of $(\Psi_0^{-1})^*\psi_{KE,0}-\psi_{cusp}$, we have to compose $\Psi_\sigma$ with $scale_{e^{-{\mathfrak{s}}}}$, but it is harmless for our purposes to assume directly that $\mathfrak{s} = 0$ (see Remark \ref{scalelambda}). Thus we achieve that $(\Psi_0^{-1})^*\psi_{KE,0} - \psi_{cusp} = O(e^{-\delta_0\sqrt{-t}})$ as $t \to -\infty$ for some $\delta_0>0$. Here, $\psi_{cusp}$ has already been normalized by adding a constant such that
    $e^{-\psi_{cusp}}\omega_{cusp}^2=\Omega_{\mathcal C}\wedge\overline{\Omega}_{\mathcal C}$.
    The same relation holds for the horn metrics $(\psi_T,\omega_T)$ of Section \ref{ss:new_neck} instead of $(\psi_{cusp},\omega_{cusp})$.

We review the relations between different parameters. The position of the horn is fixed via
\begin{align}
    T := \frac{2}{3}\log|\sigma|.
\end{align}
This is motivated by the discussion in Section \ref{sec:outline}. By \eqref{bandT}, the parameter $b<0$ satisfies that $|b|$ is uniformly comparable to $|T|^{-3}$. Our preferred radius coordinate $t = \log h$ differs from $\log |z|^2$ by a uniformly bounded function, see again the discussion before \eqref{eq:naive_gluing}. Here, $|z|$ is the standard radius in $\mathbb{C}^3$, where $TY_\sigma$ and $\mathcal{Y}_\sigma$ are embedded by definition. Thus, $|\sigma| |z|^{-3}$ is comparable to $e^{-(3/2)(t-T)}$.

The orange gluing region, where the end of the Tian-Yau space is attached to the left end of the new neck, is parametrized by $t \in [T+T_0,T+2T_0]$ (see Figure \ref{fig:cusp_horn} and Remark \ref{rem:colored_regions}). Here, as in \eqref{eq:def_T0},
\begin{align}
    T_0 := |T|^\alpha\;\,\text{for some fixed}\;\,\alpha \in (0,1).
\end{align}
The limit $\alpha \to 1$ corresponds to the naive gluing of the cusp metric and the Tian-Yau metric described after \eqref{eq:naive_gluing}, whereas in this paper we will always fix $\alpha$ to be arbitrarily close to zero. It is also worth noting that we do not glue with the Tian-Yau metric normalized exactly as above but, rather, with a scaled copy of it, where the scaling factor $\mathfrak{c}$ equals $\sqrt{2}$ for $n = 2$; see \eqref{cccccc}. We could have hidden this factor in our normalization of $\Omega_{\mathcal{C}}$, but we chose not to do so because $\mathfrak{c}$ appears much later in the paper than $\Omega_{\mathcal{C}}$, and a lot of other choices depend on the initial choice of $\Omega_{\mathcal{C}}$.

Lastly, the green gluing region between the neck and the cusp is parametrized by $t \in [2\tau,\tau]$ (see Remark \ref{rem:colored_regions}). The only requirement so far was that $\tau/T \to 0$ as $\sigma \to 0$. We now fix $\tau$ such that
\begin{align}\label{eq-fix-tau}
    |b||\tau|^3 = e^{-\delta_0\sqrt{-\tau}},
\end{align}
hence in particular $\tau \sim -((3/\delta_0)\log |T|)^2$. This is the classical choice of making the two gluing errors on the left side and on the right side of the gluing region comparable to each other (\eqref{eq:f_sigma_estimate}, line 3).  

\begin{table}[!ht]
\caption{Properties of $\omega_{glue,\sigma}$. The estimates are sharp up to constant factors.}
\label{tab:metric}\begin{tabular}{|l|l|l|l|l|}
\hline
\diagbox{Property}{Region}
& \parbox{32mm}{$\mathfrak{R}_1$\\compact part of $\mathcal{X}_\sigma$} & \parbox{30mm}{$\mathfrak{R}_2$\\hyperbolic cusp} & \parbox{23mm}{$\mathfrak{R}_3$\\green gluing}
& \parbox{31mm}{$\mathfrak{R}_4$\\middle neck}\\ \cline{1-1} \cline{3-5}
Range of $t = \log h$ & & $\tau<t<-N$&$2\tau<t<\tau$&$T+2T_0<t<2\tau$\\\hline
Diameter &$ 1 $&$ \log\log|T|$&{$1$}&$\log|T|$\\\hline
Curvature &$ 1 $&$ 1$&$1$&$|b|^{-\frac{1}{2}}(t-T)^{-\frac{3}{2}}$\\\hline\hline
\diagbox{Property}{Region}
&\parbox{32mm}{$\mathfrak{R}_5$\\orange gluing}&\parbox{30mm}{$\mathfrak{R}_6$\\Tian-Yau end}&\multicolumn{1}{l||}{\parbox{23mm}{$\mathfrak{R}_7$\\Tian-Yau cap}}& \multirow{4}{*}{\parbox{31mm}{ $|T|\sim |{\log |\sigma|}|$ \\  $T_0 = |T|^\alpha$ ($\alpha \ll 1$) \\ $|b| \sim |T|^{-3}$ \\ $\tau \sim (\log |T|)^2$  \\ $N,R$ $const$ $\gg 1$}} \\\cline{1-3}
Range of $t = \log h$ &$ T_0 <t-T<2T_0$&$ \log R < t-T<T_0$& \multicolumn{1}{l||}{} & \\\cline{1-4}
Diameter &$|T|^{-\frac{3}{4}(1-\alpha)}$&$|T|^{-\frac{3}{4}(1-\alpha)}$& \multicolumn{1}{l||}{$|T|^{-\frac{3}{4}}$} & \\\cline{1-4}
Curvature &$|b|^{-\frac{1}{2}}(t-T)^{-\frac{3}{2}}$&$|b|^{-\frac{1}{2}}(t-T)^{-\frac{3}{2}}$& \multicolumn{1}{l||}{$|b|^{-\frac{1}{2}}$} & \\\hline
\end{tabular}
\end{table}

\begin{figure}[!ht]
\caption{The seven regions of $\mathcal{X}_\sigma$. For the middle neck $\mathfrak{R}_4$ see also Figure \ref{fig:cusp_horn}.}\label{fig:gluing_regions}
\begin{center}
\includegraphics[width=150mm]{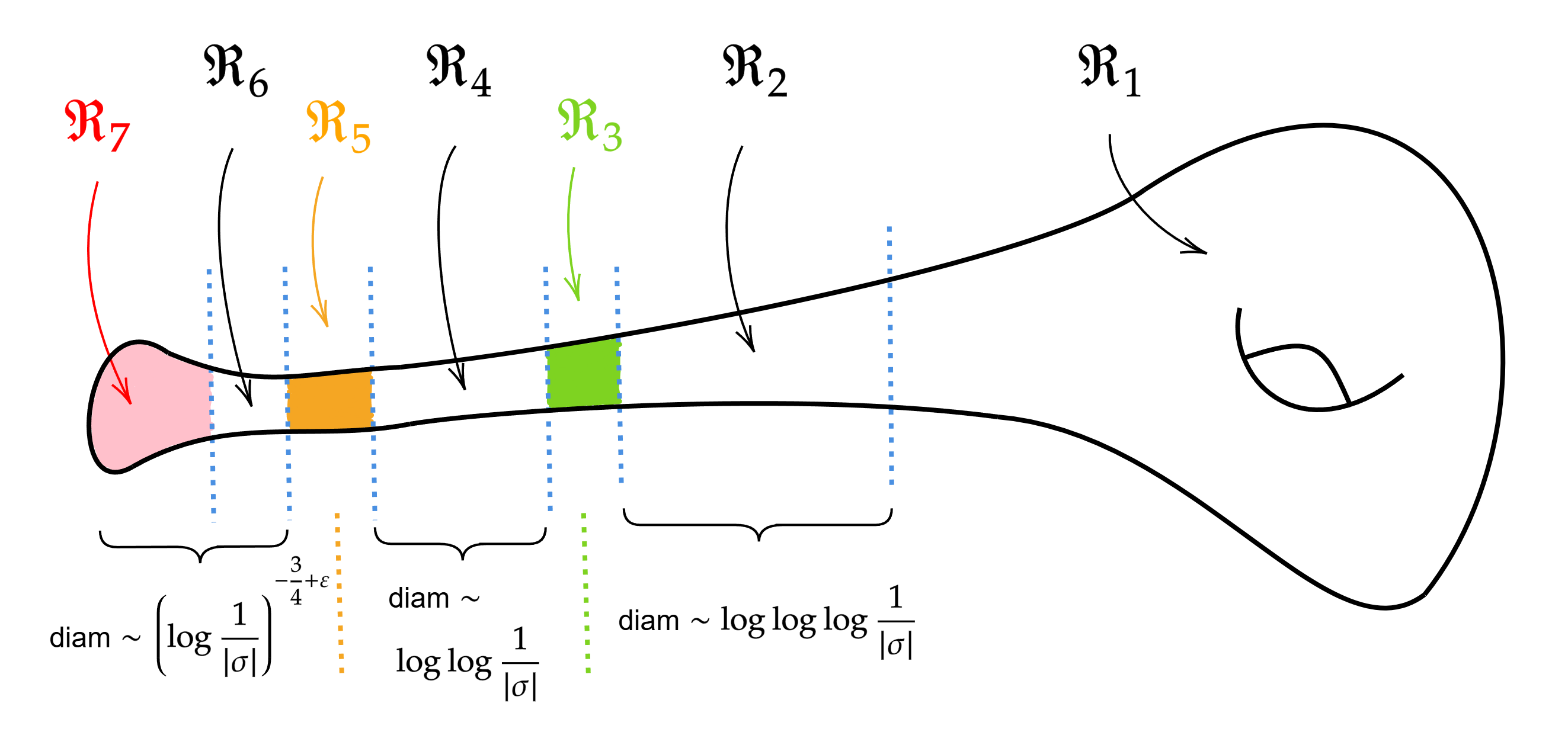}
\end{center}
\vspace{-2mm}
\end{figure}

\begin{definition}
We now define our glued approximate K\"ahler-Einstein metric $\omega_{glue,\sigma}$ on $\mathcal{X}_\sigma$. On the affine surface $\mathcal{Y}_\sigma = \mathcal{X}_\sigma \setminus \mathcal{D}_\sigma$ we set $\omega_{glue,\sigma} := \ddbars \psi_{glue,\sigma}$, where $\psi_{glue,\sigma}$ is defined as follows. We need to distinguish $7$ different regions $\mathfrak{R}_1, \ldots, \mathfrak{R}_7$; see Table \ref{tab:metric} and Figure \ref{fig:gluing_regions}. We begin by setting
\begin{equation}\label{def:metric}
\psi_{glue,\sigma}:=\psi_{FS,\sigma}+(G_\sigma^{-1})^*\psi_0\;\,\text{on}\;\,\mathfrak{R}_1 := \Psi_\sigma^{-1}(\Phi_\sigma(\{t>-N\})) \cup (\mathcal{Y}_\sigma \setminus {\rm dom}\,\Psi_\sigma).
\end{equation}
Here $\psi_{FS,\sigma}$ was defined in \eqref{FS_sigma}, $\psi_0 = \psi_{KE,0} - \psi_{FS,0}$ (Lemma \ref{lem:exloglogsoln}), and $G_\sigma$ is the $C^5$ diffeomorphism from Lemma \ref{outdiff}, while $N$ is an arbitrary but fixed large positive constant such that $\Phi_\sigma(\{t < -N/2\})$ is contained in ${\rm dom}\,\Psi_\sigma^{-1}$. Clearly $\omega_{glue,\sigma}$ is then at least $C^3$ at the divisor $\mathcal{D}_\sigma$.

For the remaining $6$ regions we prefer to write down formulas for $(\Psi^{-1}_\sigma)^*\psi_{glue,\sigma}$ on $TY_\sigma$ rather than for $\psi_{glue,\sigma}$ on $\mathcal{X}_\sigma$. To this end, let $\chi_1(t)$ be smooth and increasing with $\chi_1(t)=0$ for $t\leq T+T_0$ and $\chi_1(t)=1$ for $t\geq T+2T_0$, and with $|\partial_t^j\chi_1| = O_j(T_0^{-j})$ for all $j \geq 0$.  Similarly,
let $\chi_2(t)$ be smooth and increasing with $\chi_2(t)=0$ for $t\leq 2\tau$ and $\chi_2(t)=1$ for $t\geq \tau$, and with $|\partial_t^j\chi_2| = O_j(|\tau|^{-j})$ for all $j \geq 0$. Then the desired formulas for $(\Psi^{-1}_\sigma)^*\psi_{glue,\sigma}$ are as follows:
\begin{align}
\begin{split}
    (\Psi^{-1}_\sigma)^*\psi_{FS,\sigma}+((\Psi_\sigma\circ G_\sigma)^{-1})^*\psi_0 &\;\,\text{on}\;\,\Psi_\sigma(\mathfrak{R}_2) := \Phi_\sigma(\{\tau<t<-N\}), \\
   (\Phi_\sigma^{-1})^*\Big(\chi_2(\Psi^{-1}_0)^*(G_\sigma^*\psi_{FS,\sigma}+\psi_0)+(1-\chi_2)\psi_T\Big)&\;\,\text{on}\;\,\Psi_\sigma(\mathfrak{R}_3) := \Phi_\sigma(\{2\tau<t<\tau\}),\\
    (\Phi^{-1}_\sigma)^*\psi_T  &\;\,\text{on}\;\,\Psi_\sigma(\mathfrak{R}_4) := \Phi_\sigma(\{T+2T_0<t<2\tau\}),\\
    (\Phi_\sigma^{-1})^*\Big(\chi_1\psi_T+(1-\chi_1)\Phi_\sigma^*(\mathfrak{c}|b|^{\frac{1}{2}}\psi_{TY_\sigma}+\psi_T(T))\Big)&\;\,\text{on}\;\,\Psi_\sigma(\mathfrak{R}_5) := \Phi_\sigma(\{T_0<t-T<2T_0\}),\\
    \mathfrak{c}|b|^{\frac{1}{2}}\psi_{TY_\sigma}+\psi_T(T)&\;\,\text{on}\;\,\Psi_\sigma(\mathfrak{R}_6) := \Phi_\sigma(\{t < T+ T_0\}),\\
    \mathfrak{c}|b|^{\frac{1}{2}}\psi_{TY_\sigma}+\psi_T(T)&\;\,\text{on}\;\,\Psi_\sigma(\mathfrak{R}_7) := ({\rm dom}\,\Psi_\sigma^{-1})\setminus ({\rm im}\,\Phi_\sigma).
\end{split}
\end{align}
This concludes our definition of the glued approximate Kähler-Einstein manifold $(\mathcal{X}_\sigma,\omega_{glue,\sigma})$.
\end{definition}

The Monge-Ampère equation that we want to solve is 
\begin{equation}\label{mainMA}
(\omega_{glue,\sigma}+\ddbar u_\sigma)^2=e^{u_\sigma-f_\sigma}\omega_{glue,\sigma}^2,
\end{equation}
where $f_\sigma$ is the Ricci potential, defined up to a constant by the condition that
\begin{align}\label{eq:def_ric_pot}
{\rm Ric}(\omega_{glue,
\sigma})+\omega_{glue,\sigma}+\ddbar f_\sigma=0.
\end{align}

\begin{lemma}
Up to an arbitrary constant,
\begin{equation}\label{riccipotential}
    f_\sigma|_{\mathcal{Y}_\sigma} = \log\left(\frac{\omega_{glue,\sigma}^2}{\Omega_\sigma\wedge\overline{\Omega}_\sigma}\right)-\psi_{glue,\sigma}.
\end{equation} 
\end{lemma}

\begin{proof}
	Denote the right-hand side of \eqref{riccipotential} by $g_\sigma$.
    
    \begin{claim}\label{claim:yetanotherone}
        $g_\sigma$ extends at least $C^3$ to $\mathcal{X}_\sigma$.
    \end{claim} 

    \begin{proof}[Proof of Claim \ref{claim:yetanotherone}]
    We know that $\omega_{glue,\sigma}^2$ is at least $C^3$ on $\mathcal{X}_\sigma$ and $\Omega_\sigma\wedge\overline{\Omega}_\sigma = |H_{\mathcal{D}_\sigma}|^{4} \cdot (\textit{smooth on $\mathcal{X}_\sigma$})$, where $H_{\mathcal{D}_\sigma}$ is a defining section of the line bundle $[\mathcal{D}_\sigma]$ and $|\cdot|$ is any smooth Hermitian metric on this line bundle. Thus, the log volume ratio in \eqref{riccipotential} is of the form $\log |H_{\mathcal{D}_\sigma}|^{-4} + (\textit{at least $C^3$ on $\mathcal{X}_\sigma$})$. Finally, since the reference K\"ahler form $\omega_{FS,\sigma}$ represents the Poincaré dual of the divisor class $[2\mathcal{D}_\sigma]$, we have that  $\psi_{glue,\sigma}= \psi_{FS,\sigma} + (\textit{at least $C^5$ on $\mathcal{X}_\sigma$}) = \log|H_{\mathcal{D}_\sigma}|^{-4} + (\textit{at least $C^5$ on $\mathcal{X}_\sigma$})$. Thus, the two log terms in the definition of $g_\sigma$ cancel each other out and the remainder is at least $C^3$.
    \end{proof}
    
    From its definition, \eqref{eq:def_ric_pot}, and by elliptic regularity, $f_\sigma$ is at least $C^{2,\alpha}$ on $\mathcal{X}_\sigma$. Hence, by Claim \ref{claim:yetanotherone}, so is $f_\sigma - g_\sigma$. From the definition of $g_\sigma$ and from the standard formula for the Ricci curvature of a Kähler metric, $g_\sigma$ satisfies ${\rm Ric}(\omega_{glue,\sigma})+\omega_{glue,\sigma}+\ddbar g_\sigma=0$ (on $\mathcal{Y}_\sigma$ and thus, by Claim \ref{claim:yetanotherone}, on $\mathcal{X}_\sigma$). Thus, $f_\sigma - g_\sigma$ is pluriharmonic on $\mathcal{X}_\sigma$ and hence constant.
	\end{proof}

 Of course, we already know that, given $f_\sigma$, the Monge-Amp\`ere equation \eqref{mainMA} has a unique solution $u_\sigma$ by the Aubin-Yau theorem \cite{Aubin,Yau}, and if we change $f_\sigma$ by a constant then $u_\sigma$ changes by the same constant. Our main goal in this paper is to prove that $u_\sigma$ is actually small modulo constants, meaning at the very least that $\sup_{\mathcal{X}_\sigma} |\ddbar u_\sigma|_{\omega_{glue,\sigma}} \to 0$ as $\sigma \to 0$. Clearly, the first step here is to prove that $f_\sigma$ is sufficiently small modulo constants, using the expression in \eqref{riccipotential}. This estimate is the main result of Section \ref{sec:glued_approx_KE} and we record it in Theorem \ref{f:Riccipotential}. To state the theorem, we need one other definition.

 \begin{definition}\label{def:reg_scale}
The regularity scale function $\mathbf{r}_\sigma: \mathcal{X}_\sigma \to \mathbb{R}^+$ of $\omega_{glue,\sigma}$ is defined as follows:
\begin{align}\label{eq:reg_scale}
\mathbf{r}_\sigma := 
\begin{cases}
1 &\text{on}\;\,\Psi_\sigma^{-1}(\Phi_\sigma(\{t > T/2\})) \cup (\mathcal{X}_\sigma \setminus {\rm dom}\,\Psi_\sigma),\\
\Psi_\sigma^*(\Phi_\sigma^{-1})^*\left(|b|^{\frac{1}{4}}(t-T)^{\frac{3}{4}}\right)&\text{on}\;\,\Psi_\sigma^{-1}(\Phi_\sigma(\{t < T/2\})),\\
|b|^{\frac{1}{4}} &\text{on}\;\,\mathfrak{R}_7.
\end{cases}
\end{align}
Up to bounded factors that we suppress, $\mathbf{r}_\sigma(p)$ is the maximal radius of an $\omega_{glue,\sigma}$-ball $B_r(p)$ such that $r^{-2} \omega_{glue,\sigma}$ is $C^\infty$ bounded in coordinates on the universal cover of $B_r(p)$. In particular, the curvature of $\omega_{glue, \sigma}$ is uniformly $O(\mathbf{r}_\sigma^{-2})$. This is the curvature estimate recorded in Table \ref{tab:metric} and it is sharp.
\end{definition}

 We are now able to state our main result in this section. The proof is deferred to Section \ref{sec:proof_ric_pot}.
 
 \begin{theorem}\label{f:Riccipotential}
Let $f_\sigma$ be the Ricci potential defined in \eqref{riccipotential}. Then for all $\varepsilon>0$ the function
\begin{equation}
    |f_\sigma|+\mathbf{r}_\sigma |\nabla_{\omega_{glue,\sigma}} f_\sigma|_{\omega_{glue,\sigma}}
\end{equation}
satisfies the following pointwise estimates as $\sigma \to 0$:
     \begin{align}\label{eq:f_sigma_estimate}
\begin{cases}
   O(|\sigma|) &\textnormal{on $\mathfrak{R}_1$},\\
O(e^{-(\frac{1}{2}-\epsilon)(t-T)}) &\textnormal{on $\mathfrak{R}_2$},\\
    O(|b||\tau|^3) + O(e^{-\delta_0\sqrt{-t}})&\textnormal{on $\mathfrak{R}_3$},\\
    O(e^{-(\frac{1}{2}-\epsilon)(t-T)})+O(e^{(\frac{1}{2}-\epsilon)t})  &\textnormal{on $\mathfrak{R}_4$},\\
   O((T_0/|T|)^{\frac{3}{2}})&\textnormal{on $\mathfrak{R}_5 \cup \mathfrak{R}_6 \cup \mathfrak{R}_7$}.
    \end{cases}
    \end{align}
Here we abuse notation by writing $t$ instead of the correct $\Psi_\sigma^*(\Phi_\sigma^{-1})^*t$.
\end{theorem}

\begin{remark}
    The shape of \eqref{eq:f_sigma_estimate} makes sense intuitively: $f_\sigma$ decays faster than any polynomial in $|T|$ on $\mathfrak{R}_1 \cup \mathfrak{R}_2 \cup \mathfrak{R}_4$ because $\omega_{glue,\sigma}$ is almost an exact solution of the negative Kähler-Einstein equation in these regions. On the other hand, on $\mathfrak{R}_5 \cup \mathfrak{R}_6 \cup \mathfrak{R}_7$ we are gluing with the scaled Tian-Yau metric, which is Ricci-flat, so up to some additive constant $f_\sigma$ then equals minus the Kähler potential of this metric, whose oscillation is $\sim$ $(T_0/|T|)^{3/2}$. The error in the gluing region $\mathfrak{R}_3$ is also polynomial in $|T|$ but it is $O(|T|^{-3}|{\log |T|}|^6)$, i.e., almost quadratic compared to the error in the Tian-Yau region. 
\end{remark}

\begin{remark}\label{rem:savin}
    Using only \eqref{eq:f_sigma_estimate} and standard facts from the theory of the complex Monge-Ampère equation, one can already deduce quite a bit of information about the solution $u_\sigma$. For example, the maximum principle applied to \eqref{mainMA} immediately yields that
\begin{align}\label{eq-insuf}
   \sup\nolimits_{\mathcal{X}_\sigma} |u_\sigma| \leq \sup\nolimits_{\mathcal{X}_\sigma} |f_\sigma| = O(|b|^\frac{1-\alpha}{2}).
\end{align}
On $\mathfrak{R}_1 \cup \mathfrak{R}_2 \cup \mathfrak{R}_3$ and on a large portion of the middle neck $\mathfrak{R}_4$, the regularity scale $\mathbf{r}_\sigma$ is uniformly bounded below, i.e., we have uniform $C^\infty$ bounds for $\omega_{glue,\sigma}$ on the universal cover of a geodesic ball of definite size centered at any point. Thus, on all of these regions, $|\ddbar u_\sigma|_{\omega_{glue,\sigma}} = O(|b|^{(1-\alpha)/2})$ by combining \eqref{eq-insuf} and Savin's small perturbation theorem \cite[Thm 1.3]{Savin}. This estimate already implies $C^\infty_{loc}$ convergence of $\omega_{KE,\sigma}$ to $\omega_{KE,0}$ away from any fixed neighborhood of the origin in $\mathbb{C}^3$.

On the other hand, moving towards the left boundary of the middle neck $\mathfrak{R}_4$ and into $\mathfrak{R}_5 \cup \mathfrak{R}_6 \cup \mathfrak{R}_7$, the regularity scale $\mathbf{r}_\sigma$ decays until it reaches its minimum of $|b|^{1/4}$ on the Tian-Yau cap $\mathfrak{R}_7$. To be able to apply Savin's theorem in this situation we would need that $|u_\sigma| \ll \mathbf{r}_\sigma^2$, i.e., $|u_\sigma| \ll |b|^{1/2}$ on $\mathfrak{R}_7$. This is obviously out of reach of \eqref{eq-insuf} no matter how small we make $\alpha$.
\end{remark}

The point of the weighted Hölder space theory developed in the rest of the paper (after the proof of Theorem \ref{f:Riccipotential}) is precisely to improve the naive $C^0$ estimate \eqref{eq-insuf} in the Tian-Yau region. We will for instance be able to prove that $\sup_{\mathfrak{R}_7} |u_\sigma| = O_\varepsilon(|b|^{(5/6)-\varepsilon})$ for all $\varepsilon>0$. This is enough to obtain $C^{1,\beta}$ closeness of $\omega_{KE,\sigma}$ to the scaled Tian-Yau metric on $\mathfrak{R}_7$ for all $\beta < \frac{1}{3}$, which in particular proves the Main Theorem. In fact, we conjecture that a more systematic approach to the obstruction theory in Section \ref{sec:kill} would even yield $\sup_{\mathfrak{R}_7} |u_\sigma| = O_\varepsilon(|b|^{1-\varepsilon})$ and thus $C^{1,\beta}$ closeness for all $\beta < 1$. This would then be optimal because a Ricci-flat metric cannot be $C^{1,1}$ close to a metric with Ricci $=$ $-1$.

\subsection{Proof of the Ricci potential estimate}\label{sec:proof_ric_pot}

This section is dedicated to the proof of Theorem \ref{f:Riccipotential}. This will be done at the end of this section, as a consequence of a long sequence of lemmas.

We shall pull back everything back to $TY_0$ and then estimate.

\begin{lemma}\label{l:psiJerror}
On $TY_0$, the following hold.
\begin{enumerate}\item[$(1)$] If $t\in [T+T_0,\tau]$, one has
\begin{equation}
|\nabla_{\omega_T}^k(\Phi^*_\sigma J_{TY_\sigma}-J_{\cC})|_{\omega_{T}}=O(e^{-(\frac{1}{2}-\epsilon)(t-T)})\;\, \text{for all} \;\, k \geq 0, \epsilon > 0.\end{equation}
 
\item[$(2)$] If $t\in [2\tau,-N]$, one has
\begin{equation}
|\nabla_{\omega_{cusp}}^k(\Phi^*_\sigma J_{TY_\sigma}-J_{\cC})|_{\omega_{cusp}}=O(e^{-(\frac{1}{2}-\epsilon)(t-T)})\;\, \text{for all} \;\, k \geq 0, \epsilon > 0.
\end{equation}
\end{enumerate}
\end{lemma}

\begin{proof}
In Lemma \ref{TYerror} we have proved that this holds for the reference metric $|b|^{1/2}\omega_{\cC,\sigma}$ instead of $\omega_T$ or $\omega_{cusp}$, assuming only that $t-T \to \infty$. To compare these reference metrics, we now estimate
\begin{equation}\left|\nabla^k_{|b|^{\frac{1}{2}}\omega_{\mathcal{C},\sigma}}\omega_T\right|_{|b|^{\frac{1}{2}}\omega_{\cC,\sigma}} \,\,\,\textnormal{and}\,\,\,\left|\nabla^k_{|b|^{\frac{1}{2}}\omega_{\mathcal{C},\sigma}}\omega_{cusp}\right|_{|b|^{\frac{1}{2}}\omega_{\cC,\sigma}}\,\,\,\textnormal{for all}\,\,\, k \geq 0.
\end{equation}

For $\delta>0$ fixed and $t/T \geq 1-\delta$, \eqref{eq-Psi-T-11}--\eqref{eq-Psi-T-12} say that $\omega_T$ is uniformly comparable to $|b|^{1/2}\omega_{\cC,\sigma}$.  For $k>0$, first notice that by \eqref{horderpsiT}, the $t$-derivatives of $\psi_T$ blow up at worst polynomially in $|T|$. Using the quasi-coordinates \eqref{eq-TY-coor} for $\omega_{\cC,\sigma}$ and $|b|\sim |T|^{-3}$, we deduce  that the $k$-th derivative of $\omega_T$ with respect to $|b|^{1/2}\omega_{\cC,\sigma}$ is bounded by $|T|^{N_k}$ for some $N_k \in \mathbb{N}$. The exponential term $e^{-(1/2-\varepsilon)(t-T)}$ absorbs these polynomial factors because $t-T \geq T_0 = |T|^\alpha$. This proves (1) for $t/T \geq 1-\delta$.

Now assume $t/T \leq 1-\delta$ and $t \leq -N$. We first compare $\omega_T$ to $\omega_{cusp}$ in a similar fashion, using Proposition \ref{prop-T-c}. Then we compare $\omega_{cusp}$ to $|b|^{1/2}\omega_{\cC,\sigma}$ using the fact that these metrics are explicit. This yields the remaining case of (1) ($t/T \leq 1-\delta$, $t \leq \tau$) and all cases of (2) ($2\tau \leq t \leq -N$).
\end{proof}

Similar arguments yield:

\begin{lemma}\label{l:potentialest}
On $TY_0$, the following hold.
\begin{enumerate}\item[$(1)$] If $t\in [T+T_0,\tau]$,  then for all $k\geq 0$ there exists a positive integer $N_k$ such that 
\begin{equation}
|\nabla_{\omega_T}^k\psi_T|_{\omega_{T}}=O(|T|^{N_k}).
\end{equation}
 
\item[$(2)$] If $t\in [2\tau,-N]$, then for all $k\geq 0$ there exists a positive integer $N_k$ such that
\begin{equation}
|\nabla_{\omega_{cusp}}^k\psi_{cusp}|_{\omega_{cusp}}=O(|T|^{N_k}).
\end{equation}
\end{enumerate}
\end{lemma}

Now we combine Lemma \ref{l:potentialest} and Lemma \ref{l:psiJerror} to get the following lemma, which essentially measures the non-holomorphicity of $\Phi_\sigma$.

\begin{lemma}\label{Jerror} On $TY_0$, the following hold for all $k\geq 0$, $\epsilon>0$.
\begin{enumerate}
\item[$(1)$]
If $t\in [T+T_0,\tau]$, we have that
\begin{equation}
 |\nabla^k_{\omega_T}(\Phi_\sigma^*\ddbars (\Phi_\sigma^{-1})^*\psi_{T}-\omega_{T})|_{\omega_T}=O(e^{-(\frac{1}{2}-\epsilon)(t-T)}).
\end{equation}

\item[$(2)$] If $t\in[2\tau, -N]$, we have that
\begin{equation}
|\nabla^k_{\omega_{cusp}}(\Phi_\sigma^*\ddbars (\Phi_\sigma^{-1})^*\psi_{cusp}-\omega_{cusp})|_{\omega_{cusp}}=O(e^{-(\frac{1}{2}-\epsilon)(t-T)}).
\end{equation}
\end{enumerate}
\end{lemma}

\begin{proof}
Let $A_\sigma:=\Phi_\sigma^* J_\sigma-J_0$. Then for all functions $f$ on $TY_\sigma$ and Kähler metrics $\omega$ on $TY_0$,
\begin{align}\label{eq-smallness}
 \begin{split}
    \Phi_\sigma^*\ddbars f-i\partial_0\overline\partial_0\Phi_\sigma^*f&=\Phi_\sigma^*dd^c_\sigma f-dd^c_0\Phi_\sigma^* f\\
    &=d\Phi_\sigma^* J_\sigma d f-dJ_0 d \Phi_\sigma^* f\\
    &=d\Big((A_\sigma+J_0)\Phi_\sigma^*(df)\Big) -dJ_0 d \Phi_\sigma^* f\\
    &=(\nabla_{\omega} A_\sigma) \circledast (\nabla_{\omega}(\Phi_\sigma^*f)) + A_\sigma \circledast (\nabla^2_{\omega}(\Phi_\sigma^*f)).
     \end{split}
\end{align}
Here $\circledast$ denotes a tensorial contraction involving also the reference metric $\omega$. Now we apply this with $\omega = \omega_T, \omega_{cusp}$ and with $f = (\Phi_{\sigma}^{-1})^* \psi_T, (\Phi_\sigma^{-1})^*\psi_{cusp}$, respectively. By Lemma \ref{l:psiJerror}, $A_\sigma$ and its covariant derivatives are bounded by $e^{-(1/2-\epsilon)(t-T)}$. On the other hand, $\psi_{cusp},\psi_{T}$ and their covariant derivatives blow up at worst polynomially in $|T|$ by Lemma \ref{l:potentialest}. Since $t-T\geq T_0 = |T|^\alpha$, these polynomial factors are again absorbed by the exponential decay as in the proof of Lemma \ref{l:psiJerror}.
\end{proof}

The following auxiliary lemma will be used in Lemma \ref{l:FubiniVolErr}.

\begin{lemma}\label{l:Fubini}
Abuse notation by denoting $(\Psi_\sigma^{-1})^*\psi_{FS,\sigma}$ by $\psi_{FS,\sigma}$ for all sufficiently small values of $\sigma$, including for $\sigma = 0$. This is a Kähler potential on the intersection of $TY_\sigma$ with some fixed neighborhood of the origin in $\mathbb{C}^3$. Let $z\in TY_0$ with 
\begin{align}
|\sigma|^{\frac{1}{3}}R \leq |z|\leq \epsilon,
\end{align}
where $\epsilon$ is a sufficiently small constant independent of $\sigma$ and $R$ is a large constant as in Lemma $\ref{Phis}$.  Let $\omega_{flat}:=\ddbar|z|^2$ be the flat metric on $\mathbb C^3$.  Then for $k=0,1$ and for $|\sigma|^{-1/3} |z| \to \infty$,
\begin{align}\label{eq-e1}
\left|\nabla^k_{\omega_{flat}|_{TY_0}}\left(\Phi^*_\sigma\ddbars\psi_{FS,\sigma}-\ddbar\psi_{FS,0}\right)\right|_{\omega_{flat}|_{TY_0}}=O(|\sigma||z|^{-3-k}),\\
\label{eq-e2}
\left|\nabla^k_{\omega_{flat}|_{TY_0}}\ddbar \left(\Phi^*_\sigma \psi_{FS,\sigma}-\psi_{FS,0}\right)\right|_{\omega_{flat}|_{TY_0}}=O(|\sigma||z|^{-3-k}).
\end{align}
\end{lemma}

\begin{proof}
We only prove \eqref{eq-e1} as the proof of \eqref{eq-e2} is similar. To prove \eqref{eq-e1}, we start with

\begin{claim}\label{claim:lalilu}
One has
\begin{align}\label{nnnn}
\begin{split}
\Phi^*_\sigma\ddbars\psi_{FS,\sigma}-\ddbar\psi_{FS,0}=\biggl(\sum\nolimits_{|I|+|J|=2}\tilde O(|\sigma||z|^{-3})  dz_I\wedge d\overline{z}_J\biggr)\biggr|_{TY_0},\\
\end{split}
\end{align}
where $\tilde O(|\sigma||z|^{-3})$ is a function on $B_\varepsilon(0)\subset \mathbb{C}^3$ such that for all $k \geq 0$ and $|\sigma|^{-1/3}|z|\to\infty$,
\begin{equation}\label{mmmm}
\left|\partial^k\tilde O(|\sigma||z|^{-3})\right|=O(|\sigma||z|^{-3-k}).
\end{equation}
\end{claim}

\begin{proof}[Proof of Claim \ref{claim:lalilu}] The key is the following expression of $\Phi_\sigma$ from Lemma \ref{Phis}:
\begin{equation}\label{ooo}
\Phi_\sigma^*z_i=z_i+\nu_\sigma(z)\overline{z}_i^2\;\,(i = 1,2,3).
\end{equation}
Hence, dropping the restriction symbols for simplicity,
\begin{equation}\label{DDD}
    \Phi_\sigma^* (dz_i)=dz_i+2\nu_\sigma(z)\overline{z}_idz_i+\sum_{j=1}^3\frac{\partial\nu_\sigma(z)}{\partial z_j}\overline{z}_i^2dz_j+\sum_{j=1}^3\frac{\partial\nu_\sigma(z)}{\partial\overline{z}_j}\overline{z}_i^2d\overline{z}_j.
\end{equation}
From Lemma \ref{Phis} we have for all $\ell \geq 0$ and for $|\sigma|^{-1/3} |z| \to \infty$ that 
\begin{equation}\label{EEEE}
\left|\partial^\ell\nu_\sigma(z)\right|=O(|\sigma||z|^{-4-\ell}).
\end{equation}
Therefore
\begin{equation}\label{Pullbackz}
\Phi_\sigma^*(dz_i)=dz_i+\sum_{j=1}^{3}\tilde O(|\sigma||z|^{-3})dz_j+\sum_{j=1}^{3}\tilde O(|\sigma||z|^{-3})d\overline{z}_j.
\end{equation}

On $\mathbb{C}^3$, when $|z|\leq\epsilon$ for some sufficiently small constant $\epsilon$, we have that
\begin{align}\label{FFFF}
\ddbar\psi_{FS}=\sum_{i,j=1}^3 a_{i \bar\jmath}(z)dz_i\wedge d\overline{z}_j,\;\,a_{i\bar\jmath} = \delta_{i\bar\jmath} + O(|z|^2)\;\,\text{as}\;\,|z| \to 0,
\end{align}
and  $a_{i \bar\jmath}(z)$ is actually real-analytic in $z$. Then, from \eqref{ooo} and \eqref{EEEE},
\begin{equation}\label{hhhhhhh}\Phi_\sigma^* a_{i\bar \jmath}(z)-a_{i\bar \jmath}(z)=\tilde O(|\sigma||z|^{-1})\;\,
\text{as}\;\,|\sigma|^{-\frac{1}{3}}|z|\to\infty.
\end{equation}
For example, if $z_1^2$ appears in $a_{i \bar \jmath}(z)$, then the relevant difference is
\begin{align}\label{ex-z1}\begin{split}
(z_1+\nu_\sigma(z)\overline{z}_1^2)^{2}    -z_1^{2} &= 2\nu_\sigma(z) z_1 \overline{z}_1^2+\nu_\sigma(z)^2\overline{z}_1^4
\\&=O(|\sigma||z|^{-1}) +O(|\sigma|^2|z|^{-4})\;\,\text{as}\;\,|\sigma|^{-\frac{1}{3}}|z| \to \infty,
\end{split}
\end{align}
which is $O(|\sigma||z|^{-1})$ because $|\sigma||z|^{-3} =o(1)$.

Then Claim \ref{claim:lalilu} follows from \eqref{Pullbackz} and \eqref{hhhhhhh} because $|z|^{-1} \leq \epsilon^{2}|z|^{-3}$.
\end{proof}

Next, we estimate the $2$-form \eqref{nnnn} with respect to the metric $\omega_{flat}|_{TY_0}$. To do so, without loss of generality (using the symmetry of $z_1,z_2,z_3$) we may assume that $|z_3|\geq \max\{|z_1|,|z_2|\}$. In particular, we are working on the affine piece $z_3\neq 0$. From the defining equation of  $TY_0$, 
\begin{equation}\label{eq:dfgh}
dz_3=-\frac{z_1^2}{z_3^2}dz_1-\frac{z_2^2}{z_3^2}dz_2,
\end{equation} 
where we are again dropping the restriction symbols. Thus,
\begin{align}\label{ggggg}
\begin{split}
\omega_{flat}|_{TY_0}=&\left(1+\left|\frac{z_1}{z_3}\right|^4\right)dz_1\wedge d\overline{z}_1+\left(1+\left|\frac{z_2}{z_3}\right|^4\right)dz_2\wedge d\overline{z}_2\\&+\left(\frac{z_1^2}{z_3^2}\frac{\overline{z}_2^2}{\overline{z}_3^2}\right)dz_1\wedge d\overline{z}_2+\left(\frac{z_2^2}{z_3^2}\frac{\overline{z}_1^2}{\overline{z}_3^2}\right)dz_2\wedge d\overline{z}_1.
\end{split}
\end{align}
Viewing $z_1,z_2$ as local coordinates and denoting the coefficients of the K\"ahler form in \eqref{ggggg} by $g_{i\bar \jmath}$, it is easy to see that $\det(g_{i\bar\jmath})\geq 1$, so every component of the inverse matrix $g^{\bar\imath j}$ is bounded by $2$ because $|z_3| \geq \max\{|z_1|,|z_2|\}$. This fact and \eqref{nnnn} imply the desired estimate \eqref{eq-e1} for $k = 0$.

To finish the proof, we also need to estimate the first covariant derivative of \eqref{nnnn} with respect to $\omega_{flat}|_{TY_0}$. Applying $-z_3^3=z_1^3+z_2^3$ and $|z_3|
 \geq \max\{|z_1|,|z_2|\}$, we have that
 \begin{align}\label{kkk}\begin{split}\frac{\partial z_3}{\partial z_1}=-\frac{z_1^2}{z_3^2} = O(1),\;\,
 \frac{\partial z_3}{\partial z_2}=-\frac{z_2^2}{z_3^2} = O(1).
 \end{split}\end{align}
 Combining the claimed estimates \eqref{mmmm} and \eqref{kkk}, it follows from the chain rule that
\begin{align}\label{aaaaaa}
 \begin{split}\left\vert\frac{\partial\tilde O(|\sigma||z|^{-3})}{\partial z_i} \right\vert
 + \left\vert\frac{\partial\tilde O(|\sigma||z|^{-3})}{\partial \overline{z}_i} \right\vert=O(|\sigma||z|^{-4})\;\,(i=1,2).
 \end{split}\end{align}
The remaining task is to bound the first covariant derivative of $dz_1, dz_2, dz_3$ with respect to $\omega_{flat}|_{TY_0}$. Since $g^{\bar\imath j}$ is bounded, it is enough to bound the second partials of $z_1,z_2,z_3$ and the first partials of $g_{i\bar\jmath}$. This can be done by differentiating \eqref{kkk} and \eqref{ggggg}, respectively. Because of homogeneity reasons it is clear that all of these derivatives are bounded by a constant times $|z|^{-1}$.
\end{proof}

\begin{lemma}\label{l:FubiniVolErr}
We work on $TY_0$ and abuse notation by replacing
\begin{align}
\begin{split}
    (\Psi_0^{-1})^*\omega_{KE,0} \leadsto \omega_{KE,0},\;\;
    (\Psi_0^{-1})^*\psi_0 \leadsto \psi_0, \;\;
    (\Psi_\sigma^{-1})^*\psi_{FS,\sigma} \leadsto \psi_{FS,\sigma}.
\end{split}
\end{align}
Let $N>0$ be a fixed large constant.
Then for $k=0,1$ and $t\in [2\tau, -N]$ we have that
\begin{equation}\label{4444}
\left|\nabla^k_{\omega_{KE,0}}\frac{(\Phi^*_\sigma\ddbars ((\Phi^{-1}_\sigma)^*\psi_0+\psi_{FS,\sigma}))^2-\omega^2_{KE,0}}{\omega^2_{KE,0}}\right|_{\omega_{KE,0}}=O(e^{-(\frac{1}{2}-\epsilon)(t-T)}).
\end{equation}
\end{lemma}

\begin{proof}
Inserting one more term, we have
\begin{align}\label{fff}
\begin{split}
\Phi^*_\sigma(\ddbars((\Phi^{-1}_\sigma)^*\psi_0+\psi_{FS,\sigma}))^2-\omega^2_{KE,0}=(\ddbar (\psi_0+\Phi^*_\sigma\psi_{FS,\sigma}))^2-\omega^2_{KE,0}\\
+\;\Phi^*_\sigma(\ddbars((\Phi^{-1}_\sigma)^*\psi_0+\psi_{FS,\sigma}))^2-(\ddbar (\psi_0+\Phi_\sigma^*\psi_{FS,\sigma}))^2.
\end{split}
\end{align}
We begin by estimating the first line of \eqref{fff}. Canceling $\ddbar\psi_0$, one has
\begin{align}
\ddbar (\psi_0+\Phi^*_\sigma\psi_{FS,\sigma})-\omega_{KE,0}=\ddbar (\Phi_\sigma^*\psi_{FS,\sigma}-\psi_{FS,0}).
\end{align}
For 2-forms $\omega_1, \omega_2$, one has
\begin{align}\label{eq:furz}
\omega_1^2 - \omega_2^2 &=2(\omega_1 -\omega_2) \wedge \omega_2 +(\omega_1-\omega_2)^2,
\end{align}
which implies that
\begin{equation}\label{mmm}
\frac{(\ddbar (\psi_0+\Phi^*_\sigma\psi_{FS,\sigma}))^2-\omega^2_{KE,0}}{\omega^2_{KE,0}}=2\Delta_{\omega_{KE,0}}(\Phi^*_\sigma\psi_{FS,\sigma}-\psi_{FS,0})+\frac{(\ddbar(\Phi^*_\sigma\psi_{FS,\sigma}-\psi_{FS,0}))^2}{\omega^2_{KE,0}}.
\end{equation}
To estimate these terms, we replace the reference metric $\omega_{KE,0}$ by $\omega_{flat}|_{TY_0}$.

\begin{claim}\label{claim:herrdeninger}
    For any $j\geq 0$ there is a positive integer $K_j$ such that as $|z| \to 0$, 
\begin{equation}\label{hhh1}
|\nabla^j_{\omega_{KE,0}}(\omega_{flat}|_{TY_0})|_{\omega_{KE,0}} = O(|z|^2|t|^{K_j}).
\end{equation} 
Moreover, for $j = 0$ the left-hand side is actually uniformly equivalent to $|z|^2 |t|^{K_0}$.
\end{claim}

\begin{proof}[Proof of Claim \ref{claim:herrdeninger}]
Thanks to \cite[Thm 1.4]{DFS}, $\omega_{cusp}$ and $\omega_{KE,0}$ are uniformly equivalent to any order for $t\leq -N$, so it suffices to prove the estimate using $\omega_{cusp}$ as the reference metric.

On $TY_0\setminus \{0\}$ we clearly have that $|z|^2 = e^{t + \psi}$ for some smooth function $\psi$ on the elliptic curve $E$, viewed as a $0$-homogeneous function on $TY_0 \setminus \{0\}$. Thus,
\begin{align}\label{eq:shgjl}
\omega_{flat}|_{TY_0}=\ddbar e^{t+\psi}&=e^{t+\psi}\ddbar(t+\psi)+e^{t+\psi}i\partial(t+\psi)\wedge\overline{\partial}(t+\psi).
\end{align}
For any fixed $t_* \leq -N$, quasi-coordinates for $\omega_{cusp}$ on a neighborhood of the hypersurface $\{t = t_*\}$ are given by $(\rho,\check{x},\check{y},\check{\theta})$, where $\rho = \log |t|$, $(\check{x},\check{y}) = |t_*|^{-1/2}(x,y)$ for a fixed pair $(x,y)$ of linear coordinates on $E$ and $\check\theta = |t_*|^{-1}\theta$ for a fixed angular coordinate $\theta$ along the Hopf circles in $\mathbb{C}^3$. Moreover,
\begin{align}\begin{split}\label{eq:ehtn}
    \ddbar t = |t_*| \cdot (\textit{a fixed $2$-form in $\check{x},\check{y}$}),\;\,\partial t = -\frac{1}{2}e^\rho d\rho + |t_*| \cdot (\textit{a fixed $1$-form in $\check{x},\check{y},\check{\theta}$}).
    \end{split}
\end{align}
Because covariant derivatives with respect to $\omega_{cusp}$ are equivalent to ordinary partial derivatives with respect to $(\rho,\check{x},\check{y},\check{\theta})$, the claim follows from \eqref{eq:shgjl} and \eqref{eq:ehtn}.
\end{proof}

We now estimate the terms on the right-hand side of \eqref{mmm}. By  Lemma \ref{l:Fubini} and Claim \ref{claim:herrdeninger},
\begin{align}\label{111}
\begin{split}
    &|\Delta_{\omega_{KE,0}}(\Phi^*_\sigma\psi_{FS,\sigma}-\psi_{FS,0})|\\
    &\leq|i\partial\overline{\partial} (\Phi^*_\sigma\psi_{FS,\sigma}-\psi_{FS,0})|_{\omega_{flat}|_{TY_0}} \cdot {\rm{tr}}_{\omega_{KE,0}}(\omega_{flat}|_{TY_0})\\
    &=O(|\sigma||z|^{-3} |z|^2|t|^{K_0})
\end{split}
\end{align}
as $|\sigma|^{-1/3}|z| \sim e^{(1/2)(t-T)} \to \infty$. Because the latter condition is satisfied for $t \in [2\tau,-N]$, the claimed estimate \eqref{4444} follows for the Laplacian term for $k = 0$ (with a good factor of $|\sigma|^{2/3}$ to spare). The first derivative of this term can be estimated in a similar fashion:
\begin{align}\label{2222}
\begin{split}
    &|\nabla_{\omega_{KE,0}}(\Delta_{\omega_{KE,0}}(\Phi^*_\sigma\psi_{FS,\sigma}-\psi_{FS,0}))|_{\omega_{KE,0}}\\
    &\leq|\nabla_{\omega_{KE,0}}(\ddbar(\Phi^*_\sigma\psi_{FS,\sigma}-\psi_{FS,0}))|_{\omega_{KE,0}}\\
    &\leq |\nabla_{\omega_{flat}|_{TY_0}}(\ddbar(\Phi^*_\sigma\psi_{FS,\sigma}-\psi_{FS,0}))|_{\omega_{flat}|_{TY_0}} \cdot ({\rm{tr}}_{\omega_{KE,0}}(\omega_{flat}|_{TY_0}))^{\frac{3}{2}}\\
    &+ |(\ddbar(\Phi^*_\sigma\psi_{FS,\sigma}-\psi_{FS,0}))|_{\omega_{KE,0}} \cdot {\rm{tr}}_{\omega_{flat}|_{TY_0}}(\omega_{KE,0}) \cdot |\nabla_{\omega_{KE,0}}(\omega_{flat}|_{TY_0})|_{\omega_{KE,0}}\\
    &= O(|\sigma||z|^{-4}(|z|^2|t|^{K_0})^{\frac{3}{2}}+|\sigma||z|^{-3} |z|^{-2}|t|^{-K_0}|z|^2|t|^{K_1})
\end{split}
\end{align}
as $|\sigma|^{-1/3}|z| \to \infty$. This is again of the desired order with a factor of $|\sigma|^{2/3}$ to spare. The estimate of the quadratic term and its first derivative is very similar and we omit it.

It remains to estimate the second line of \eqref{fff}. Using \eqref{eq-smallness} with $f = (\Phi_\sigma^{-1})^*\psi_0 + \psi_{FS,\sigma}$ and with reference metric $\omega = \omega_{KE,0}$, and using Lemma \ref{l:psiJerror} and \cite[Thm 1.4]{DFS} (to compare $\omega_{cusp}$ to $\omega_{KE,0}$),
\begin{align}\label{eq:hgec}
\sum_{k=0}^1 |\nabla^k_{\omega_{KE,0}}(\Phi^*_\sigma\ddbars f-\ddbar \Phi_\sigma^*f)|_{\omega_{KE,0}} \leq O(e^{-(\frac{1}{2}-\varepsilon)(t-T)})\sum_{k=1}^3 |\nabla^k_{\omega_{cusp}}\Phi_\sigma^*f|_{\omega_{cusp}}.
\end{align}
We will first show that the three derivatives of $\Phi_\sigma^*f$ that appear on the right-hand side are $O(1)$.

To this end, observe that
\begin{align}\label{eq:gfhj}
\Phi_\sigma^*f = (\psi_{FS,0} + \psi_0) + (\Phi_\sigma^*\psi_{FS,\sigma} - \psi_{FS,0}).
\end{align}
The first term is equal to the potential $\psi_{KE,0}$ of $\omega_{KE,0}$. Using Proposition \ref{KE-CUSP} and quasi-coordinates for $\omega_{cusp}$, one checks that the $\omega_{cusp}$-gradient and all higher covariant derivatives of this potential are $O(1)$. For the second term in \eqref{eq:gfhj}, we first show as in the proof of \eqref{hhhhhhh} that
\begin{align}\label{eq:xbdf}
    \Phi_\sigma^*\psi_{FS,\sigma} - \psi_{FS,0} = \tilde{O}(|\sigma||z|^{-1})|_{TY_0} 
\end{align}
as $|\sigma|^{-1/3}|z| \to \infty$, using the fact that $\psi_{FS}(z)$ is real-analytic and $O(|z|^2)$ as $|z| \to 0$. Second, as in \eqref{eq:dfgh}--\eqref{aaaaaa}, we use \eqref{eq:xbdf} to further estimate the covariant derivatives of this function with respect to $\omega_{flat}|_{TY_0}$. Here we also need the second covariant derivative of $dz_i|_{TY_0}$ ($i = 1,2,3$), but this is seen to be bounded by a constant times $|z|^{-2}$ by the same argument as after \eqref{aaaaaa}. Thus,
\begin{align}
    |\nabla^k_{\omega_{flat}|_{TY_0}}(\Phi_\sigma^*\psi_{FS,\sigma} - \psi_{FS,0})|_{\omega_{flat}|_{TY_0}} = O(|\sigma||z|^{-1-k})
\end{align}
for $k=1,2,3$ and $|\sigma|^{-1/3}|z|\to\infty$. Lastly, by using \eqref{hhh1} and following the pattern of \eqref{111}--\eqref{2222} (i.e., expanding  $\nabla^k_{\omega_{flat}|_{TY_0}}u$  in $\omega_{cusp}$-normal coordinates and solving for $\partial^k u= \nabla_{\omega_{cusp}}^k u$), we get
\begin{align}\label{eq:goo}
    |\nabla^k_{\omega_{cusp}}(\Phi_\sigma^*\psi_{FS,\sigma} - \psi_{FS,0})|_{\omega_{cusp}} = O(|\sigma||z|^{-1}|t|^{K}) = o(1)
\end{align}
for $k=1,2,3$ and $|\sigma|^{-1/3}|z|\to\infty$, where $K > 0$. To summarize, the three derivatives of $\Phi_\sigma^*f$ on the right-hand side of \eqref{eq:hgec} are $O(1)$, and this is sharp due to the contribution of $\nabla_{\omega_{cusp}}\psi_{KE,0}$.

To finish the proof of the lemma, we estimate the second line of \eqref{fff} using \eqref{eq:hgec}. To do so, let us write the second line of \eqref{fff} as $\omega_1^2 - \omega_2^2$. Then, taking norms with respect to $\omega_{KE,0}$,
\begin{align}\begin{split}
&\left|\frac{\omega_1^2-\omega_2^2}{\omega_{KE,0}^2}\right|
= |\omega_1^2 - \omega_2^2| \leq |\omega_1-\omega_2| \cdot (2|\omega_2| + |\omega_1 - \omega_2|).
\end{split}
\end{align}
We estimated $|\omega_1-\omega_2|$ in \eqref{eq:hgec}, and $|\omega_2| = O(1)$ because $\omega_2 = \omega_{KE,0} + \ddbar(\Phi_\sigma^*\psi_{FS,\sigma} - \psi_{FS,0})$ and the $\omega_{KE,0}$-norm of the second term was proved to be $o(1)$ in \eqref{111}. The first derivative of the second line of \eqref{fff} can be bounded in a similar fashion, using also \eqref{eq:hgec} with $k = 1$ and \eqref{2222}.
\end{proof}

\begin{lemma}\label{l:metricerror}
Abuse notation by replacing $(\Psi_0^{-1})^*\psi_0 \leadsto \psi_0$ and $(\Psi_\sigma^{-1})^*\psi_{FS,\sigma} \leadsto \psi_{FS,\sigma}$.
Then:
\begin{enumerate}
\item[$(1)$] For all $k \geq 0$ and $t\in [2\tau,\tau]$, one has
\begin{align}
\begin{split}
&|\nabla^k_{\omega_{cusp}}(\Phi_\sigma^*\ddbars (\Phi_\sigma^{-1})^*\psi_{cusp}-\Phi_\sigma^*\ddbars (\Phi_\sigma^{-1})^*\psi_{T})|_{\omega_{cusp}}\\
    &=O(|b||\tau|^{3})+O(e^{-(\frac{1}{2}-\epsilon)(t-T)}).
\end{split}
\end{align}

\item[$(2)$] For $k=0,1$ and $t\in [2\tau,\tau],$ one has
\begin{align}
\begin{split}
     &|\nabla^k_{\omega_{cusp}}(\Phi_\sigma^*\ddbars (\Phi_\sigma^{-1})^*\psi_{cusp}-\Phi_\sigma^*\ddbars ((\Phi_\sigma^{-1})^*\psi_{0}+\psi_{FS,\sigma}))|_{\omega_{cusp}}\\
     &= O(e^{-\delta_0\sqrt{-t}}).
\end{split}
\end{align}
     \item [$(3)$]For $k=0,1$ and $t\in [T+T_0, T+2T_0]$,
     one has
\begin{align}
\begin{split}
    &|\nabla^k_{|b|^{\frac{1}{2}}\omega_{\cC,\sigma}}(\Phi_\sigma^*\ddbars (\Phi_\sigma^{-1})^*\psi_{T}-|b|^{\frac{1}{2}}\Phi^*_\sigma\omega_{TY,\sigma})|_{\omega_{cusp}}\\
&=O(e^{-\delta_0 (t-T)^{1/2}})+ O(|b|^{\frac{2-k}{4}}T_0^{\frac{3(2-k)}{4}}).
\end{split}
\end{align}
     \end{enumerate}
\end{lemma}

\begin{proof}
For item $(1)$: By Lemma \ref{CusptoNeckerror}, one has
 for any $k \geq 0$ and for $t \in [2\tau,\tau]$ that
\begin{align}\label{eq-psi-a}
|\nabla_{\omega_{cusp}}^k (\omega_T - \omega_{cusp})|_{\omega_{cusp}} \leq C_k |b||\tau|^{3}.
\end{align}
Here $C_k$ does not depend on $b, \tau$. Then by the triangle inequality and Lemma \ref{Jerror},
\begin{equation}\label{eq:nfhf}
   |\nabla^k_{\omega_{cusp}}(\Phi_\sigma^*\ddbars (\Phi_\sigma^{-1})^*\psi_{cusp}-\Phi_\sigma^*\ddbars (\Phi_\sigma^{-1})^*\psi_{T})|_{\omega_{cusp}}\leq  C_k( |b||\tau|^{3}+e^{-(\frac{1}{2}-\epsilon)(t-T)}).
 \end{equation}

  For item $(2)$: By Lemma \ref{Jerror}, we can replace $\Phi_\sigma^*\ddbars (\Phi_\sigma^{-1})^*\psi_{cusp}$ on the left-hand side of the claim by $\omega_{cusp}$ up to an error of $O(e^{-(1/2-\varepsilon)(t-T)})$, which is negligible because $(t-T)/|t| \rightarrow \infty$ for $t\in [2\tau,\tau]$. Similarly, by Lemma \ref{l:Fubini} and \eqref{hhh1} we can replace $\Phi_\sigma^*\ddbars \psi_{FS,\sigma}$ by $\omega_{FS,0}$ (note that a closely related implication was already proved in \eqref{111}--\eqref{2222}). Next,
  \begin{align}
  \omega_{cusp} - \omega_{FS,0} = (\omega_{cusp}-\omega_{KE,0}) + \ddbar \psi_0,
  \end{align}
  and the first term here was estimated in Proposition  \ref{KE-CUSP}. Thus, it remains to estimate
  \begin{align}
      |\nabla_{\omega_{cusp}}^k (\ddbar \psi_0 - \Phi_\sigma^*\ddbars ((\Phi_\sigma^{-1})^*\psi_{0}))|_{\omega_{cusp}}
  \end{align}
for $k = 0,1$. To do so, we use \eqref{eq-smallness} with $\omega = \omega_{cusp}$ and $f = (\Phi_\sigma^{-1})^*\psi_0$. For this we require a bound on the first three $\omega_{cusp}$-derivatives of $\Phi_\sigma^*f = \psi_0 = \psi_{KE,0} - \psi_{FS,0}$. For $\psi_{KE,0}$, such a bound, which is actually $O(1)$, follows from Proposition \ref{KE-CUSP} using quasi-coordinates for $\omega_{cusp}$. For $\psi_{FS,0}$, we need to go through the same steps as after \eqref{eq:nfhf}: for $k = 1,2,3$ and $|z| \to 0$ we clearly have that
\begin{align}
    |\nabla_{\omega_{flat}|_{TY_0}}^k \psi_{FS,0}|_{\omega_{flat}|_{TY_0}} = O(|z|^{2-k}),
\end{align}
and hence, by \eqref{hhh1} and by computations in $\omega_{cusp}$-normal coordinates, for some $K > 0$,
\begin{align}
|\nabla^k_{\omega_{cusp}}\psi_{FS,0}|_{\omega_{cusp}} = O(|z|^2 |t|^K) = o(1).
\end{align}
In sum, all errors are negligible compared to the one from $\omega_{cusp}-\omega_{KE,0}$, which is $O(e^{-\delta_0\sqrt{-t}})$.

For item $(3)$: Using \eqref{TYError},  \eqref{TY-neckError} and Lemma \ref{Jerror}, one has
 for $k=0,1$,
 \begin{align}
 \begin{split}
  &\left|\nabla^k_{|b|^\frac{1}{2}\omega_{\cC,\sigma}}(\Phi_\sigma^*\ddbars (\Phi_\sigma^{-1})^*\psi_{T}-|b|^{\frac{1}{2}}\Phi^*_\sigma\omega_{TY_\sigma})\right|_{|b|^{\frac{1}{2}}\omega_{\cC,\sigma}}
  \\&=O(e^{-(\frac{1}{2}-\epsilon)(t-T)})+ O(|b|^{\frac{2-k}{4}}T_0^{\frac{3(2-k)}{4}})+O(e^{-\delta_0 (t-T)^{1/2}}).
\end{split}
\end{align}
 When applying \eqref{TYError}, we absorb the factor $|b|^{(2-k)/4}$ into $e^{-\delta_0(t-T)^{1/2}}$ by slightly changing $\delta_0$.
  \end{proof}

\begin{lemma}\label{holovolumeError}
Abuse notation by replacing $(\Psi_\sigma^{-1})^*\Omega_\sigma \leadsto \Omega_\sigma$. Then for some $K > 0$ and for all $\varepsilon>0$, for $T+T_0\leq t\leq -N$, for $j = 0,1$ and for $\omega$ equal to either $\omega_{cusp}$ or $\omega_T$, it holds on $TY_0$ that 
\begin{equation}
\left|\nabla^j_{\omega}\log\frac{\Phi_\sigma^*\Omega_\sigma\wedge\Phi_\sigma^*\overline{\Omega}_\sigma}{\Omega_0\wedge\overline{\Omega}_0}\right|_\omega=O(|\sigma||z|^{-3}|t|^K)=O(e^{-(\frac{3}{2}-\varepsilon)(t-T)}).
\end{equation}
\end{lemma}

\begin{proof}
In the following, $z_1,z_2,z_3$ and $dz_1,dz_2,dz_3$ are automatically understood to be restricted from $\mathbb{C}^3$ to $TY_\sigma$ or $TY_0$. Similarly, $\tilde{O}(\ldots)$ as in \eqref{mmmm} denotes a function on a neighborhood of the origin in $\mathbb{C}^3$, and we will restrict this function to $TY_\sigma$ or $TY_0$ without writing the restriction symbol.

By the proof of Lemma \ref{Hvolumeratio} we have for all $\sigma$ (including $\sigma=0$) that
\begin{equation}\label{eq-Omega-sigma-1}
\Omega_\sigma=(1+f(z))\Omega_{TY_\sigma},
\end{equation}
 where $f(z)$ is holomorphic on some neighborhood of the origin in $\mathbb{C}^3$ with $f(0) = 0$. We may assume without loss of generality that $|z_3|\geq\max\{|z_1|,|z_2|\}$ at the point of $TY_0$ at which we are working. Then $z_3 \neq 0$ because $t > -\infty$. Also, $\Phi_\sigma(z)_3 = z_3 + \nu_\sigma(z)\overline{z}_3^2 \neq 0$ because otherwise $|z_3| = O(|\sigma||z|^{-4})|z_3|^2$, so $1 = O(|\sigma||z|^{-4})|z_3| = O(|\sigma| |z|^{-3}) = o(1)$ because $t-T \to \infty$, and this is a contradiction. This means that we can use the same Poincar\'e residue to represent $\Omega_{TY_0}|_z$ and $\Omega_{TY_\sigma}|_{\Phi_\sigma(z)}$, i.e., 
 \begin{align}\label{eq:poinc}
\Omega_{TY_0}|_z = \frac{dz_1 \wedge dz_2}{z_3^2},\;\, \Omega_{TY_\sigma}|_{\Phi_\sigma(z)} = \frac{dz_1 \wedge dz_2}{\Phi_\sigma(z)_3^2}.
\end{align}
Then from \eqref{ooo} and \eqref{Pullbackz} we have that
\begin{align}\label{eq:ilikethis}
\begin{split}\Phi_\sigma^*(z_i)&=z_i+\tilde O(|\sigma||z|^{-2}),\\
 \Phi_\sigma^*(dz_1)\wedge  \Phi_\sigma^*(dz_2)&=dz_1\wedge dz_2+\sum\nolimits_{|I|+|J|=2} \tilde O(|\sigma||z|^{-3}) \, dz_I\wedge d\overline{z}_J,
 \end{split}
\end{align}
where $I,J$ are multi-indices with values in $\{1,2,3\}$. Combining \eqref{eq-Omega-sigma-1}, \eqref{eq:poinc} and \eqref{eq:ilikethis} with
\begin{align}
    dz_3=-\frac{z_1^2}{z_3^2}dz_1-\frac{z_2^2}{z_3^2}dz_2,
\end{align}
which holds identically on $TY_0 \cap \{z_3 \neq 0\}$, we obtain that
\begin{align}\label{eq:tod}\begin{split}
\frac{\Phi_\sigma^*\Omega_\sigma\wedge\Phi_\sigma^*\overline{\Omega}_\sigma}{\Omega_0\wedge\overline{\Omega}_0}
&=\left|\frac{1+f(\Phi_\sigma(z))}{1+f(z)} \right|^2 \cdot \frac{\Phi_\sigma^*(dz_1) \wedge \Phi_\sigma^*(dz_2) \wedge \Phi_\sigma^*(d\overline{z}_1) \wedge \Phi_\sigma^*(d\overline{z}_2)}{dz_1 \wedge dz_2 \wedge d\overline{z}_1 \wedge d\overline{z}_2} \cdot \left|\frac{z_3}{\Phi_\sigma^*(z_3)}\right|^4\\
&= \left|\frac{1+f(\Phi_\sigma(z))}{1+f(z)} \right|^2 \cdot \left(1 + {\rm Re}\left(\tilde O(|\sigma||z|^{-3})\hspace{-0.5mm}\left[\frac{z_1}{z_3},\frac{z_2}{z_3}\right]\right)\right) \cdot \left|1 + z_3^{-1} \tilde{O}(|\sigma| |z|^{-2}) \right|^{-4}.
\end{split}\end{align}
The square bracket notation stands for a polynomial in $z_1/z_3, z_2/z_3$ with $\tilde{O}(|\sigma||z|^{-3})$ coefficients.

We can analyze $|\nabla^j_\omega \log \eqref{eq:tod} |_\omega$ using the same technique as in the proofs of Lemmas \ref{l:Fubini} and \ref{l:FubiniVolErr}: First, view $z_1,z_2$ as local coordinates on $TY_0 \cap \{z_3 \neq 0\}$, represent $\omega_{flat}|_{TY_0}$ by a matrix and estimate the $\omega_{flat}|_{TY_0}$-covariant derivatives of $\log \eqref{eq:tod}$ in terms of its partial derivatives with respect to $z_1,z_2$. Here, the properties $-z_3^3 = z_1^3 + z_2^3$ and $|z_3| \geq \max\{|z_1|,|z_2|\}$ need to be used to eliminate or estimate terms that explicitly depend on $z_3$. The upshot is that for $j = 0,1$ and $|\sigma|^{-1/3}|z| \to \infty$,
\begin{align}\label{eq:tod2}
\left|\nabla^j_{\omega_{flat}|_{TY_0}} \log \frac{\Phi_\sigma^*\Omega_\sigma\wedge\Phi_\sigma^*\overline{\Omega}_\sigma}{\Omega_0\wedge\overline{\Omega}_0} \right|_{\omega_{flat}|_{TY_0}} = O(|\sigma||z|^{-3-j}).
\end{align}
Secondly, we can rewrite \eqref{eq:tod2} in terms of $\omega_{cusp}$ instead of $\omega_{flat}|_{TY_0}$ by using \eqref{hhh1}. This implies the statement of the lemma for $\omega = \omega_{cusp}$. For $\omega = \omega_T$ we use a reduction as in the proof of Lemma \ref{l:psiJerror}. Fix a small $\delta > 0$. If $t/T \leq 1-\delta$, then $\omega_T$ and $\omega_{cusp}$ are uniformly comparable by Proposition \ref{prop-T-c}. Thus, the version of the lemma for $\omega_{T}$ follows from the one for $\omega_{cusp}$ in this case. If $t/T \geq 1-\delta$, then $\omega_T$ is comparable to $|b|^{1/2}\omega_{\cC,\sigma}$ by \eqref{eq-Psi-T-11}--\eqref{eq-Psi-T-12}. Redoing the proof of \eqref{hhh1} for $\omega_{\cC,\sigma}$ instead of $\omega_{cusp}$ as the reference metric, it is clear that $|\omega_{flat}|_{TY_0}|_{\omega_{\cC,\sigma}} = O(|z|^2 |b|^{-L} (t - T)^{L})$ for some $L > 0$ if $|z| \to 0$ with $t - T \geq 1$. Thus, the statement of the lemma again follows from \eqref{eq:tod2}.
\end{proof}

We now finally come to the proof of our main result, Theorem \ref{f:Riccipotential}.

\begin{proof}[Proof of Theorem \ref{f:Riccipotential}]

Firstly, let us recall from \eqref{riccipotential} the expression of the Ricci potential:
\begin{equation}
    f_\sigma|_{\mathcal{Y}_\sigma} = \log\left(\frac{\omega_{glue,\sigma}^2}{\Omega_\sigma\wedge\overline{\Omega}_\sigma}\right)-\psi_{glue,\sigma}.
\end{equation}
We will do the proof case by case by analyzing this formula.\medskip\

\noindent {\bf Region $\mathfrak{R}_1$:} In this region, we have that $t>-N$. It follows from the smoothness of $\Omega_\sigma$ and $\psi_{FS,\sigma}$ with respect to $\sigma$ and from the K\"ahler-Einstein equation 
\begin{equation}\label{BBB}
    \log\frac{\omega_{KE,0}^2}{\Omega_0\wedge\overline{\Omega}_0}=\psi_{KE,0}
\end{equation}
that $|\nabla^k_{\omega_{glue,\sigma}}f_\sigma|_{\omega_{glue,\sigma}}=O(|\sigma|)$ for $k=0,1$.\medskip\

\noindent {\bf Region $\mathfrak{R}_2$:} For region $\mathfrak{R}_2$,  as in Lemma \ref{l:FubiniVolErr},  we omit the pull-back map $\Psi_\sigma^{-1}$ for simplicity. In this region,  ${\tau}<t<-N$. Using \eqref{BBB}, we have that
   \begin{align}\label{eq-f-R3}
   \begin{split}
     \Phi^*_\sigma f_\sigma&= \log\frac{\Phi^*_\sigma(\ddbars\psi_{glue,\sigma})^2}{\Phi^*_\sigma\Omega_\sigma\wedge\Phi^*_\sigma\overline{\Omega}_\sigma}+\log\frac{\Omega_0\wedge\overline{\Omega}_0}{\omega^2_{KE,0}}-\psi_0-\Phi_\sigma^*\psi_{FS,\sigma}+\psi_{KE,0}\\
     &=\log\frac{\Phi^*_\sigma(\ddbars\psi_{glue,\sigma})^2}{\omega_{KE,0}^2}+\log\frac{\Omega_0\wedge\overline{\Omega}_0}{\Phi^*_\sigma\Omega_\sigma\wedge\Phi^*_\sigma\overline{\Omega}_\sigma}-\Phi_\sigma^*\psi_{FS,\sigma}+\psi_{FS,0}.
     \end{split}
    \end{align}
We estimate the three terms separately. 
By Lemma \ref{l:FubiniVolErr}, for $k=0,1$, we have that
\begin{align}
\left|\nabla^k_{\omega_{KE,0}}\log\frac{\Phi^*_\sigma(\ddbars\psi_{glue,\sigma})^2}{\omega_{KE,0}^2}\right|_{\omega_{KE,0}}=O(e^{-(\frac{1}{2}-\epsilon)(t-T)}).\end{align}
By Lemma \ref{holovolumeError}, for $k=0,1$, the second term can be estimated as
\begin{equation}\label{eq-Omega-ratio}
\left|\nabla^k_{\omega_{KE,0}}\log\frac{\Phi_\sigma^*\Omega_\sigma\wedge\Phi_\sigma^*\overline{\Omega}_\sigma}{\Omega_0\wedge\overline{\Omega}_0}\right|_{\omega_{KE,0}}=O(e^{-(\frac{3}{2}-\epsilon)(t-T)}).
\end{equation}
For the third term, for $k=0,1$, by \eqref{eq:nfhf} and Proposition \ref{KE-CUSP},
\begin{equation}|\nabla^k_{\omega_{KE,0}}(\Phi_\sigma^*\psi_{FS,\sigma}-\psi_{FS,0})|_{\omega_{KE,0}}
=O(|\sigma||z|^{-1}|t|^{K}) = O(|z|^2 e^{-(\frac{3}{2}-\epsilon)(t-T)}).
\end{equation}
Dropping the smaller terms, we get
\begin{align}
|\nabla_{\omega_{KE,0}}^k\Phi_\sigma^* f_\sigma|_{\omega_{KE,0}}=O(e^{-(\frac{1}{2}-\epsilon)(t-T)}) \;\,(k=0,1).
\end{align} 
Using the metric equivalence of $\Phi^*_\sigma g_{glue,\sigma}$ and $g_{KE,0}$ in this region, which follows from Lemmas \ref{l:psiJerror} and \ref{Jerror} (using also \eqref{eq:goo} to compare the term $\psi_{FS,\sigma}$ in $\psi_{glue,\sigma}$ to $\psi_{FS,0}$), we conclude that
\begin{align}\label{eq-Df-R3}
|\nabla^k_{\omega_{glue,\sigma}} f_\sigma|_{\omega_{glue,\sigma}}=O(e^{-(\frac{1}{2}-\epsilon)(t-T)})\;\,(k=0,1).\end{align} 

\noindent \textbf{Region $\mathfrak{R}_4$:} In this case $\psi_{glue,\sigma}=(\Phi_\sigma^{-1})^*\psi_T$ and $\omega_{glue,\sigma}=\ddbars(\Phi_\sigma^{-1})^*\psi_T$.
    So we have 
\begin{equation}
    f_\sigma=\log\frac{(\ddbars\psi_{glue,\sigma})^2}{\Omega_\sigma\wedge\overline{\Omega}_\sigma}-(\Phi^{-1}_\sigma)^*\psi_T.
\end{equation}
\noindent Hence
\begin{align}
\begin{split}
   \Phi^*_\sigma f_\sigma&= \log\frac{\Phi^*_\sigma(\ddbars\psi_{glue,\sigma})^2}{\Phi^*_\sigma\Omega_\sigma\wedge\Phi^*_\sigma\overline{\Omega}_\sigma}-\psi_T\\
   &=\log\frac{\Phi^*_\sigma(\ddbars\psi_{glue,\sigma})^2}{\omega_T^2}+\log\frac{\Omega_0\wedge\overline{\Omega}_0}{\Phi^*_\sigma\Omega_\sigma\wedge\Phi^*_\sigma\overline{\Omega}_\sigma}+\log\frac{\Omega_{\cC}\wedge\overline{\Omega}_{\cC}}{\Omega_0\wedge\overline{\Omega}_0}+\log\frac{\omega^2_T}{\Omega_{\cC}\wedge\overline{\Omega}_{\cC}}-\psi_T.
   \end{split}
\end{align}
By equation \eqref{normalizingpotential}, the last two terms combine to zero. 

Now we estimate the first three terms one by one. Using Lemma \ref{Jerror}, for $j=0,1$ one has
\begin{equation}
\left|\nabla^j_{\omega_T}\left(\log\frac{\Phi^*_\sigma(\ddbars\psi_{glue,\sigma})^2}{\omega_T^2}\right)\right|_{\omega_T}=O(e^{-(\frac{1}{2}-\epsilon)(t-T)}).
\end{equation}
By Lemma \ref{holovolumeError}, for $j=0,1,$ we estimate the second term by
\begin{equation}
\left|\nabla^j_{\omega_T}\left(\log\frac{\Phi_\sigma^*\Omega_\sigma\wedge\Phi_\sigma^*\overline{\Omega}_\sigma}{\Omega_0\wedge\overline{\Omega}_0}\right)\right|_{\omega_T}=O(e^{-(\frac{3}{2}-\epsilon)(t-T)}).
\end{equation}
By \eqref{Omega0-OmegaC} and \eqref{ggg} with $C=1$, for $j = 0,1$ we estimate the third term by
\begin{equation}\left|\nabla^j_{\omega_T}\left(\log\frac{\Omega_{\cC}\wedge\overline{\Omega}_{\cC}}{\Omega_0\wedge\overline{\Omega}_0}\right)\right|_{\omega_T}=O(e^{\frac{(1-\epsilon)t}{2}}).
\end{equation}
Dropping the smaller term,  one has
\begin{equation}
|\nabla_{\omega_T}^j\Phi^*_\sigma f_\sigma|_{\omega_T}=O(e^{-(\frac{1}{2}-\epsilon)(t-T)})+O(e^{(\frac{1}{2}-\varepsilon)t})\;\,(j = 0,1).
\end{equation}
 Using the metric equivalence of $\Phi^*_\sigma g_{glue,\sigma}$ and $g_T$ in this region, we conclude that 
\begin{align}\label{eq-f-R7}
|\nabla_{\omega_{glue,\sigma}}^j f_\sigma|_{\omega_{glue,\sigma}}=O(e^{-(\frac{1}{2}-\epsilon)(t-T)})+O(e^{(\frac{1}{2}-\epsilon)t})\;\,(j = 0,1).\end{align} 

\noindent {\bf Regions $\mathfrak{R}_6, \mathfrak{R}_7$:}  For regions $\mathfrak{R}_6,\mathfrak{R}_7$ we estimate $f_\sigma$ directly on $TY_\sigma$. In these two cases, one has
\begin{equation}\label{eq:i_need_this}
    f_\sigma=\log\frac{(\mathfrak{c}|b|^{\frac{1}{2}}m_\sigma^*\ddbar\psi_{TY_1})^2}{\Omega_\sigma\wedge\overline{\Omega}_\sigma}-\mathfrak{c}|b|^{\frac{1}{2}}m_\sigma^*\ddbar\psi_{TY_1}-\psi_T(T).
\end{equation}
It follows from 
\begin{equation}
    {(m_\sigma^*\ddbar\psi_{TY_1})^2}={m_\sigma^*\Omega_{TY_1}\wedge m_\sigma^*\overline{\Omega}_{TY_1}}
\end{equation}
 that
\begin{equation}
    \log\frac{(\mathfrak{c}|b|^{\frac{1}{2}}m_\sigma^*\ddbar\psi_{TY_1})^2}{\Omega_\sigma\wedge\overline{\Omega}_\sigma}-\psi_T(T)=\log\Big(\mathfrak{c}^2|b||F_\sigma|^2e^{-\psi_{T}(T)}\Big),
\end{equation}
where
\begin{align}
    F_\sigma:=\frac{m_\sigma^*\Omega_{TY_1}}{\Omega_\sigma}.
\end{align}

Now we claim that
\begin{equation}\label{eq:herrnikolaus}\mathfrak{c}^2|b|=e^{\psi_T(T)}.
\end{equation}
This can easily be checked as follows: Using the fact that $\mathfrak{c} = \sqrt{2}$ in dimension $n = 2$ and that 
\begin{equation}
    e^{\psi_T(T)+a}+b=0,
\end{equation}
we see that \eqref{eq:herrnikolaus} is equivalent to $a = -{\log 2}$, which follows from the equation
\begin{equation}
e^{-\psi_{cusp}}\omega_{cusp}^2 =\Omega_{\mathcal C}\wedge\overline{\Omega}_{\mathcal C} = \omega_{\cC}^2
\end{equation}
and from the explicit form of $\psi_{cusp}, \psi_{\mathcal C}$ in \eqref{eq-psi-cusp-2}, \eqref{eq-psi-C} respectively.

Hence 
\begin{equation}
    \log\Big(\mathfrak{c}^2|b||F_\sigma|^2e^{-\psi_T(T)}\Big)=\log|F_\sigma|^2.
\end{equation}
On $TY_1$, if $|z|\gg1$, then $|z|^2$ is uniformly comparable to $h = e^t$.
Recall that $m_\sigma(z) = \sigma^{-1/3} z$. So on $TY_\sigma$, points in regions $\mathfrak{R}_6$ and $\mathfrak{R}_7$ satisfy, for some constant $C$ independent of $\sigma$,
\begin{align}
|z|^2 \leq C |\sigma|^{\frac{2}{3}}e^{T_0} = Ce^{T + T_0}.   
\end{align}
Using Lemma \ref{Hvolumeratio}, one has
\begin{equation}\label{AAA}\log|F_\sigma|=\log\left|\frac{m_\sigma^*\Omega_{TY_1}}{\Omega_\sigma}\right|=\tilde O(|z|)= O(e^{\frac{T+T_0}{2}}).
\end{equation}
Noticing that $\mathfrak{c}|b|^{1/2}m_\sigma^*\ddbar\psi_{TY_1}$ is Ricci-flat, we can deduce from the Cheng-Yau gradient estimate for harmonic functions \cite[p.350, Thm 6]{CY} and from \eqref{AAA} that 
\begin{equation}\label{eee}
\left|\nabla_{\mathfrak{c}|b|^{\frac{1}{2}}m_\sigma^*\omega_{TY_1}} (\log|F_\sigma|^2)\right|_{\mathfrak{c}|b|^{\frac{1}{2}}m_\sigma^*\omega_{TY_1}}=O(|b|^{-\frac{1}{4}}e^{\frac{T+T_0}{2}}).
\end{equation}
Using that $|b|$ is comparable to $|T|^{-3}$, we have
\begin{equation}\label{ffff}
\mathfrak{c}|b|^{\frac{1}{2}}m_\sigma^*\psi_{TY_1}=O((T_0/|T|)^{\frac{3}{2}}).
\end{equation}
Also, for $\log h \leq T_0$,
\begin{equation}
\left|\nabla_{\omega_{TY_1}}\psi_{TY_1}\right|_{\omega_{TY_1}}=O(T_0^{\frac{3}{4}}),
\end{equation}
and hence
\begin{equation}
\left|\nabla_{\mathfrak{c}|b|^{\frac{1}{2}}m_\sigma^*\psi_{TY_1}}\mathfrak{c}|b|^{\frac{1}{2}}m_\sigma^*\psi_{TY_1}\right|_{\mathfrak{c}|b|^{\frac{1}{2}}m_\sigma^*\psi_{TY_1}}=O(|b|^{\frac{1}{4}}T_0^{\frac{3}{4}}).
\end{equation}
Thus, combining the $C^0$ norm and the $C^{1}$ seminorm (weighted by $\mathbf{r}_\sigma$) and dropping the smaller terms \eqref{AAA}--\eqref{eee}, we finish the proof of the case of regions $\mathfrak{R}_6$ and $\mathfrak{R}_7$.\medskip

\noindent {\bf Region $\mathfrak{R}_3$:} In this region, we glue the two Kähler-Einstein metrics $\omega_{KE,0}$ and $\omega_T$. Set $\psi_1, \psi_2$ to be the expression of
$\Phi_\sigma^* \psi_{glue, \sigma}$ in regions $\mathfrak{R}_2,\mathfrak{R}_4$ respectively. By \eqref{eq-f-R3} and \eqref{eq-Omega-ratio},
\begin{align}
\Phi_\sigma^* f_\sigma = \log \frac{\Phi_\sigma^* (i\partial_\sigma\overline{\partial}_\sigma ( (\Phi_\sigma^{-1})^*(\psi_2 + \chi_2 (\psi_1-\psi_2) )))^2}{\omega_{KE, 0}^2}  - (\psi_2 + \chi_2 (\psi_1-\psi_2))+O(e^{-(\frac{3}{2}-\epsilon)(t-T)}).  
\end{align}
If the error term $\chi_2(\psi_1 -\psi_2)$ vanishes, this is  exactly the case of region $\mathfrak{R}_4$, see \eqref{eq-f-R7}. 
In general, we first apply \eqref{eq-smallness} to change 
$\Phi_\sigma^*\ddbars f$ to $i\partial\overline\partial\Phi_\sigma^*f$ for $f=(\Phi_\sigma^{-1})^*(\psi_2 + \chi_2 (\psi_1-\psi_2) )$.
Then we only need to estimate the following additional terms:
\begin{align}\label{eq-error-R3}
\frac{2i\partial\overline{\partial} \psi_2 \wedge  i\partial\overline{\partial}  (\chi_2 (\psi_1-\psi_2) ) + (i\partial\overline{\partial}(\chi_2 (\psi_1-\psi_2)))^2}{\omega_{KE, 0}^2} -\chi_2(\psi_1-\psi_2).
\end{align}

As $t\in [2\tau, \tau],$ we have, for all $j \geq 0$, \begin{align}
|\chi_2^{(j)}(t)| \leq C_j|\tau|^{-j}.   
\end{align}
Applying the quasi-coordinates \eqref{eq-quasi}, we obtain that
\begin{align}\label{gggggg}
|\nabla^j_{\omega_{KE,0}}\chi_2| \leq C_j.   
\end{align}
Combining \eqref{gggggg} with the estimate of $\psi_1-\psi_2$  from Proposition \ref{KE-CUSP} and \eqref{eq-cusp-T-1}, we deduce that
\begin{align}\label{aa}
|\nabla^j_{\omega_{KE,0}} (\chi_2 (\psi_1-\psi_2))|\leq C_j(|b||\tau|^3 + e^{-\delta_0\sqrt{-t}})\;\, (j = 0,1).  
\end{align}
Thus, since $g_{KE, 0}$ and $\Phi_\sigma^*g_{glue, \sigma}$ are uniformly equivalent for $t\in [2\tau, \tau]$,
\begin{align}
|\nabla^j_{\Phi_\sigma^*g_{glue,\sigma}} (\chi_2 (\psi_1-\psi_2))|\leq C_j(|b||\tau|^3 + e^{-\delta_0\sqrt{-t}}) \;\, (j = 0,1). 
\end{align}
This provides the required estimate of the cutoff errors \eqref{eq-error-R3}.\medskip\

\noindent {\bf Region $\mathfrak{R}_5$:}  The situation is similar to $\mathfrak{R}_3$. Here we glue $\mathfrak{c}|b|^{1/2}m_\sigma^*\omega_{TY_1}$ and $\omega_T$. From the estimate \eqref{ffff} in region $\mathfrak{R}_6$ and elementary inequalities, we deduce that
\begin{equation}\label{vvv2}
|f_\sigma|+\mathbf{r}_\sigma|\nabla_{\omega_{glue,\sigma}}f_\sigma|_{\omega_{glue,\sigma}}\leq C(T_0/|T|)^{\frac{3}{2}} + |\ddbar(\chi_1E)|_{|b|^{\frac{1}{2}}\omega_{\cC}}+\mathbf{r}_\sigma\left|\nabla_{|b|^{\frac{1}{2}}\omega_{\cC}}\ddbar(\chi_1E)\right|_{|b|^{\frac{1}{2}}\omega_{\cC}},
\end{equation}
where $E(t):= \psi_T(t) - \psi_T(T) - \mathfrak{c} |b|^\frac{1}{2} \psi_{\mathcal{C}}(t)$ is the difference of K\"ahler potentials defined in \eqref{ErrorE}. We estimate $\chi_1E$ as follows, using the same idea as in Proposition \ref{prop:estimate_of_E}.  Firstly, for $t-T \in [T_0, 2T_0]$,
\begin{align}\label{eq-chi}
|\chi_1^{(j)}(t)| \leq C_j T_0^{-j}\;\,(j \geq 0).
\end{align}
Using the estimate of $E$ from \eqref{eq-E-j}, we further deduce that there exists a $C$ such that, for $t-T \in [T_0, 2T_0]$ and $j=0,1,2,3$,
\begin{align}\label{eq-f-j}
|((t-T)\partial_t)^j (\chi_1E)|\leq C |b|T_0^3.
\end{align}
Using quasi-coordinates, we have,
for $t-T\in [T_0, 2T_0]$ and $j=0,1$,
\begin{equation}\label{vvv1}
\left| \nabla^j_{|b|^{\frac{1}{2}}\omega_{\cC}} (i\partial   \overline{\partial}
(\chi_1 E))\right|_{|b|^{\frac{1}{2}}\omega_{\cC}} \leq  C|b|^{\frac{1}{4}(2-j)}T_0^{\frac{3}{4}(2-j)}.
\end{equation} 
The claimed estimate in region $\mathfrak{R}_5$ now follows from \eqref{vvv2}
and \eqref{vvv1}.\medskip

This completes the proof of Theorem \ref{f:Riccipotential}.
\end{proof}

\subsection{Definition of the weight functions and weighted Hölder norms}
         
We have set up the pre-glued manifold $(\mathcal{X}_\sigma, \omega_{glue,\sigma})$ and the relevant Monge-Amp\`ere equation in \eqref{def:metric}--\eqref{mainMA} and estimated its right-hand side, i.e., the Ricci potential $f_\sigma$ of $\omega_{glue,\sigma}$, in Theorem \ref{f:Riccipotential}. In Remark \ref{rem:savin} we discussed what can be said about the solution $u_\sigma$ using off-the-shelf arguments. This was insufficient (only) on the Tian-Yau region, motivating the development of a weighted Hölder space theory.
 
\begin{definition}\label{def:weights}
Fix a parameter $\delta > 0$. In the rest of the paper $\delta$ will always be chosen arbitrarily close to zero. With this in mind we define two weight functions ${w}_\sigma, \tilde{w}_\sigma: \mathcal{X}_\sigma \to \mathbb{R}^+$ as follows:
\begin{align}\label{eq:def_weight}
w_\sigma &:= 
\begin{cases}
|T|^{-\delta} &\text{on}\;\,\Psi_\sigma^{-1}(\Phi_\sigma(\{t > -N\})) \cup (\mathcal{X}_\sigma \setminus {\rm dom}\,\Psi_\sigma),\\
\Psi_\sigma^*(\Phi_\sigma^{-1})^*\left((t-T)^{-\delta}\right)&\text{on}\;\,\Psi_\sigma^{-1}(\Phi_\sigma(\{t < -N\})),\\
1 &\text{on}\;\,\mathfrak{R}_7,
\end{cases}\\
\tilde{w}_\sigma &:= \mathbf{r}_\sigma^{-2} w_\sigma.\label{eq-w-tilde}
\end{align}
 
\begin{definition}
With the weight $w_\sigma$ from \eqref{eq:def_weight} and the regularity scale $\mathbf{r}_\sigma$ from \eqref{eq:reg_scale} we define for all $0 \leq k \leq 4$ and for all locally $C^k$ functions $\phi$ on $\mathcal{X}_\sigma$:
\begin{align}\label{eq:def_wtd_norm}
\|\phi\|_{C^k_{w}} : =\sum_{j=0}^k \left\|w_\sigma^{-1}\mathbf{r}_\sigma^{j}\left|\nabla^j_{\omega_{glue,\sigma}}\phi\right|_{\omega_{glue,\sigma}}\right\|_{L^\infty(\mathcal{X}_\sigma)}.
\end{align}
Moreover, for all $0 \leq k \leq 3$ and $\bar\alpha \in (0,1)$ and for all locally $C^{k,\bar\alpha}$ functions $\phi$ on $\mathcal{X}_\sigma$,
\begin{align}\label{eq-C-k-alpha-nrom}
[\phi]_{C^{k,\bar\alpha}_w}:=\sup\left\{\frac{\mathbf{r}_\sigma(p)^{k+\bar\alpha}}{w_\sigma(p)}
\frac{|(\nabla^k_{\omega_{glue,\sigma}}\phi)(p)-(\nabla^k_{\omega_{glue,\sigma}}\phi)(q)|_{\omega_{glue,\sigma}}}{d_{\omega_{glue,\sigma}}(p,q)^{\bar\alpha}}: 0 < d_{\omega_{glue,\sigma}}(p,q) < \mathbf{r}_\sigma(p) \right\}.
\end{align}
Here the numerator of the difference quotient is to be understood using the trivialization of the tangent bundle in quasi-coordinates. Given this, we define
\begin{align}\label{def:weightnorm}
\|\phi\|_{C^{k,\bar\alpha}_{{w}}}:=\|\phi\|_{C^{k}_w}+[\phi]_{C^{k,\bar\alpha}_w}.
\end{align}
Replacing $w_\sigma$ by $\tilde w_\sigma$, we may similarly define a weighted $C^k_{\tilde w}$ and $C^{k,\bar\alpha}_{\tilde w}$ norm. For us, $\bar\alpha$ will always be an arbitrary number in $(0,1)$ whose choice affects neither the arguments nor the results.
\end{definition}

Table \ref{tab:weight} summarizes the behavior of the scale and weight functions and of the Ricci potential in a simplified manner (ignoring constant factors and allowing slightly suboptimal exponents). The only remaining geometric parameters are $|T| \sim |{\log |\sigma|}|$ and $|\tau| \sim (\log |T|)^2$. This information is sufficient to understand almost all of the numerology in Sections \ref{sec-L-inf}--\ref{s:obstruction_theory}.

\begin{table}[!ht]
\caption{Regularity scale $\mathbf{r}_\sigma$, weights $w_\sigma, \tilde{w}_\sigma$, Ricci potential $f_\sigma$.}
\label{tab:weight}\begin{tabular}{|l||l|l|l|l|l|l|}\hline
Region & $\mathfrak{R}_1$ & $\mathfrak{R}_2$ & $\mathfrak{R}_3$ & $\mathfrak{R}_4$  & $\mathfrak{R}_5 \cup \mathfrak{R}_6$ & $\mathfrak{R}_7$\\\hline\hline
Range of $t$ & \diagbox[width=16mm,height=7mm]{}{} & $(\tau,-N)$ & $(2\tau,\tau)$ & $(T + 2|T|^\alpha,2\tau)$ & $(T+ \log R, T + 2|T|^\alpha)$ & \diagbox[width=22mm,height=7mm]{}{}\\\hline
$\mathbf{r}_\sigma$ &$1$&$1$&$1$&$(1-t/T)^{\frac{3}{4}}$&$(1-t/T)^{\frac{3}{4}}$&$|T|^{-\frac{3}{4}}$\\\hline
$w_\sigma$ &$|T|^{-\delta}$&$|T|^{-\delta}$&$|T|^{-\delta}$&$(t-T)^{-\delta}$&$(t-T)^{-\delta}$&$1$\\\hline
$\tilde{w}_\sigma$ &$|T|^{-\delta}$&$|T|^{-\delta}$&$|T|^{-\delta}$&$|T|^{\frac{3}{2}}(t-T)^{-\frac{3}{2}-\delta}$&$|T|^{\frac{3}{2}}(t-T)^{-\frac{3}{2}-\delta}$&$|T|^{\frac{3}{2}}$\\\hline
$f_\sigma$ &$e^{-1.5|T|}$&$e^{-0.4|T|}$&$|\tau|^3|T|^{-3}$&$|T|^{-{\log |T|}}$&$|T|^{-\frac{3}{2}(1-\alpha)}$&$|T|^{-\frac{3}{2}(1-\alpha)}$\\\hline
\end{tabular}
\vspace{-2mm}
\end{table}
\end{definition}

It is clear from standard Schauder theory on a ball of radius $\mathbf{r}_\sigma$ in $\C^2$ that for every $\bar\alpha \in (0,1)$ there exists a constant $C(\bar\alpha)$ independent of $\sigma$ such that the linearization
\begin{align}
    L_\sigma := \Delta_{\omega_{glue,\sigma}} - {\rm Id}
\end{align}
of the complex Monge-Amp\`ere operator satisfies the estimate
\begin{align}\label{eq:weighted_schauder}
\|\phi\|_{C^{2,\bar\alpha}_w} \leq C(\bar\alpha)(\|L_\sigma\phi\|_{C^{0,\bar\alpha}_{\tilde{w}}} + \|\phi\|_{C^0_w})
\end{align}
for all functions $\phi$. It is also clear that $L_\sigma$ is invertible and that we have a uniform $L^2$ bound for $L_\sigma^{-1}$, i.e., $\|\phi\|_{L^2} \leq \|L_\sigma\phi\|_{L^2}$ for all $\phi$ with respect to the $L^2(\mathcal{X}_\sigma,\omega_{glue,\sigma})$ norm on both sides. However, we will see in Section \ref{sec-L-inf} that the obvious desirable strengthening, $\|\phi\|_{C^0_w} \leq C\|L_\sigma\phi\|_{C^0_{\tilde{w}}}$, does not hold for  all $\phi$ for any $C$ independent of $\sigma$. This requires us to introduce an obstruction space.
 
\section{Uniform estimate of the inverse of the linearization modulo obstructions}\label{sec-L-inf}

In Section \ref{ss:def_obstr} we construct a $1$-dimensional function space $\mathbb{R}\cdot \hat{u}_\sigma \subset C^\infty(\mathcal{X}_\sigma)$ such that there is a chance of proving uniform weighted Hölder estimates for the inverse of the restriction of $L_{\sigma}$ to the $L^2(\mathcal{X}_\sigma,\omega_{glue,\sigma})$-orthogonal complement of $\mathbb{R} \cdot \hat{u}_\sigma$. In the rest of this section, starting in Section \ref{ss:uniform-statement}, we then prove via a standard blowup-and-contradiction scheme that these uniform estimates are actually true. The lack of uniformity on the obstruction space $\mathbb{R}\cdot\hat{u}_\sigma$ will be dealt with in Section \ref{s:obstruction_theory}.

\subsection{Definition of the obstruction space}\label{ss:def_obstr}

We first sketch the idea: $\hat{u}_\sigma$ is constant to the left and zero to the right of the middle neck $\mathfrak{R}_4$, and converges in a sufficiently strong sense to $\hat{u}$, a solution to an ODE $L_\infty \hat{u} = 0$ on the half-line $\mathbb{R}^+ = {\rm GH}\lim_{\sigma\to 0}\mathfrak{R}_4$ with boundary values $1$ on the left and $0$ on the right. The solution $\hat{u}$ is uniquely determined by these conditions and is precisely the obstruction that breaks the obvious attempt at proving uniform estimates for $L_{\sigma}^{-1}$ via blowup and contradiction. However, going from $\hat{u}$ to $\hat{u}_\sigma$ turns out to be quite complicated because $\hat{u}_\sigma$ is \emph{not} uniquely defined by the properties we need it to satisfy, so there is no canonical choice of $\hat{u}_\sigma$ and we need to come up with some construction that works. We now describe our solution to this problem.

The first step is to construct good coordinates on the neck $\mathfrak{R}_4$, which will be used throughout this subsection and which will allow us to state the crucial Proposition \ref{p:construct_obstruct}.

As in Corollary \ref{cor-control-on-neck} and Proposition \ref{prop-L-s}, we parametrize the model neck by
\begin{align}
    s = 1-\frac{t}{T} \in (0,1)\;\,\text{and}\;\,\eta = 1-s = \frac{t}{T} \in (0,1).
\end{align}
For any $s_1 < s_2$ in $(0,1)$ we introduce coordinates on the universal cover of $\{s_1 < s < s_2\}$ via 
\begin{align}\label{eq-quasi-10}
\begin{split}
(\check x_\alpha, \check y_\alpha, \check x, \check \theta) := ((-t_*)^{-\frac{1}{2}} x_\alpha, (-t_*)^{-\frac{1}{2}} y_\alpha, -t_* x, (-t_*)^{-1}\theta ),\\
t_* := \left(1-\frac{s_1+s_2}{2}\right)T, \;\, x := -\frac{1}{t}.
\end{split}
\end{align}
Note that, as a function on the universal cover,
\begin{align}
s = s(\check{x}) = 1 - \frac{1-\frac{s_1+s_2}{2}}{\check{x}}
\end{align}
is increasing in $\check{x}$, and is uniformly smoothly bounded in the chart \eqref{eq-quasi-10} if $s_1,s_2$ are fixed.

\begin{convention}\label{conv}
From now on we identify the middle neck $\mathfrak{R}_4 \subset \mathcal{X}_\sigma$ with $\{T+2T_0 < t < 2\tau\} \subset TY_0$ via the diffeomorphism $\Phi_\sigma^{-1} \circ \Psi_\sigma$. Then $s,\eta$ naturally become functions on $\mathfrak{R}_4$ and $\check x_\alpha, \check y_\alpha, \check x, \check \theta$ become functions on the universal cover of $\mathfrak{R}_4$. The range of $s$ on $\mathfrak{R}_4$ is an open interval which exhausts all of $(0,1)$ as $\sigma \to 0$. Also, the $J_\sigma$-Kähler metric $\omega_{glue,\sigma}$ on $\mathfrak{R}_4$ then has the same Kähler potential, $\psi_T$, as the $J_0$-Kähler metric $\omega_T$. By Lemma \ref{Jerror}, the difference of the associated metric tensors $g_{glue,\sigma}$ and $g_T$ measured with respect to either of them is $O(e^{-(1/2-\epsilon)(t-T)})$ including all covariant derivatives.
\end{convention}

In addition to defining the obstruction function $\hat{u}_\sigma$ (which will be done in Definition \ref{def:obstr}), our final goal in this subsection is to prove the following property of $\hat{u}_\sigma$, a key result of this paper. The main cause of complication in our definition of $\hat{u}_\sigma$ is the need to ensure that this is true. Recall the radial volume density $\mu_\infty(s)$ from Lemma \ref{l:volumeform} and the weight $w_\sigma$ from Definition \ref{def:weights}.

\begin{proposition}\label{p:construct_obstruct}
For $\sigma_i \to 0$ let $x_i \in \mathfrak{R}_4 \subset \mathcal{X}_{\sigma_i}$ satisfy $s(x_i) \to c \in (0,1)$ and let $\psi_i \in C^0(\mathcal{X}_{\sigma_i})$ satisfy $w_{\sigma_i}(x_i) |\psi_i| \leq w_{\sigma_i}$ on $\mathcal{X}_{\sigma_i}$. Assume that there exists $\psi_\infty: (0,1) \to \mathbb{R}$ such that for all $s_1 < s_2$ in $(0,1)$, $\psi_i \to \psi_\infty \circ s$ uniformly in the coordinates \eqref{eq-quasi-10} on the universal cover of $\{s_1 < s < s_2\}$. Then
    \begin{align}\label{eq:limit_orth}
       \lim_{i\to\infty}|T_i|^2 \int_{\mathcal{X}_{\sigma_i}} \psi_i\hat{u}_{\sigma_i} \, \omega_{glue,\sigma_i}^2 = 2\pi {\rm Vol}(E) \int_0^1 \psi_\infty(s) \hat{u}(s)\,\mu_\infty(s)\, ds.
    \end{align}
\end{proposition}

The remaining steps are now roughly as follows:
\begin{itemize}
\item In Lemma \ref{l:modeloperators} we prove that as $\sigma \to 0$ the operators $L_{\sigma}|_{\mathfrak{R}_4}$ collapse to a second-order ordinary differential operator $L_\infty$ on the interval $s \in (0,1)$. Moreover, $L_\infty$ is asymptotic to an explicit model operator $L_\infty^\pm$ near each of the two endpoints.
\item Fundamental solutions for $L_\infty^\pm$ can be calculated explicitly (Lemma \ref{l:fundsolutions}).
\item These endpoint asymptotics imply a Liouville theorem (Lemma \ref{l:construct_uhat}) of independent interest: up to scalar multiples there exists a unique entire solution $\hat{u}$ to $L_\infty \hat{u} = 0$ compatible with our weights $w_{\sigma}$, and we have that $\hat{u}(s) \to 1$ as $s \to 0^+$ and $\hat{u}(s) \to 0$ as $s \to 1^-$.
\item The inhomogeneous ODE $L_\infty \hat{v} = \hat{u}$ has a unique solution $\hat{v}$ satisfying $\hat{v}(s) \to 0$ as $s \to 1^-$ and satisfying a Neumann type boundary condition as $s \to 0^+$ (Lemma \ref{l:construct_vhat}).
\item We transplant $\hat{v}$ to a function $\hat{v}_\sigma$ on $\mathcal{X}_\sigma$ using radial cutoff functions and define $\hat{u}_\sigma := L_\sigma \hat{v}_\sigma$ (Definition \ref{def:obstr}). These functions $\hat{u}_\sigma$ converge back to $\hat{u}$ as $\sigma \to 0$ (Lemma \ref{l:conv_uhat}).
\item Using the Neumann property of $\hat{v}$, we prove that $\hat{u}_\sigma$ also satisfies Proposition \ref{p:construct_obstruct}.
\end{itemize}
We will carry out these steps in the following sub-subsections.

\subsubsection{Collapse to an ODE on an interval}

The study of the limit ODE operator $L_\infty$ requires careful attention to various fractional exponents. We work in the general $n$-dimensional setting because the values of these exponents might look like random numbers for $n = 2$.

\begin{lemma}\label{l:modeloperators}
There exists a smooth second-order linear differential operator $L_\infty$ on $(0,1)$ such that:
\begin{itemize}
    \item[$(1)$] Let $\sigma_i \to 0$ as $i \to \infty$. Let $\psi_i \in C^{2}(\mathcal{X}_{\sigma_i})$ be such that for any fixed $s_1 < s_2$ in $(0,1)$,
\begin{equation}\label{ass}
    \sup_{\{s_1 < s < s_2\}} |\psi_i| = O_{s_1,s_2}(1)\;\,\text{and}\;\,\sup_{\{s_1 < s < s_2\}} |L_{\sigma_i}\psi_i| = o_{s_1,s_2}(1)\;\,\text{as}\;\,i\to\infty.
\end{equation}
Then there is a smooth function $\psi_\infty: (0,1) \to \mathbb{R}$ such that $L_\infty \psi_\infty = 0$ and, after passing to a subsequence and pulling back to the universal cover, we have for all $s_1 < s_2$ that $\psi_i \to \psi_\infty \circ s$ weakly in $W^{2,p}_{loc}$ and strongly in $C^{1,\beta}_{loc}$ with respect to the coordinates \eqref{eq-quasi-10} for all $p$ and $\beta$.
    \item[$(2)$] While there is no simple formula for $L_\infty$ globally on $(0,1)$, at the endpoints we have that
    \begin{align}
    L_\infty &= \frac{1}{n+1} L_\infty^{+} + O(\eta^{n+1})\;\,\text{as}\;\,\eta = 1-s \to 0^+,\label{ModelOperator1}\\
    L_\infty &= \frac{1}{d(n)} L_\infty^{-} + O(s^{\frac{n+1}{n}})\;\,\text{as}\;\,s \to 0^+,\;\, d(n) := \left(\frac{n}{n+1}\right)^{\frac{1}{n}}c(n)^{\frac{n+1}{n}}.\label{ModelOperator2}
\end{align}
Here the $O$ notation is to be understood in the sense of Definition \ref{def-pert}, $c(n)$ is as in \eqref{eq:def_c(n)}, and the two model operators $L_\infty^{\pm}$ are given by
\begin{align}
    L_\infty^{+} &:= \eta^2\partial_{\eta\eta}^2 - (n-1)\eta \partial_\eta - (n+1)\cdot{\rm Id},\label{eq:DefMO-}\\
    L_\infty^{-} &:= n s^\frac{n-1}{n}\partial^2_{ss} + (n-1)s^{-\frac{1}{n}}\partial_s - d(n) \cdot {\rm Id}.\label{eq:DefMO+}
\end{align}
\end{itemize}
\end{lemma}

\begin{proof}
The chart $(\check{x},\check{\theta},\check{x}_\alpha,\check{y}_\alpha)$ identifies the universal cover of $\{s_1 < s < s_2\}$ with
\begin{equation}
\left(\frac{1-\frac{s_1+s_2}{2}}{1-s_1},\frac{1-\frac{s_1+s_2}{2}}{1-s_2}\right) \times \mathbb{R}^{2n-1},
\end{equation}
where the interval is compactly contained in $(0,1)$. As recalled in Convention \ref{conv}, $g_{glue,\sigma}$ and $g_T$ are uniformly smoothly bounded on this coordinate slab and differ from each other by 
\begin{align}
    O(e^{-(\frac{1}{2}-\epsilon)(t-T)})\;\,\text{with}\;\, t-T \geq 2|T|^\alpha \;\,\text{as}\;\,\sigma \to 0,
\end{align}
including all derivatives. Also, by construction, $g_{glue,\sigma_i}$, $g_{T_i}$ and $\psi_i$ are invariant under the deck group $\Gamma_i$ of the universal cover, which preserves $\check{x}$ and acts as a discrete Heisenberg group with fundamental domain diameter $\sim (-t_{*,i})^{-1/2} \sim |T_i|^{-1/2} \sim |{\log |\sigma_i|}|^{-1/2}$ on the coordinates $(\check{\theta},\check{x}_\alpha,\check{y}_\alpha) \in \mathbb{R}^{2n-1}$.

By assumption, the $\Gamma_i$-invariant functions $|\psi_i|$ are uniformly bounded and the $\Gamma_i$-invariant functions $|L_{\sigma_i}\psi_i|$ uniformly converge to zero as $i\to\infty$. Here, $L_{\sigma_i} = \Delta_{\omega_{glue,\sigma_i}} - {\rm Id}$ is a uniformly elliptic sequence of differential operators whose coefficients are uniformly smoothly bounded with respect to $i$. Thus, by standard $L^p$ elliptic regularity theory on balls, $\psi_i$ is locally uniformly bounded in $W^{2,p}$ for every $p \in (1,\infty)$ and hence, by Morrey embedding, in $C^{1,\beta}$ for every $\beta \in (0,1)$. By applying the Alaoglu and Arzel\`a-Ascoli compactness theorems and passing to a diagonal subsequence, we have that $\psi_i \to \psi_\infty$ locally weakly in $W^{2,p}$ for all $p \in (1,\infty)$ and strongly in $C^{1,\beta}$ for all $\beta \in (0,1)$. The subsequence can be taken to be independent of $s_1,s_2$ by letting $s_1 \to 0$, $s_2 \to 1$ and diagonalizing. Moreover, for any points $p,q$ in our chart with $s(p) = s(q)$, we can estimate $|\psi_i(p)-\psi_i(q)|$ as follows: Without loss of generality, $p,q$ lie in a fundamental domain of $\Gamma_i$ (because $\psi_i$ is invariant under $\Gamma_i$) and are joined by a curve $\gamma$ with $s \circ \gamma = const$ whose $\omega_{glue,\sigma_i}$-length is $O((-t_{*,i})^{-1/2}) = O(|{\log |\sigma_i|}|^{-1/2}) \to 0$. Since the gradient of $\psi_i$ with respect to $\omega_{glue,\sigma_i}$ is uniformly bounded, this implies that $|\psi_i(p) - \psi_i(q)| = O(|{\log |\sigma_i|}|^{-1/2})$ and hence $\psi_\infty(p)=\psi_\infty(q)$. Thus, $\psi_\infty$ depends only on the radial coordinate $s$. This means that the first-order classical derivatives $\psi_{\infty, \check{x}_\alpha}, \psi_{\infty,\check{y}_\alpha}, \psi_{\infty,\check{\theta}}$ are identically zero, so their weak gradients, which exist locally in $L^p$, are zero a.e. Thus, the corresponding first-order derivatives of $\psi_i$ go to zero strongly in $C^{0,\beta}_{loc}$ for every $\beta$ and their gradients go to zero weakly in $L^p_{loc}$ for every $p$.

We now argue that $\psi_\infty$ is a weak, hence classical, solution of an ODE $L_\infty\psi_\infty = 0$, where $L_\infty$ is a smooth second-order linear differential operator in $s \in (0,1)$ which is independent of the sequence $\psi_i$ and is asymptotically modeled by $L_\infty^\pm$ at the two endpoints. The key to this are Corollary \ref{cor-control-on-neck} and Proposition \ref{prop-L-s}. Written in terms of $(\eta, \check{x}_\alpha,\check{y}_\alpha,\check{\theta})$, the former says that for $\eta \leq \delta < 1$,
\begin{align}\label{eq:operator_in_quasi_coords}
\begin{split}
    \Delta_{\omega_{T_i}} &= \frac{1}{n+1}(\eta^2\partial_{\eta\eta}^2 - (n-1)\eta \partial_\eta) \cdot e^{O_\delta(\eta^{n+1})}\\
    &+ \check{C} \eta^2 \partial^2_{\check{\theta}\check{\theta}} +  \eta \check{T} +(-t_{*,i})^{-\frac{1}{2}} \eta \check{H}_i+O_\delta(|b_i|^{\frac{1}{n+1}}\tau_i),
    \end{split}
\end{align}
\noindent where $\check{C}>0$ is a constant, $\check{T}$ is a constant coefficient Laplacian in the $\check{x}_\alpha,\check{y}_\alpha$ coordinates on $\mathbb{R}^{2n-2}$, and $\check{H}_i$ is a differential operator linear homogeneous in 
\begin{align}
       \partial^2_{\check{x}_\alpha \check{\theta}},  \partial^2_{\check{y}_\alpha \check{\theta}}, (-t_{*,i})^{-\frac{1}{2}}\partial^2_{\check{\theta}\check{\theta}}
\end{align}
with coefficients that are uniformly smoothly bounded and $\Gamma_i$-invariant functions of $\check{x}_\alpha, \check{y}_\alpha$. Applying the identity \eqref{eq:operator_in_quasi_coords} to $\psi_i$ and subtracting $\psi_i$, we now observe the following:
\begin{itemize}
    \item The left-hand side can be written as
    \begin{align}
    (\Delta_{\omega_{T_i}}-{\rm Id})\psi_i = L_{\sigma_i}\psi_i+O(e^{-\frac{1}{2}|T_i|^\alpha}) \circledast (\check\partial\psi_i, \check\partial^2\psi_i),
    \end{align}
    with $\check\partial$ the standard gradient operator in our fixed coordinate chart. The first term uniformly converges to zero by assumption. The second term goes to zero strongly in $L^p_{loc}$ for every $p$.
    \item On the right-hand side, every term with at least one tangential derivative, i.e., the second line of \eqref{eq:operator_in_quasi_coords} applied to $\psi_i$, goes to zero weakly in $L^p_{loc}$ for every $p$.
\end{itemize}
This shows that $L_\infty\psi_\infty = 0$ weakly, hence classically, with $L_\infty + {\rm Id}$ the smooth ordinary differential operator in the first line of \eqref{eq:operator_in_quasi_coords}. The claim \eqref{ModelOperator1} is also clear from this. For the behavior of $L_\infty$ as $s \to 0^{+}$ we can run a similar argument based on Proposition \ref{prop-L-s}.
\end{proof}    

\subsubsection{Fundamental solutions of the endpoint model operators}

These are computed in Lemma \ref{l:fundsolutions}. What makes our life difficult in this paper is precisely the fact that both fundamental solutions of $L_\infty^{-}$ are bounded as $s \to 0^+$: if one of them was at least $\geq s^{-\delta}$ for some $\delta > 0$, which might be one's first guess based on experience with other singularities, then the gluing would be unobstructed.

\begin{lemma}\label{l:fundsolutions}
Recall that if $h_1,h_2$ are two fundamental solutions of a second order linear ODE, then their Wronskian is defined as $w = h_1h_2' - h_1'h_2$. Then the following hold.
\begin{itemize}
    \item[$(1)$] Two fundamental solutions of $L_\infty^+$ and their Wronskian are
\begin{equation}
h^+_1(\eta)=\eta^{n+1},\;\,h^+_2(\eta)=\eta^{-1},\;\,w^+(\eta) = -(n+2)\eta^{n-1}.
\end{equation}
\item[$(2)$] Denote $\lambda = (\frac{4n}{(n+1)^2}\cdot d(n))^{1/2}$ with $d(n)$ defined in \eqref{ModelOperator2}. Recall the modified Bessel functions $I_{\frac{1}{n+1}}, K_{\frac{1}{n+1}}$ of order $\frac{1}{n+1}$. Write $\mathbb{R}\{y_1,\ldots,y_k\}$ to denote the ring of convergent power series in $y_1,\ldots,y_k$ with real coefficients. Two fundamental solutions of $L_\infty^-$ and their Wronskian are
\begin{align}
h_1^-(s) = s^{\frac{1}{2n}}I_{\frac{1}{n+1}}(\lambda s^{\frac{n+1}{2n}}) \in s^{\frac{1}{n}} \cdot \mathbb{R}\{s^{\frac{n+1}{n}}\} \subset \mathbb{R}\{s^{\frac{1}{n}},s^{\frac{n+1}{n}}\},\\ h_2^-(s) = s^{\frac{1}{2n}}K_{\frac{1}{n+1}}(\lambda s^{\frac{n+1}{2n}}) \in \mathbb{R}\{s^{\frac{n+1}{n}}\} + s^{\frac{1}{n}} \cdot \mathbb{R}\{s^{\frac{n+1}{n}}\} \subset \mathbb{R}\{s^{\frac{1}{n}},s^{\frac{n+1}{n}}\},\\
w^-(s) = -\frac{n+1}{2n}s^{\frac{1-n}{n}}.
\end{align}
\end{itemize}
\end{lemma}

\begin{proof}
The computations for $L_\infty^+$ are straightforward. For $L_\infty^-$, a lengthy computation shows that $u(s)$ solves $L_\infty^{-}u = 0$ if and only $u(s) = s^{\frac{1}{2n}}v(\lambda\cdot s^{\frac{n+1}{2n}})$, where $v(y)$ solves the modified Bessel equation
\begin{equation}
    y^2 v''(y) + y v'(y) - \left(\frac{1}{(n+1)^2} + y^2\right)v(y) = 0.
\end{equation}
This yields the fundamental solutions $h_1^-(s),h_2^-(s)$ above. Their Wronskian is easily calculated using the fact that the Wronskian of the modified Bessel functions $I(y),K(y)$ (of any order) is $-\frac{1}{y}$.
\end{proof}

\subsubsection{A Liouville theorem for the limit ODE}\label{sss:ode_liou}

The ODE $L_\infty u = 0$ on the whole limit neck $(0,1)$ has a $2$-dimensional vector space of solutions. Our goal is to prove that the weight $s^{-\delta}$ that we will impose as an upper bound in our blowup argument singles out a $1$-dimensional subspace. Again, experience with other gluing problems suggests that by choosing $0 < \delta \ll 1$ \emph{all} solutions to $L_\infty u = 0$ are ruled out, but here we are left with a $1$-dimensional obstruction space no matter how small we choose $\delta$.

\begin{lemma}\label{l:construct_uhat}
There is a unique solution $\hat{u}(s)$ to the homogeneous ODE $L_\infty \hat{u} =0$ on $(0,1)$ such that $\hat{u}(s) \to 1$ as $s \to 0^+$ and $\hat{u}(s) \to 0$ as $s \to 1^-$. This satisfies the following properties:
\begin{itemize}
    \item[$(1)$] $\hat{u}(s)$ is strictly decreasing in $s$.
    \item[$(2)$] As $\eta\to 0^+$ we have that $\hat{u}(1-\eta) = O(\eta^{n+1})$.
    \item[$(3)$] As $s \to 0^+$ we have that $\hat{u}(s) = 1 - C_1 s^{\frac{1}{n}} + C_2 s^{\frac{n+1}{n}} + R(s)$ with $C_1,C_2 > 0$ and
    \begin{equation}|R(s)| + s|R'(s)| + s^2|R''(s)| = O(s^{\frac{n+2}{n}}).\label{eq:remainder}\end{equation}
    \item[$(4)$] If $L_\infty u = 0$ and $|u(s)| \leq C s^{-\delta}$ for some $\delta \in (0,1)$, then $u$ is a scalar multiple of $\hat{u}$.
\end{itemize}
\end{lemma}

\begin{proof}
The uniqueness of $\hat{u}$ is clear by the maximum principle.

If $\hat{u}$ exists, then (1) can be proved as follows: First note that $\hat{u} \geq 0$ because otherwise there is a global interior minimum with $\hat{u} < 0$ and this contradicts the maximum principle. Next, we prove that $\hat{u}$ is weakly decreasing. Otherwise there exist $s_1 < s_2$ in $(0,1)$ with $\hat{u}(s_1) < \hat{u}(s_2)$. Both of these values are nonnegative by what we said before. By the intermediate value theorem, there is an $s_2' \in (s_2,1)$ such that $\hat{u}(s_1) = \hat{u}(s_2')$, but then $\hat{u}$ attains a strictly positive interior maximum on $[s_1,s_2']$, contradicting the maximum principle. Thus, $\hat{u}$ is weakly decreasing. Next, we prove that $\hat{u}$ is strictly decreasing. If this is false, then $\hat{u}$ must be constant on some interval $[s_1,s_2] \subset (0,1)$ with $s_1 < s_2$. From the ODE, $\hat{u} = 0$ on $[s_1,s_2]$. Hence $\hat{u} = 0$ on $(0,1)$ by interior analyticity, which contradicts $\hat{u}(0) = 1$.

To prove the existence of $\hat{u}$, we solve $L_\infty\hat{u}_\varepsilon=0$ on $[\varepsilon,1-\varepsilon]$ with $\hat{u}_\varepsilon(\varepsilon)=1$ and  $\hat{u}_\varepsilon(1-\varepsilon) = 0$ for any fixed $\varepsilon \in (0,\frac{1}{2})$. This is possible by the standard Dirichlet problem for nonsingular ODEs. The same argument as above shows that $\hat{u}_\varepsilon$ is strictly decreasing, hence in particular bounded by $0$ and $1$. Then, also by standard ODE theory, $\hat{u}_\varepsilon$ satisfies uniform derivative bounds to all orders on every fixed compact interval contained in $(0,1)$. Thus, up to subsequences, $\hat{u}_\varepsilon\to\hat{u}$ locally smoothly on $(0,1)$, where $L_\infty\hat{u} = 0$. The issue is to prove that $\hat{u}$ is not identically zero, and, indeed, that it satisfies the correct boundary conditions as $s \to 0$ and $s \to 1$. This will be done using a barrier argument.

First consider $\hat{u}_{sub}(s) := 1 - C_1 s^{1/n} + C_2 s^\gamma$ for $C_1,C_2>0$ and $\gamma > 1/n$ to be determined. Fix an $s_* \in (0,1)$ such that the $O(s^{(n+1)/n})$ in \eqref{ModelOperator2} is explicitly bounded by $C_*s^{(n+1)/n}$ for all $s \in (0,s_*]$. Here $s_*, C_*$ can be chosen to depend only on $n$. By choosing $C_1,C_2$ such that $C_1 s_*{^{1/n}} = 2$ and $C_2 s_*{^\gamma} = 1$, we get $\hat{u}_{sub}(s_*) = 0$. It remains to check that $\hat{u}_{sub}$ is a subsolution on $(0,s_*]$ because then the maximum principle shows that $\hat{u}_\varepsilon \geq \hat{u}_{sub}$ on $[\varepsilon,s_*]$, and so $\hat{u} \geq \hat{u}_{sub}$, proving that $\hat{u}$ is not identically zero and in fact $\lim_{s\to 0^+} \hat{u}(s) = 1$. To this end we calculate on $(0,s_*]$, using \eqref{ModelOperator2} and the observation that $s^{1/n}$ is a homogeneous solution of the operator $L_\infty + {\rm Id}$:
\begin{equation}
L_\infty\hat{u}_{sub} \geq (-1+C_1 s^{\frac{1}{n}} - C_2 s^\gamma) + (0 - C_*'C_1 s^{\frac{1}{n}}) + (C_*''C_2(n\gamma-1)s^{\gamma-\frac{n+1}{n}}- C_*'''C_2 s^\gamma).\label{eq:helpme}
\end{equation}
Here, $C_*'>0$ depends only on $n$ while $C_*'',C_*'''>0$ depend on $n$ and $\gamma$. Thus, as long as $\frac{1}{n} < \gamma \leq \frac{n+1}{n}$, we can arrange that \eqref{eq:helpme} $\geq 0$ on $(0,s_*]$ by making $s_*$ smaller if necessary. To be precise, we first let $\gamma = \frac{n+1}{n}$ and then choose $s_*$ so small that $C_2 = 1/s_*^\gamma$ satisfies $C_*'' C_2 (n\gamma - 1) \geq 2 + 2C_*' + C_*'''$.

Next, we need to prove that $\hat{u}(s) \to 0$ as $s \to 1^-$. In fact, by applying another barrier argument to the approximating functions $\hat{u}_\varepsilon$, we will prove that (2) holds. For this we again choose a $C_*>0$ and an $\eta_* \in (0,1)$ depending only on $n$ such that the $O(\eta^{n+1})$ in \eqref{ModelOperator1} is bounded by $C_*\eta^{n+1}$ for all $\eta \in (0,\eta_*]$. Then we consider the function $\hat{u}^{sup}(1-\eta) := C_1 \eta^{n+1} - C_2\eta^\gamma$ for $C_1,C_2 > 0$ and $\gamma > n+1$ to be determined. We set $C_1 \eta_*^{n+1} = 2$ and $C_2 \eta_*^\gamma =1$, thus arranging that $\hat{u}^{sup}(1-\eta_*) = 1$. To apply the maximum principle on $[1-\eta_*,1)$, we need to check that $\hat{u}^{sup}$ is a supersolution for $\eta \in (0,\eta_*]$ after making $\eta_*$ smaller if necessary. For this we compute for $\eta \in (0,\eta_*]$ using \eqref{ModelOperator1}:
\begin{equation}
L_\infty \hat{u}^{sup} \leq (-C_1\eta^{n+1} + C_2 \eta^\gamma) + (C_1\eta^{n+1} + C_*'C_1\eta^{2n+2}) + \left(-\frac{\gamma(\gamma-n)}{n+1}C_2\eta^\gamma + C_*'''C_2\eta^{\gamma+n+1}\right),\label{eq:helpme2}
\end{equation}
where $C_*'>0$ depends only on $n$ and $C_*'''>0$ depends only on $n,\gamma$. Thus, as long as $\gamma \leq 2n+2$ and $1-\gamma(\gamma-n)/(n+1)<0$ (equivalently, $\gamma > n+1$), we can make \eqref{eq:helpme2} $\leq 0$. More precisely, it is enough to set $\gamma = 2n+2$ and choose $\eta_*$ so small that $(2C_*' + C_*''')\eta_*^{n+1} \leq \gamma(\gamma-n)/(n+1)-1$.

It remains to prove (3) and (4). The key point, carried out below, is to prove that if $|u| \leq Cs^{-\delta}$, then $u$ actually has an expansion as $s \to 0$ and as $s \to 1$. The expansion as $s \to 1$ proves (2) for a general $u$. The expansion as $s \to 0$ proves (3) for a general $u$, with some constant as the leading term ($=: u(0)$) but without the sign information on $C_1,C_2$. For $u = \hat{u}$ we already know that $\hat{u}(0) = 1$. In this case it is then easy to see that $C_2 > 0$ by plugging the expansion into the ODE $L_\infty\hat{u} = 0$ (here the estimate \eqref{eq:remainder} up to two derivatives is crucial). Then $C_1 > 0$ also follows: if $C_1 < 0$, then $\hat{u}$ would obviously be strictly increasing for a short time, which contradicts (1); if $C_1 = 0$, the same argument applies because we already know that $C_2 > 0$. Lastly, (4) follows by applying the maximum principle to $u-u(0)\hat{u}$ because the expansion tells us that $u - u(0)\hat{u}$ vanishes as $s \to 0^+$ and as $s \to 1^-$.

It remains to prove that a general solution $u$ to $L_\infty u = 0$ on $(0,1)$ with $|u| \leq C s^{-\delta}$ for some $\delta\in (0,1)$ satisfies the expansions (2) and (3). For this, we use the standard general formula
\begin{equation}\label{eq:var_of_par}
u(x) = c_1 h_1(x) + c_2 h_2(x) + h_1(x) \int_{x}^{x_0} \frac{h_2(t)f(t)}{w(t)a(t)}\,dt - h_2(x) \int_{x}^{x_0} \frac{h_1(t)f(t)}{w(t)a(t)}\,dt
\end{equation}
for solutions to $au'' + bu' + cu = f$, where $h_1,h_2$ are two fundamental solutions to the homogeneous equation, $w = h_1 h_2' - h_1' h_2$ is their Wronskian, $x_0$ is an arbitrary point and $c_1,c_2$ are constants.

We start with the easier case, $s \to 1^-$, which will give us (2) for a general $u$. By Lemmas \ref{l:modeloperators}--\ref{l:fundsolutions} and \eqref{eq:var_of_par}, for any $\eta_0 \in (0,1)$ there exist constants $c_1,c_2$ such that for all $\eta\in (0,1)$,
\begin{align}
u(1-\eta)=c_1\eta^{n+1} + c_2 \eta^{-1} + \eta^{n+1}\int_{\eta}^{\eta_0} t^{-n-2}f(t)\,dt - \eta^{-1}\int_{\eta}^{\eta_0} f(t)\,dt,\label{eq:rep}\\
|f(\eta)| \leq c\eta^{n+1}(\eta^2 |u''(1-\eta)| + \eta |u'(1-\eta)| + |u(1-\eta)|)\label{eq:est_inhomog}
\end{align}
for some dimensional constant $c$. We are assuming that $|u(1-\eta)| = O(1)$ as $\eta\to 0^+$.

\begin{claim}\label{claim:herrlohkamp}
    We actually have that
\begin{equation}\label{eq:rough_estimate}
|u(1-\eta)| + \eta|u'(1-\eta)| + \eta^2|u''(1-\eta)| = O(1)\;\,\text{as}\;\,\eta \to 0^+.
\end{equation}
\end{claim}

\begin{proof}[Proof of Claim \ref{claim:herrlohkamp}]
We will first show that there exist $\varepsilon \in (0,1)$ and $C > 0$ independent of $\eta$ such that for any fixed $0 < \eta \ll 1$ the following sub-claim is true.

\begin{subclaim}\label{subclaim:herrwilking}
Define $F: [0,1] \to [0,\infty)$ via $F(\rho) := \max_{[\eta,(1+\rho)\eta]} |u'(1-\cdot)|$. Then 
\begin{align}\label{eq:45429}
    F(\rho) \leq \varepsilon \cdot F(R) + \frac{C}{\eta}(R-\rho)^{-1}\;\,\text{for all}\;\,0\leq \rho < R \leq 1.
\end{align}
\end{subclaim}
If Sub-Claim \ref{subclaim:herrwilking} is true, then by a standard calculus iteration lemma (see \cite[Lemma 3.4]{HT20} based on \cite[Lemma 8.18]{GM}, or many other sources) there exists a universal $C_0=C_0(\epsilon)$ such that
\begin{align}\label{eq:herrwulkenhaar}
    F(\rho) \leq C_0\frac{C}{\eta}(1-\rho)^{-1}\;\text{for all}\;\,0 \leq \rho < 1.
\end{align}
Setting $\rho = 0$ in \eqref{eq:herrwulkenhaar} and using $L_\infty u = 0$ to solve for $u''$ in terms of $u,u'$, we obtain Claim \ref{claim:herrlohkamp}.\medskip\

\noindent \emph{Proof of Sub-Claim \ref{subclaim:herrwilking}.} We will prove \eqref{eq:45429} by a simple interpolation argument. Let $\tilde{u} := u(1-\cdot)$. For any $0 < \epsilon \leq \frac{1}{10}$ and any $\eta_* \in [\eta,(1+\rho)\eta]$, expand
\begin{align}
    \tilde{u}(\eta_* + \varepsilon(R-\rho)\eta) = \tilde{u}(\eta_*) + \int_{\eta_*}^{\eta_*+\varepsilon(R-\rho)\eta} \left[\tilde{u}'(\eta_*) + \int_{\eta_*}^\xi \tilde{u}''(\tilde{\xi})\,d\tilde\xi \right]d\xi.
\end{align}
Then solve this equation for $\tilde{u}'(\eta_*)$, express $\tilde{u}''$ in terms of $\tilde{u}'$ and $\tilde{u}$ using the ODE $L_\infty u = 0$, take the maximum over all $\eta_* \in [\eta,(1+\rho)\eta]$, and fix $\epsilon$ sufficiently small depending only on $n$.
\end{proof}

Plugging Claim \ref{claim:herrlohkamp} into \eqref{eq:est_inhomog}, we get that $|f(\eta)| = O(\eta^{n+1})$ as $\eta \to 0^+$. Inserting this into \eqref{eq:rep}, letting $\eta \to 0^+$ and using the boundedness of the left-hand side, we deduce that
\begin{equation}
c_2 - \int_0^{\eta_0} f(t)\,dt = 0.
\end{equation}
As a consequence, $|u(1-\eta)| = O(\eta^{n+1}|{\log \eta}|)$ as $\eta \to 0^+$. This would already be enough for us but another iteration easily yields the expected $|u(1-\eta)| = O(\eta^{n+1})$, i.e., (2) holds for a general $u$.

To finish the proof, we now deal with the more difficult case $s \to 0^+$, i.e., statement (3). As above, and using also the fact that $a(t) w(t) = -\frac{n+1}{2} = const$ in this case, we have that
\begin{align}
u(s) = c_1 h_1^-(s) + c_2 h_2^-(s) + h_1^-(s)\int_{s}^{s_0} h_2^-(t) f(t) \, dt - h_2^-(s)\int_{s}^{s_0} h_1^{-}(t)f(t) \,dt,\label{eq:rep2}\\
h_1^-(s) = c_{1,1}s^{\frac{1}{n}} + O(s^{\frac{n+2}{n}}),\;\,h_2^-(s) = c_{2,0} + c_{2,1}s^{\frac{1}{n}} + c_{2,2}s^{\frac{n+1}{n}} + O(s^{\frac{n+2}{n}}),\\
|f(s)| \leq c s^{\frac{n+1}{n}}\left(s^{\frac{n-1}{n}}|u''(s)| + s^{-\frac{1}{n}}|u'(s)| + |u(s)|\right)\label{eq:est_inhomog_2}
\end{align}
for some dimensional constants $c_{1,1},c_{2,0},c_{2,1},c_{2,2},c$. We are assuming that $|u(s)| = O(s^{-\delta})$ as $s \to 0^+$ for some fixed $\delta\in(0,1)$. From this and from the ODE $L_\infty u = 0$, we get
\begin{align}
|u(s)| + s |u'(s)| + s^2 |u''(s)| = O(s^{-\delta})\;\,\text{as}\;\,s \to 0^+
\end{align}
as in the proof of Claim \ref{claim:herrlohkamp}. This again directly fits into \eqref{eq:est_inhomog_2}, yielding an estimate $|f(s)| = O(s^{-\delta})$ as $s \to 0^+$. Because of this, the integrals in \eqref{eq:rep2} converge as $s \to 0^+$. In particular, $|u(s)| = O(1)$ as $s \to 0^+$. Again arguing as in the proof of Claim \ref{claim:herrlohkamp}, this implies that
\begin{align}\label{eq:der_bounds_u}
|u(s)| + s |u'(s)| + s^2 |u''(s)| = O(1)\;\,\text{as}\;\,s \to 0^+,
\end{align}
so $|f(s)| = O(1)$. Writing $\int_{s}^{s_0} = \int_{0}^{s_0} - \int_0^s$ in \eqref{eq:rep2}, we thus obtain with new constants $\tilde{c}_1,\tilde{c}_2$ that
\begin{align}\label{eq:almost_there}
u(s) = \tilde{c}_1 h_1^-(s) + \tilde{c}_2 h_2^-(s) - h_1^-(s)\int_0^s h_2^-(t) f(t)\,dt + h_2^-(s)\int_0^s h_1^{-}(t)f(t)\,dt,
\end{align}
where $|f(s)| = O(1)$. By itself this says that $u(s) = d_0 + d_1 s^{1/n} + O(s^{(n+1)/n})$, which is good but not enough. We now use this information to prove that $|f(s)| = O(s^{1/n})$. (Once we have this, we can feed it back into \eqref{eq:almost_there} to get that the two inhomogeneous terms are in fact $O(s^{(n+2)/n})$, and this is the pointwise part of the desired estimate \eqref{eq:remainder}.) Here we need to be a bit more careful than before.

Consider again \eqref{eq:est_inhomog_2}. The third term is negligible because $|u(s)| = O(1)$. For the first and second term, \eqref{eq:der_bounds_u} and its proof tell us nothing new because these arguments are based on upper bounds and the upper bound $|u(s)| = O(1)$ cannot be improved. However, we now have enough information to use \eqref{eq:almost_there} directly. In fact, we can simply take one derivative of \eqref{eq:almost_there} and use $|f(s)| = O(1)$ to bound $s |u'(s)| = O(s^{1/n})$. Then, from the ODE $L_\infty u = 0$,
\begin{align}
s^2 |u''(s)| \leq O(s)|u'(s)| + O(s^{\frac{n+1}{n}})|u(s)| + O(s^{\frac{n+1}{n}})|f(s)| = O(s^\frac{1}{n}),
\end{align}
using what we already know about $u',u,f$. Thus, from \eqref{eq:est_inhomog_2}, $|f(s)| = O(s^{1/n})$. By feeding this back into \eqref{eq:almost_there}, $|R(s)| = O(s^{(n+2)/n})$ in \eqref{eq:remainder}. This is the pointwise part of the desired expansion (3).

 We still need to prove the derivative estimates of $R(s)$ in \eqref{eq:remainder} (as, without these, we would not be able to plug the expansion of $\hat{u}$ back into the ODE and deduce the sign of $C_1,C_2$, which is crucial). First, $s|R'(s)| = O(s^{(n+2)/n})$ is easily proved by differentiating the inhomogeneous terms in \eqref{eq:almost_there} and using that $|f(s)| = O(s^{1/n})$. Notice that the terms obtained by letting $d/ds$ act on the integral signs cancel out. Thus, when we differentiate one more time to estimate $R''$, there is no need for a bound on $f'$, and $|f(s)| = O(s^{1/n})$ implies $s^2|R''(s)| = O(s^{(n+2)/n})$. So (3) is proved for a general $u$.
\end{proof}

\subsubsection{Solving \texorpdfstring{$L_\infty \hat{v} = \hat{u}$}{} with a Neumann boundary condition on the left}\label{sss:neumann}

The inhomogeneous ODE $L_\infty v = \hat{u}$ has a $2$-dimensional affine space of solutions on $(0,1)$. For the sake of transplanting $v$ to the approximating manifolds, we again require that $v(s) \to const$ as $s \to 0^+$ and $s \to 1^-$. As in Section \ref{sss:ode_liou} this singles out a $1$-dimensional affine subspace of solutions, any two members of which differ by a scalar multiple of $\hat{u}$. What makes things work in the end is the fact that there exists a canonical element $\hat{v}$ of this $1$-dimensional affine subspace that satisfies a \emph{Neumann} boundary condition at $s = 0$. We also need to know that $\hat{v} \leq 0$ and $\hat{v}(0) \neq 0$, but the precise value of $\hat{v}(0)$ is irrelevant.

\begin{lemma}\label{l:construct_vhat}
There exists a solution $\hat{v} \leq 0$ to the inhomogeneous ODE $L_\infty\hat{v} = \hat{u}$ on $(0,1)$ such that
\begin{align}
|\hat{v}(1-\eta)| + \eta|\hat{v}'(1-\eta)| + \eta^2|\hat{v}''(1-\eta)| = O(\eta^{n+1})\;\,\text{as}\;\,\eta \to 0^+,\label{eq:guess_who}\\
\label{eq:sol_inhomog_ode}
\hat{v}(s) = C_0 + C_2 s^{\frac{n+1}{n}} + R(s), \;\,|R(s)| + s|R'(s)| + s^2|R''(s)| = O(s^{\frac{n+2}{n}})\;\,\text{as}\;\,s \to 0^+,
\end{align}
where $C_0, C_2$ are constants and $C_0 \neq 0$.
\end{lemma}

This satisfies a Neumann condition in the sense that the expected term $C_1 s^{1/n}$ in \eqref{eq:sol_inhomog_ode} vanishes.

\begin{proof}
We begin by constructing one particular solution $v$ to the ODE, which may not be the one we seek. For this we proceed as in the proof of Lemma \ref{l:construct_uhat}, i.e., by solving Dirichlet rather than Neumann problems. Thus, for any fixed $\varepsilon \in (0,1)$ let $v_\varepsilon: [\varepsilon,1-\varepsilon] \to \mathbb{R}$ be the unique solution to $L_\infty v_\varepsilon = \hat{u}|_{[\varepsilon,1-\varepsilon]}$ with $v(\varepsilon) = -1$ and $v(1-\varepsilon) = 0$. At an interior maximum we must have that $-v_\varepsilon \geq L_\infty v_\varepsilon = \hat{u}$, i.e., $v_\varepsilon \leq -\hat{u} \leq 0$, and at an interior minimum we must have that $-v_\varepsilon \leq L_\infty v_\varepsilon = \hat{u}$, i.e., $v_\varepsilon \geq -\hat{u} \geq -1$. Thus, by comparison with the boundary values, $-1 \leq v_\varepsilon \leq 0$. Up to a subsequence, we can now pass to a limit $v: (0,1) \to \mathbb{R}$ locally smoothly. Clearly $L_\infty v = \hat{u}$ and $-1 \leq v \leq 0$.

A barrier argument as in the proof of Lemma \ref{l:construct_uhat} yields $\lim_{s \to 0^+} v(s) = -1$ and $\lim_{s \to 1^-} v(s) = 0$. Indeed, consider the function $\hat{u}_{sub}(s) = 1 - C_1 s^{1/n} + C_2 s^{(n+1)/n}$ from the previous proof, which satisfies $\hat{u}_{sub}(s_*) = 0$ and $L_\infty\hat{u}_{sub} \geq 0$ on $(0,s_*]$ for some small dimensional $s_* \in (0,1)$. Then $v^{sup} := -\hat{u}_{sub}$ trivially satisfies $L_\infty v^{sup} \leq 0 \leq \hat{u}$, so we can use $v^{sup}$ as an upper barrier for $v_\varepsilon$ on $[\varepsilon,s_*]$, proving that $\lim_{s\to 0^+} v(s) = -1$. For $s \to 1^-$ almost the same trick works. From the proof of Lemma \ref{l:construct_uhat} we know that $\hat{u}^{sup}(1-\eta) = C_1 \eta^{n+1} - C_2\eta^{2n+2}$ satisfies $\hat{u}^{sup}(1-\eta_*) = 1$ and $L_\infty\hat{u}^{sup} \leq 0$ on $[1-\eta_*,1)$ for a small universal $\eta_* \in (0,1)$. Then we would like to use $v_{sub} := -\hat{u}^{sup}$ as a lower barrier for $v_\varepsilon$ on $[1-\eta_*,1-\epsilon]$. This does not work immediately because we only have $L_\infty v_{sub} \geq 0$ but not $L_\infty v_{sub} \geq \hat{u}$. However, we can improve the construction of $\hat{u}^{sup}$ slightly by choosing $\eta_*$ small enough so that $(2C_*' + C_*''')\eta_*^{n+1}$ $\leq$ $(1/2)(\gamma(\gamma-n)/(n+1)-1)$ in the previous proof, hence $L_\infty v_{sub} \geq (1/2)(\gamma(\gamma-n)/(n+1)-1) > 0$ on $[1-\eta_*,1)$. Since we already know that $\hat{u}(1-\eta) = O(\eta^{n+1})$ with a dimensional constant, this means we can arrange that $L_\infty v_{sub} \geq \hat{u}$ on $[1-\eta_*,1)$ by making $\eta_*$ even smaller if necessary.

To sum up, we have constructed a solution $v: (0,1) \to \mathbb{R}$ to $L_\infty v = \hat{u}$ with $-1\leq v \leq 0$ as well as $\lim_{s\to 0^+}v(s) = -1$ and $\lim_{s \to 1^-} v(s) = 0$. In fact, $v(1-\eta) = O(\eta^{n+1})$ as $\eta\to 0^+$, and by using Lemma \ref{l:construct_uhat}(2) and the interpolation argument proving \eqref{eq:rough_estimate} we can easily upgrade this to \eqref{eq:guess_who}. Also as in the proof of Lemma \ref{l:construct_uhat}, we can now prove that $v(s) = -1 + C_1s^{1/n} + C_2 s^{(n+1)/n} + R(s)$, where $R(s) = O(s^{(n+2)/n})$ with two derivatives as in \eqref{eq:remainder} or \eqref{eq:sol_inhomog_ode}. In fact, nothing needs to be changed in the proof except that $\hat{u}(s) = 1 + O(s^{1/n})$ needs to be added to $f$.

Having proved this expansion of $v$, we note that if $C_1 = 0$, then we are done with $\hat{v} := v$. Otherwise we can replace $v$ by $\hat{v} := v + \lambda \hat{u}$ for some $\lambda \neq 0$ to make the $s^{1/n}$ term vanish. Then we also need to make sure that $\hat{v} \leq 0$ and $\hat{v}(0) \neq 0$. By the maximum principle, it is enough to prove that $\hat{v}(0) < 0$, which we do by contradiction. If $\hat{v}(0) \geq 0$, then the constant term of $\hat{v}(s)$ is nonnegative and the $s^{1/n}$ term vanishes. Thus, in order to have $L_\infty \hat{v}(s) = \hat{u}(s) = 1 + O(s^{1/n})$, we need $C_2 > 0$ (again thanks to the fact that $R(s)$ is estimated with two derivatives). Thus, $\hat{v}(s) > \hat{v}(0)$ for $0 < s \ll 1$. But this contradicts the inequality $\hat{v} \leq \hat{v}(0)$, which holds by the maximum principle.
\end{proof}

\subsubsection{Definition of the obstruction function}

For all $0 < |\sigma| \ll 1$ we first choose an $s_{*,\sigma} \in (0,\frac{1}{5}]$ with
\begin{align}
\max\left\{3|T|^{\alpha-1},3\frac{\tau}{T}\right\} \leq s_{*,\sigma} \to 0\;\,\text{as}\;\,\sigma \to 0.\label{eq:pretty_obvious}
\end{align}
Then we choose a smooth function $\chi_\sigma: (0,1) \to [0,1]$ such that
\begin{align}\label{eq:define_chi}
\chi_\sigma(s) = \left.\begin{cases}\;0 &\text{for}\;\,s \in (0,s_{*,\sigma}]\\
\;1 &\text{for}\;\,s \in [2s_{*,\sigma},1)
\end{cases}\right\}
\;\text{and}\;\,s_{*,\sigma}|\chi_\sigma'| + s_{*,\sigma}^2|\chi_\sigma''| \leq C
\end{align}
for some constant $C$ independent of $\sigma$. Now recall Convention \ref{conv}: We identify $\mathfrak{R}_4 \subset \mathcal{X}_\sigma$ with the set $\{T + 2T_0 < t < 2\tau\} \subset TY_0$ using the diffeomorphism $\Phi_\sigma^{-1} \circ \Psi_\sigma$. In this way we can view (functions of) $s,\eta$ as functions on $\mathfrak{R}_4$. Thus, the following definition makes sense.

\begin{definition}\label{def:obstr}
We define the obstruction function $\hat{u}_\sigma \in C^\infty(\mathcal{X}_\sigma)$ via $\hat{u}_\sigma := L_\sigma \hat{v}_\sigma$, where
\begin{align}\label{eq:def:hatvsigma}
    \hat v_\sigma := [\chi_\sigma \hat v + (1 - \chi_\sigma) \hat v(0)]\chi_\sigma(1-\cdot) \in C^\infty(\mathcal{X}_\sigma).
\end{align}
Note that $\hat{v}_\sigma$ actually extends smoothly from $\mathfrak{R}_4$ to all of $\mathcal{X}_\sigma$ because it is constant equal to $\hat{v}(0)$ for $s < s_{*,\sigma}$ and zero for $s > 1-s_{*,\sigma}$. By \eqref{eq:pretty_obvious} this cutoff happens strictly within the boundaries of $\mathfrak{R}_4$. Also, $\hat{u}_\sigma$ is then constant equal to $-\hat{v}(0)$ to the left of $\mathfrak{R}_4$ and zero to the right of $\mathfrak{R}_4$.
\end{definition}

\begin{lemma}\label{l:conv_uhat}
The obstruction function $\hat{u}_\sigma$ satisfies the following properties:
\begin{itemize}
    \item[$(1)$] Fix $s_1 < s_2$ in $(0,1)$. Lift $\hat{u}_\sigma$ and $\hat{u} \circ s$ to functions on the domain of the chart \eqref{eq-quasi-10} on the universal cover of  $\{s_1 < s < s_2\}$. Then $\hat{u}_\sigma \to \hat{u} \circ s$ uniformly as $\sigma \to 0$.
    \item[$(2)$] There exists a uniform constant $C$ such that $|\hat{u}_{\sigma}| \leq C\eta^{3}$ as functions on $\mathfrak{R}_4$ for all $\sigma$.
\end{itemize}
\end{lemma}

The main content of (2) is that $\hat{u}_\sigma$ remains bounded independently of $\sigma$ at the left boundary of $\mathfrak{R}_4$. This is not at all obvious due to the poor regularity of $\hat{u}(s)$ and $\hat{v}(s)$ as $s \to 0$. Here we crucially use the Neumann property of $\hat{v}$. (2) is also the key to the proof of our main result, Proposition \ref{p:construct_obstruct}.

\begin{proof}[Proof of Lemma \ref{l:conv_uhat}]
For item (1): From the proof of Lemma \ref{l:modeloperators} we know that for $|\sigma|\ll 1$,
\begin{align}\label{eq:345}
\begin{split}
        \hat{u}_\sigma &= [(\Delta_{\omega_T} - {\rm Id}) + (\Delta_{\omega_{glue,\sigma}}-\Delta_{\omega_T})]\hat{v}_\sigma \\
        &= (\Delta_{\omega_T} - {\rm Id})\hat{v}_\sigma + O_{s_1,s_2}(e^{-\frac{1}{2}|T|^\alpha}) \circledast (\check\partial \hat{v}_\sigma, \check\partial^2\hat{v}_\sigma)\\
        &= L_\infty \hat{v} + [O_{s_1,s_2}(|b|^{\frac{1}{3}}|\tau|) + O_{s_1,s_2}(e^{-\frac{1}{2}|T|^\alpha})] \circledast (\hat{v}',\hat{v}'').
        \end{split}
\end{align}
The point is that $\hat{v}_\sigma$ is radial in the chart \eqref{eq-quasi-10} and for $|\sigma| \ll 1$ we have $2s_{*,\sigma} < s_1 < s_2 < 1 - 2s_{*,\sigma}$, so the cutoffs in \eqref{eq:def:hatvsigma} are irrelevant and $\hat{v}_\sigma = \hat{v}$. Thus, according to the definition of $L_\infty$ as the first line of \eqref{eq:operator_in_quasi_coords}, $(\Delta_{\omega_T}-{\rm Id})\hat{v}_\sigma$ is equal to $L_\infty \hat{v}$ modulo the operator error $O_\delta(|b|^{1/3}|\tau|)\hat{v}_\sigma$ from the second line of \eqref{eq:operator_in_quasi_coords}. Moreover, $\check\partial$ reduces to $d/ds$. This explains \eqref{eq:345}. Of course, $L_\infty\hat{v} = \hat{u}$, and the rest of the third line of \eqref{eq:345} obviously goes to zero uniformly as $\sigma \to 0$.

For item (2): In this proof we prefer to use $n$-dimensional notation for clarity. We fix a sufficiently small but universal $\delta > 0$ and distinguish three subregions of $\mathfrak{R}_4$ as follows. \medskip\

\noindent \emph{Region $(\textup{a})$}: $s = 1-\eta \in [\delta,1-\delta]$. In this region, item (2) simply states that $\hat{u}_\sigma$ is bounded uniformly, independently of $\sigma$, and this statement is trivial from \eqref{eq:345} for $s_1 = \delta$ and $s_2 = 1-\delta$.\medskip\

\noindent \emph{Region $(\textup{b})$}: $\eta \in (0,\delta]$. Here the goal is to prove that $|\hat{u}_\sigma| \leq C\eta^{n+1}$ with $C$ independent of $\sigma$. We want to argue as in \eqref{eq:345} but $s_2$ is not bounded away from $1$, so we need to make the following changes:\smallskip

$\bullet$ Use the formula $\hat{v}_\sigma = \chi_\sigma(1-\cdot)\hat{v}$ instead of $\hat{v}_\sigma = \hat{v}$.\smallskip

$\bullet$ Instead of using the charts \eqref{eq-quasi-10}, which degenerate as $s_2 \to 1$, we estimate the difference between $\Delta_{\omega_{glue,\sigma}}$ and $\Delta_{\omega_T}$ covariantly. Recall from Convention \ref{conv} and Lemma \ref{Jerror} that $g_{glue,\sigma}$ and $g_T$ differ from each other by $O(e^{-(1/2-\epsilon)(t-T)}) = O(e^{-(1/2)|T|^\alpha})$ uniformly over $\mathfrak{R}_4$, including derivatives. Also, from Proposition \ref{prop-T-c}, $g_T$ and $g_{cusp}$ are uniformly equivalent over $\{0 < \eta \leq \delta\}$, including derivatives because both are Einstein. Thus, we can estimate
\begin{align}\label{eq:2b1}
\begin{split}
|(\Delta_{\omega_{glue,\sigma}}-\Delta_{\omega_T})\hat{v}_\sigma| \leq O_\delta(e^{-\frac{1}{2}|T|^\alpha})\sum_{k=1}^2 |\nabla^k_{\omega_{cusp}}\hat{v}_\sigma|_{\omega_{cusp}}
&\leq O_\delta(e^{-\frac{1}{2}|T|^\alpha})(\eta^2|\hat{v}_\sigma''| + \eta|\hat{v}_\sigma'|).
\end{split}
\end{align}

$\bullet$ Note that the operator error $O_\delta(|b|^{1/(n+1)}|\tau|)$ from \eqref{eq:operator_in_quasi_coords} only depends on an upper bound for $\eta$. Thus, using \eqref{eq:operator_in_quasi_coords} and the fact that $\hat{v}_\sigma$ is radial, we have on all of $\{0 < \eta \leq \delta\}$ that
\begin{align}\label{eq:2b2}
\begin{split}
|(\Delta_{\omega_T} - {\rm Id})\hat{v}_\sigma| = |(L_\infty  + O_\delta(|b|^{\frac{1}{n+1}}|\tau|))\hat{v}_\sigma|\leq O_\delta(1) (\eta^2|\hat{v}_\sigma''| + \eta|\hat{v}_\sigma'| + |\hat{v}_\sigma|).
\end{split}
\end{align}

\noindent The desired estimate of $\hat{u}_\sigma$ now easily follows from the estimates \eqref{eq:guess_who} for $\hat{v}$ and \eqref{eq:define_chi} for $\chi_\sigma$.\medskip\

\noindent \emph{Region $(\textup{c})$}: $s \in (0,\delta]$. Here we prove that $|\hat{u}_\sigma| \leq C$ with $C$ independent of $\sigma$. We again follow a similar pattern but now need to use \eqref{eq:666} instead of \eqref{eq-g-T-eta} or \eqref{eq:operator_in_quasi_coords}. Thus, we proceed as follows:\smallskip

$\bullet$ We use the formula $\hat{v}_\sigma = \chi_\sigma\hat{v} + (1-\chi_\sigma)\hat{v}(0)$.\smallskip

$\bullet$ The comparison of $g_{glue,\sigma}$ and $g_T$ is the same as in (b). Then we compare $\omega_T$ and $|b|^{1/2}\omega_{\mathcal{C},\sigma}$ using \eqref{eq-Psi-T-11}--\eqref{eq-Psi-T-12}. The comparison is uniform over $\{0 < s \leq \delta\}$ and each derivative costs at most a factor of $|T|^K (t-T)^K$ for some $K > 0$ as long as $t-T \geq 1$. Thus, 
\begin{align}\label{eq:c1}
\begin{split}
|(\Delta_{\omega_{glue,\sigma}}-\Delta_{\omega_T})\hat{v}_\sigma| \leq O_\delta(e^{-\frac{1}{2}|T|^\alpha}|T|^{2K})(s^{\frac{n-1}{n}}|\hat{v}_\sigma''| + s^{-\frac{1}{n}}|\hat{v}_\sigma'|).
\end{split}
\end{align}

$\bullet$ The operator error $O_\delta(|b|^{1/(n+1)}|\tau|)$ from \eqref{eq:666} only depends on an upper bound for $s$. Applying \eqref{eq:666}, we therefore get on all of $\{0 < s \leq \delta\}$ that
\begin{align}\label{eq:c2}
\begin{split}
|(\Delta_{\omega_T} - {\rm Id})\hat{v}_\sigma| = |(L_\infty  + O_\delta(|b|^{\frac{1}{n+1}}|\tau|))\hat{v}_\sigma|\leq O_\delta(1) (s^{\frac{n-1}{n}}|\hat{v}_\sigma''| + s^{-\frac{1}{n}}|\hat{v}_\sigma'| + |\hat{v}_\sigma|).
\end{split}
\end{align}
To estimate the right-hand sides of \eqref{eq:c1}--\eqref{eq:c2} we need to be a little more careful than in (b). First notice that the $|\hat{v}_\sigma|$ term in \eqref{eq:c2} is obviously uniformly bounded. Next, $\hat{v}_\sigma = \chi_\sigma(\hat{v} - \hat{v}(0)) + \hat{v}(0)$, so in the remaining terms with $\hat{v}_\sigma', \hat{v}_\sigma''$ we can replace $\hat{v}_\sigma$ by $\chi_\sigma(\hat{v} - \hat{v}(0))$. Using \eqref{eq:define_chi} to estimate $\chi_\sigma$ and \eqref{eq:sol_inhomog_ode} to estimate $\hat{v} - \hat{v}(0)$, we get the desired uniform bound. However, at this point it is crucially important that $\hat{v}$ satisfies a Neumann condition at $s = 0$, i.e., that the expansion of $\hat{v} - \hat{v}(0)$ starts with $s^{(n+1)/n}$ rather than $s^{1/n}$. (Otherwise the best bound we could get is $|\hat{u}_\sigma| \leq Cs^{-1}$, which is not $L^1_{loc}$, destroying the dominated convergence argument in the proof of Proposition \ref{p:construct_obstruct}.) 
\end{proof}

\subsubsection{Proof of Proposition \ref{p:construct_obstruct}}\label{sss:final_boss}

As $w_{\sigma_i}(x_i) |\psi_i| \leq w_{\sigma_i}$ on $\mathcal{X}_{\sigma_i}$ and $x_i$ lies in $\mathfrak{R}_4$ with $s(x_i) \to c \in (0,1)$, we can deduce from the definition, \eqref{eq:def_weight}, of the weight function $w_{\sigma_i}$ that
\begin{align}\label{eq:global_weighted_bounds}
    |\psi_i| \leq \begin{cases}
    C &\text{on}\;\,\mathfrak{R}_1, \\
    C \Psi_\sigma^*(\Phi_\sigma^{-1})^* s^{-\delta} &\text{on}\;\,\mathfrak{R}_2 \cup \mathfrak{R}_3 \cup \mathfrak{R}_4 \cup \mathfrak{R}_5 \cup \mathfrak{R}_6,\\
    C|T_i|^\delta &\text{on}\;\,\mathfrak{R}_7.
    \end{cases} 
\end{align}
Thus, RHS\eqref{eq:limit_orth} exists because \eqref{eq:global_weighted_bounds} implies $|\psi_\infty(s)| \leq C s^{-\delta}$ for all $s \in (0,1)$, and $|\hat{u}(1-\eta)| \leq C\eta^{3}$ from Lemma \ref{l:construct_uhat} and $|\mu_\infty(1-\eta)| \leq C \eta^{-3}$ from Lemma \ref{l:volumeform} for all $\eta \in (0,1)$. We now decompose $\mathcal{X}_{\sigma_i}$ into three regions and analyze their contributions to LHS\eqref{eq:limit_orth}.\medskip\

\noindent \emph{Region} $\mathfrak{R}_1 \cup \mathfrak{R}_2 \cup \mathfrak{R}_3$: Here we obviously have $\hat{u}_{\sigma_i}=0$.\medskip\

\noindent \emph{Region} $\mathfrak{R}_5 \cup \mathfrak{R}_6 \cup \mathfrak{R}_7$: Here $\hat{u}_{\sigma_i}$ is constant equal to $-\hat{v}(0)$. In the unrescaled Tian-Yau space, volume grows like $\log h$ as $h \to \infty$. Thus, the volume of $\mathfrak{R}_5 \cup \mathfrak{R}_6 \cup \mathfrak{R}_7$ with respect to $\omega_{glue,\sigma_i}$ is bounded by $C|b_i| |T_i|^\alpha$. Recalling that $|b_i| \sim |T_i|^{-3}$ and using the bound $|\psi_i| \leq C|T_i|^\delta$ from \eqref{eq:global_weighted_bounds}, we get that the total contribution of this region to LHS\eqref{eq:limit_orth} goes to zero provided that $\alpha + \delta < 1$. This is consistent with our standing convention that $\alpha,\delta$ are always chosen arbitrarily close to zero.\medskip

\noindent \emph{Region $\mathfrak{R}_4$}: By \eqref{eq:volumeform_T}, the volume form of $\omega_{T_i}$ satisfies for all $\eta = 1-s \in (0,1)$ that
\begin{align}
\omega_{T_i}^2|_{t=\eta T_i} = |T_i|^{-2} \mu_{T_i}(1-\eta)\, ds \wedge d\theta \wedge d{\rm Vol}_E.
\end{align}
Thus, by Fubini, and using Lemma \ref{Jerror} to compare the two volume forms,
    \begin{align}\label{eq:yetanotherone}
    |T_i|^2\int_{\mathfrak{R}_4}\psi_i \hat{u}_{\sigma_i}\,\omega_{glue,\sigma_i}^2 = 2\pi {\rm Vol}(E)\int_{2|T_i|^{\alpha-1}}^{1-2\frac{\tau_i}{T_i}} \overline{\psi_i \hat{u}_{\sigma_i}}(s)\,\mu_{T_i}(s)(1 + O(e^{-\frac{1}{2}|T_i|^\alpha}))\,ds,
    \end{align}
    where the bar denotes the average over the relevant level set of $s$ with respect to the fixed volume form $d\theta \wedge d{\rm Vol}_E$. By assumption and by Lemma \ref{l:conv_uhat}, $\psi_i \to \phi_\infty$ and $\hat{u}_{\sigma_i} \to \hat{u}$ uniformly on every fixed level set of $s$. Thus, we obviously have that $\overline{\psi_i \hat{u}_i}(s) \to \phi_\infty(s)\hat{u}(s)$ for every fixed $s$. Thus, the integrand on the right-hand side of \eqref{eq:yetanotherone} converges pointwise to $\psi_\infty(s)\hat{u}(s)\,\mu_\infty(s)$. It is also uniformly bounded by $Cs^{-\delta} \cdot C\eta^{3} \cdot C\eta^{-3} \leq Cs^{-\delta}$ thanks to \eqref{eq:global_weighted_bounds}, Lemma \ref{l:conv_uhat} and Lemma \ref{l:volumeform}. Since $Cs^{-\delta}$ is integrable on $(0,1)$, dominated convergence implies that \eqref{eq:yetanotherone} converges to the right-hand side of \eqref{eq:limit_orth}.
\hfill $\Box$

\subsection{Statement of the uniform estimate modulo obstructions and set-up of the proof}\label{ss:uniform-statement}
         
Here we state the uniform weighted Hölder estimate of the inverse linearized operator modulo obstructions and explain the strategy of the proof. This is the standard blowup-and-contradiction scheme common to this type of problem, and the aim is to obtain a contradiction to some Liouville theorem in each of the 7 cases. In Sections \ref{ss:4:3}--\ref{ss:4:7} we carry out the details of this scheme in each case.

\begin{definition}
Given the obstruction function $\hat{u}_\sigma = L_\sigma \hat{v}_\sigma$ from Definition \ref{def:obstr}, we define
\begin{align}
\langle \hat{u}_\sigma \rangle := \mathbb{R} \cdot \hat{u}_\sigma\;\,\text{and}\;\,\langle \hat{v}_\sigma \rangle := \mathbb{R} \cdot \hat{v}_\sigma
\end{align}
as $1$-dimensional subspaces of $C^\infty(\mathcal{X}_\sigma)$. We also write
\begin{align}
    L_\sigma^\perp: \langle \hat{u}_\sigma \rangle^\perp \to \langle \hat v_\sigma \rangle^\perp
\end{align}
to denote the restriction of $L_\sigma = \Delta_{\omega_{glue,\sigma}} - {\rm Id}$ to the orthogonal complements of $\langle \hat{u}_\sigma\rangle$ resp. $\langle \hat{v}_\sigma\rangle$ inside $C^{2,\bar\alpha}(\mathcal{X}_\sigma)$ resp. $C^{0,\bar\alpha}(\mathcal{X}_\sigma)$ with respect to the  $L^2(\mathcal{X}_\sigma, \omega_{glue, \sigma})$-inner product for any $\bar\alpha \in (0,1)$. By basic elliptic theory, $L_\sigma^\perp$ is properly defined and is an isomorphism of Banach spaces.
\end{definition}

\begin{theorem}\label{t:uniform_Linfty}
For all $\bar\alpha \in (0,1)$ there exists a constant $C(\bar\alpha)$ independent of $\sigma$ such that
\begin{align}
\|\phi\|_{C^{2,\bar\alpha}_w} \leq C(\bar\alpha)\|L_\sigma^\perp\phi\|_{C_{\tilde w}^{0,\bar\alpha}}
\end{align}
for all $0<|\sigma|\ll 1$ and for all $\phi \in \langle \hat{u}_\sigma \rangle^\perp \subset C^{2,\bar \alpha}(\mathcal{X}_\sigma)$.
\end{theorem}

Using the standard weighted Schauder estimates \eqref{eq:weighted_schauder}, the theorem reduces to proving that
\begin{align}
    \|\phi\|_{C_w^0} \leq C\|L_\sigma^\perp\phi\|_{C_{\tilde w}^0}.
\end{align}
We prove this estimate by contradiction in the rest of Section \ref{sec-L-inf}, starting now.

Assume that there is a sequence of functions $\phi_i \in \langle \hat{u}_{\sigma_i} \rangle^\perp \subset C^{2,\bar\alpha} (\mathcal{X}_{\sigma_i})$ such that $\sigma_i \to 0$ and
\begin{align}
1 = \|\phi_i\|_{C^0_w} > i \|L_{\sigma_i} \phi_i\|_{C^0_{\tilde w}}.
\end{align}   
Let $x_i \in \mathcal{X}_{\sigma_i}$ such that $|\phi_i(x_i)| = w_{\sigma_i}(x_i)$. Set
    \begin{align}\label{eq-psi}
    \psi_i := \frac{\phi_i}{w_{\sigma_i}(x_i)}.
    \end{align}
    Then $\psi_i(x_i)=1$ and 
    \begin{align}
    |\psi_i|\leq \frac{w_{\sigma_i}}{w_{\sigma_i}(x_i)}.
    \end{align}
    
\begin{ass}\label{ass:I}
It is possible to choose scaling factors $\mu_i \geq 1$ such that if
\begin{align}\label{eq-g-tilde}
\tilde{g}_i := \mu_i^2 g_{glue,\sigma_i},
\end{align}
then the spaces $(\mathcal{X}_{\sigma_i}, \tilde{g}_i, x_i)$ have a pointed Cheeger-Gromov limit (possibly collapsed) and the functions $w_{\sigma_i}/w_{\sigma_i(x_i)}$ have a locally uniform limit $w_\infty$ under this pointed Cheeger-Gromov convergence.
\end{ass}

Then the rescaled operators
\begin{align}
\tilde{L}_i := \Delta_{\tilde{g}_i} -\mu_i^{-2} {\rm Id} = \mu_i^{-2} (\Delta_{g_{\sigma_i}} - {\rm Id})=\mu_i^{-2} L_{\sigma_i}
\end{align}    
satisfy
\begin{align}\label{eq-Lgi-psi}
|\tilde{L}_i\psi_i| =\mu_i^{-2} \frac{|L_{\sigma_i} \phi_i|}{w_{\sigma_i}(x_i)} < \mu_i^{-2}\frac{1}{i} \frac{\tilde w_{\sigma_i} }{w_{\sigma_i}(x_i)}.
\end{align}
This will converge to $0$ locally uniformly if we additionally make the following assumption:

\begin{ass}\label{ass:II}
The sequence
\begin{align}
\mu_i^{-2}\frac{ \tilde w_{\sigma_i}}{w_{\sigma_i}(x_i)}
\end{align}
is locally uniformly bounded under the Cheeger-Gromov limit of Assumption \ref{ass:I}.
\end{ass}
    
In this situation, we can pass $\psi_i$ to a subsequential limit $\psi_\infty$ weakly in $W^{2,p}_{loc}$ and strongly in $C^{1,\beta}_{loc}$ for all $p,\beta$, where $\psi_\infty$ is smooth and satisfies an elliptic equation $\tilde{L}_\infty \psi_\infty = 0$. This is clear from standard elliptic compactness and regularity if there is no collapsing in the Cheeger-Gromov limit, and in this case $\tilde L_\infty = \Delta_{\tilde{g}_\infty} - \mu_\infty^{-2} \, {\rm Id}$. In the collapsing cases, we first need to pass to a local universal cover and work in quasi-coordinates. Then $\psi_\infty$ will be radial and $\tilde L_\infty$ will reduce to an ODE operator. See the proof of Lemma \ref{l:modeloperators} for the details of this argument in the most complicated case (region $\mathfrak{R}_4$).

\begin{ass}\label{ass:III}
    There is a Liouville theorem: $\tilde L_{\infty} \psi_\infty =0$ and $|\psi_\infty|\leq w_\infty$ imply $\psi_\infty=0.$
\end{ass}

This contradicts the property $\psi_\infty (x_\infty) =1$, which holds due to strong $C^{1,\beta}_{loc}$ convergence.

In the following sections, we will show that up to passing to subsequences, the above Assumptions \ref{ass:I}, \ref{ass:II} and \ref{ass:III} are indeed satisfied in our situation. The resulting contradictions prove Theorem \ref{t:uniform_Linfty}.

\subsection{GH limit and Liouville on the Kähler-Einstein building block (region \texorpdfstring{$\mathfrak{R}_1$}{})}

Assume that after passing to a subsequence we have for all $i$ that $x_i \in \mathfrak{R}_1$, or $x_i \in \mathfrak{R}_2$ with $t(x_i)$ uniformly bounded from below. Then we set $\mu_i := 1$. The pointed limit space of the sequence $(\mathcal{X}_{\sigma_i}, \tilde{g}_i, x_i)$ is the complete cuspidal Kähler-Einstein manifold $(\mathcal{X}_0^{reg}, g_{KE,0})$. We now verify our three assumptions.

\ref{ass:I}: For simplicity we only consider the case $x,x_i \not\in \mathfrak{R}_1$, so that $t(x), t(x_i)$ are actually defined. They are then uniformly bounded above and below, so
\begin{align}
\begin{split}\label{eq-tends-1}
\frac{w_{\sigma_i}(x)}{w_{\sigma_i}(x_i)} &=\frac{(t(x)-T_i)^{-\delta}}{(t(x_i)-T_i)^{-\delta}}\rightarrow 1
\end{split}
\end{align}
locally uniformly. Hence $w_\infty =1$.

\ref{ass:II}: Again assuming $x,x_i \not\in \mathfrak{R}_1$ we have that
\begin{align}
    \mu_i^{-2}\frac{\tilde{w}_{\sigma_i}(x)}{w_{\sigma_i}(x_i)} = |b_i|^{-\frac{1}{2}}(t(x)-T_i)^{-\frac{3}{2}} \frac{w_{\sigma_i}(x)}{w_{\sigma_i}(x_i)} \to 1
\end{align}
locally uniformly because of \ref{ass:I} and because $|b_i| \sim |T_i|^{-3}$.

\ref{ass:III}: By the generalized maximum principle on complete Riemannian manifolds with Ricci curvature bounded below \cite[p.207, Cor 1]{CY2}, there is a sequence $y_i$ in $\mathcal X_0^{reg}$ such that $\lim_{i\rightarrow\infty} \psi_\infty(y_i)=\sup \psi_\infty$
and $\limsup_{i\rightarrow \infty} \Delta_{g_{KE,0}}\psi_\infty(y_i)\leq 0.$ 
From the equation $\tilde L_{\infty} \psi_\infty =0$ one has $\psi_\infty(y_i)=\Delta_{g_{KE,0}} \psi_\infty(y_i)$, hence $\sup\psi_\infty \leq 0$. Similarly, we can show $\inf \psi_\infty\geq 0$. Hence $\psi_\infty=0$.

\subsection{GH limit and Liouville on the cusp (regions \texorpdfstring{$\mathfrak{R}_2, \mathfrak{R}_3$}{})}\label{ss:cases:cusp}

\subsubsection{Region \texorpdfstring{$\mathfrak{R}_2\textup{:}$}{} The genuine cusp}

Here we assume that $x_i \in \mathfrak{R}_2$ for all $i$. Also, we assume that $x_i$ tends to have infinite distance from $\mathfrak{R}_1$ as well as from $\mathfrak{R}_3$. Precisely,
\begin{align}
t(x_i) \rightarrow -\infty, \quad t(x_i)/\tau_i \rightarrow 0.
\end{align}
We set $\mu_i:=1$. The Gromov-Hausdorff limit (better: the collapsed Cheeger-Gromov limit) in this case is a line. \ref{ass:I} and \ref{ass:II} work exactly as in the previous case because $|t(x_i)| = o(|\tau_i|) = o(|T_i|)$.

\ref{ass:III}: Proceeding as in the proof of Lemma \ref{l:modeloperators}, we obtain a smooth limit function $\psi_\infty(t)$ with
\begin{align}
\tilde{L}_{\infty}\psi_\infty = \frac{1}{3}\left( t^2\frac{d^2\psi_\infty}{d t^2} -t\frac{d\psi_\infty}{d t}\right) - \psi_\infty = 0.
\end{align}
Thus, $\psi_\infty(t) = a (-t)^{3}+b (-t)^{-1}$, where $a, b$ are constants.
$w_\infty=1$ rules out all possibilities. 

\subsubsection{Region \texorpdfstring{$\mathfrak{R}_3\textup{:}$}{} The green gluing region between the cusp and the new neck}

In this case, we consider the region within finite distance from $\mathfrak{R}_3$ with respect to $g_{glue,\sigma_i}$. More precisely, we assume that
\begin{align}
t(x_i) \rightarrow -\infty, \quad t(x_i)/\tau_i \rightarrow c \in (0, \infty).
\end{align}
Set $\mu_i:=1$ and apply Lemma \ref{CusptoNeckerror}, extended so that \eqref{eq-Ck-T-c} holds for $t \in [C\tau_i, C^{-1} \tau_i]$ given any fixed constant $C>1$.  The remainder of the discussion is similar to that in the case of region $\mathfrak{R}_2$.
         
\subsection{GH limit and Liouville on the Tian-Yau building block (region \texorpdfstring{$\mathfrak{R}_7$}{})}\label{ss:4:3}

Consider the case that the points $x_i$ stay on the Tian-Yau side, i.e., for all $i$ we either have that $x_i \in \mathfrak{R}_7$, or $x_i \in \mathfrak{R}_6$ and $t(x_i) - T_i = \log h(m_{\sigma_i}(x_i))$ remains uniformly bounded above. Set $\mu_i := |b_i|^{-1/4}$. Then $\tilde{g}_i =m_{\sigma_i}^* g_{TY_1}$ and the pointed limit space is $(TY_1,g_{TY_1})$ after applying the map $m_{\sigma_i}$. 

We now apply the map $m_{\sigma_i}$ without writing it explicitly. Thus, $x_i \in TY_1$ and $\log h(x_i) \leq N$.

For the three conditions, first, by taking a subsequence we can assume that $w_{\sigma_i}(x_i)\rightarrow c \in [N^{-\delta} ,1]$. Then $w_{\sigma_i}/w_{\sigma_i}(x_i)$ converges locally uniformly to $c^{-1}$ on $\mathfrak{R}_7$ and to $c^{-1}(\log h)^{-\delta}$ on $\mathfrak{R}_6$. Next,
\begin{align}
    \mu_i^{-2}\frac{\tilde w_{\sigma_i}}{w_{\sigma_i}(x_i)} = |b_i|^{\frac{1}{2}}\cdot \left(|b_i|^{\frac{1}{4}}(\log h)^{\frac{3}{4}}\right)^{-2} \cdot \frac{w_{\sigma_i}}{w_{\sigma_i}(x_i)}
\end{align}
is locally uniformly convergent by the previous step. Lastly, 
\begin{align}
    \tilde L_{\infty} \psi_\infty = \Delta_{g_{TY_1}} \psi_\infty = 0,\quad |\psi_\infty| \leq c^{-1} (\log h)^{-\delta}\;\,\text{as}\;\,h \to \infty.
\end{align}
Then the maximum principle implies that $\psi_\infty = 0$.

\subsection{GH limit and Liouville on the Tian-Yau end (regions \texorpdfstring{$\mathfrak{R}_6, \mathfrak{R}_5$}{})}\label{ss:cases:TYend}

\subsubsection{Region \texorpdfstring{$\mathfrak{R}_6\textup{:}$}{} The genuine Tian-Yau end}

In this case, we assume that the points $x_i$ lie strictly in the interior of region $\mathfrak{R}_6$ in the sense that 
\begin{align}
    t(x_i) - T_i \to \infty, \quad \rho(x_i)^{-1} |T_i|^{\frac{3}{4}\alpha} \rightarrow \infty,
\end{align}
where $\rho$ is a geometric distance function for the unrescaled Tian-Yau metric,
\begin{align}\label{eq:def_rho}
\rho(x) := (t(x) -T_i)^{\frac{3}{4}}.
\end{align}
Then we define our scaling factors by
\begin{align}
\mu_i := \rho(x_i)^{-1}|b_i|^{-\frac{1}{4}} \gg |T_i|^{\frac{3}{4}(1-\alpha)} \to \infty,
\end{align}
so that the pointed Gromov-Hausdorff limit is the tangent cone at infinity of the Tian-Yau space, i.e., a half-line $[-1,\infty)$ in the natural parametrization given by $\tilde{\rho} := (\rho - \rho(x_i))/\rho(x_i)$. Technically this is a Cheeger-Gromov limit only away from the endpoint $\tilde{\rho} = -1$ but this subtlety is irrelevant for us.

For properties \ref{ass:I} and \ref{ass:II} we compute
\begin{align}
\frac{w_{\sigma_i}}{w_{\sigma_i}(x_i)} &= \frac{(t - T_i)^{-\delta}}{(t(x_i)-T_i)^{-\delta}} = (1+\tilde \rho)^{-\frac{4}{3}\delta},\\
    \mu_i^{-2}\frac{\tilde w_{\sigma_i}}{w_{\sigma_i}(x_i)} &= \rho(x_i)^2|b_i|^\frac{1}{2} \cdot \left(|b_i|^{\frac{1}{4}}(t - T_i)^\frac{3}{4}\right)^{-2} \cdot \frac{w_{\sigma_i}}{w_{\sigma_i(x_i)}} = (1 + \tilde{\rho})^{-2-\frac{4}{3}\delta}.
\end{align}

For \ref{ass:III} we obtain a smooth limit function $\psi_\infty(\tilde\rho)$ of $\tilde{\rho} \in (-1,\infty)$ solving an ODE $\tilde{L}_\infty \psi_\infty = 0$ as in the proof of Lemma \ref{l:modeloperators}. Here we show only the computation of $\tilde{L}_\infty$. The key step is to find suitable coordinates on the universal cover of the annulus $\{\tilde{\rho}_1 < \tilde\rho < \tilde{\rho}_2\}$, where $\tilde\rho_1 < \tilde\rho_2$ in $(-1,\infty)$ are given. For this we follow our work in Lemma \ref{lem-Ca-coor}, although the $T_0$ of that lemma gets replaced by $\rho(x_i)^{4/3}$ and instead of $y = t-T_i\in [T_0,2T_0]$ we are now considering the range $y \in [a_1\rho(x_i)^{4/3}, a_2\rho(x_i)^{4/3}]$ with
$a_j = (1+\tilde{\rho}_j)^{4/3}$ ($j = 1,2$). We define $(\check z, \check w)$ and $(z, w)$ as in Lemma \ref{lem-Ca-coor}. In addition, we assume that $\varphi(z) = -|z|^2$ after scaling and translating $z$ if necessary. Then holomorphic quasi-coordinates for the scaled model metric $\mu_i^2 |b_i|^{1/2}\omega_{\mathcal{C},\sigma_i} = \rho(x_i)^{-2}\omega_{\mathcal{C},\sigma_i}$ are given by 
\begin{align}
 (\hat z, \hat w) := \biggl(\rho(x_i)^{-1}\check z, \rho(x_i)^{-1} \biggl(\check w -\rho(x_i)^{-\frac{1}{3}}\frac{T_i}{2}\biggr)\biggr) = 
\biggl(\rho(x_i)^{-\frac{2}{3}}  z, \rho(x_i)^{-\frac{4}{3}}\biggl(w-\frac{T_i}{2}\biggr)\biggr).
\end{align}
These are also (complex but slightly non-holomorphic) quasi-coordinates for our metric $\tilde{g}_i = \mu_i^2 g_{glue,\sigma_i}$ by Lemma \ref{l:psiJerror}. Thus, by \eqref{eq:formula666}, up to errors that decay exponentially as $i \to \infty$,
\begin{align}\begin{split}
 \tilde\omega_i \approx \tilde\omega_\infty = 
  \frac{1}{2} (1+\tilde\rho)^{-\frac{2}{3}} i(d \hat w + \overline{\hat{z}} \, d\hat z) \wedge
(d \overline {\hat w} + \hat{z}\, d \overline{{\hat z}})
+(1+\tilde\rho)^\frac{2}{3} id\hat z \wedge d\overline{\hat z}.
\end{split}
\end{align}
Now notice that, from \eqref{eq-t-w},  
\begin{align}
\hat{w}+\overline{\hat{w}} + |\hat{z}|^2 = (1 + \tilde\rho)^{\frac{4}{3}}.
\end{align}
To follow the proof of Lemma \ref{l:modeloperators} we would need to switch from $(\hat{z},\hat{w})$ to real coordinates, where one of these coordinates is $\tilde\rho$ and the other three parametrize the level sets of $\tilde\rho$ or of $\hat{w}+\overline{\hat{w}} + |\hat{z}|^2$, which are the orbits of an $\tilde\omega_\infty$-isometric action of the continuous Heisenberg group $H_3(\mathbb{R})$. For simplicity we skip this coordinate change and instead note that thanks to the $H_3(\mathbb{R})$-symmetry it is enough to work at $\hat{z} = 0$, where $\tilde\omega_\infty$ is diagonal in the coordinates $(\hat{z},\hat{w})$. Then, since $\mu_i\to \infty$,
\begin{align}\label{eq:formula}
\tilde{L}_\infty \psi_\infty = \Delta_{\tilde\omega_\infty}\psi_\infty = 2(1+\tilde\rho)^{\frac{2}{3}}(\psi_{\infty})_{\hat{w}\overline{\hat{w}}} +(1+\tilde\rho)^{-\frac{2}{3}}(\psi_\infty)_{\hat{z}\overline{\hat{z}}} = 0.
\end{align}
The trivial solution is $\psi_\infty(\tilde\rho) = const$, and working at $\hat{z}=0$ one checks that $\psi_\infty(\tilde\rho) = (1+\tilde\rho)^{2/3}$ is the second fundamental solution. Then $|\psi_\infty| \leq w_\infty = (1+\tilde\rho)^{-(4/3)\delta}$ implies $\psi_\infty = 0$, as desired.

\subsubsection{Region \texorpdfstring{$\mathfrak{R}_5\textup{:}$}{} The orange gluing region between the Tian-Yau end and the new neck}

Now assume either that the sequence $x_i$ is contained in region $\mathfrak{R}_5$, or that it stays within finite distance of $\mathfrak{R}_5$ from the Tian-Yau side in an appropriately rescaled metric. More precisely, we assume that
\begin{align}\label{eq-case6}
0\leq \rho(x_i)^{-1}((2|T_i|^\alpha)^\frac{3}{4} - \rho(x_i)) \leq C
\end{align}	
for some $C > 0$, where $\rho$ is defined as in \eqref{eq:def_rho}. Also as in the previous case we set
\begin{align}
    \mu_i :=  \rho(x_i)^{-1}|b_i|^{-\frac{1}{4}} \sim |T_i|^{\frac{3}{4}(1-\alpha)} \to \infty.
\end{align}
Using Proposition \ref{prop:estimate_of_E} to estimate the gluing errors, one then checks that the arguments go through as before. The essential point is that the pointed convergence behavior of the rescaled spaces does not change because for every fixed $N > 1$, due to the proof of Proposition \ref{prop:estimate_of_E}, we still have an excellent comparison of $g_{glue,\sigma_i}$ and $|b_i|^{1/2}g_{\mathcal{C},\sigma_i}$ in the range $t - T_i \in [T_{0,i},NT_{0,i}]$ (even though for $N > 2$ this set is not contained in the Tian-Yau side of $\mathcal{X}_{\sigma_i}$, i.e., in the region $\mathfrak{R}_5 \cup \mathfrak{R}_6 \cup \mathfrak{R}_7$). 

\subsection{GH limit and Liouville on the new neck (region \texorpdfstring{$\mathfrak{R}_4$}{})}\label{ss:4:7}

For $x_i\in \mathcal{X}_{\sigma_i}$ satisfying
	\begin{align}
	x_i \in \mathfrak{R}_4 \Longleftrightarrow T_i+2 T_{0,i}< t(x_i) =: t_i < \tau_i,
	\end{align}
	we define the quantity
\begin{align}\label{eq-F}
	F_i := e^{-\psi_{T_i}(t_i)-a}|b_i|.
	\end{align}
By \eqref{eq-psi-5} and by the monotonicity of $\psi_{T_i}$, this takes values in $(0,1)$. Thus, up to a subsequence,
\begin{align}
	F_i \rightarrow c \in [0,1].
\end{align}

\subsubsection{The subcase \texorpdfstring{$c = 1\textup{:}$}{c=1} Essentially the same as the Tian-Yau end} 

Then $\psi_{T_i} (t_i) - \log |b_i| \rightarrow - a$, or equivalently $\psi_{T_i}(t_i)-\psi_{T_i}(T_i)\rightarrow 0$.
As in \eqref{eq-psi-t-T},
\begin{align}\label{eq-psi-t-T-2}
\int_{0}^{\psi_{T_i}(t_i)-\psi_{T_i}(T_i)} (e^{s}-1 )^{-\frac{1}{3}}\,ds = \sqrt[3]{3}\,|b_i|^{\frac{1}{3}}(t_i-T_i),
\end{align}
so $t_i -T_i\rightarrow 0.$
Given any $A \in \mathbb{R}$, if $i \gg 1$ and $t = T_i +A(t_i -T_i)$,
then, as in \eqref{eq-Psi-t-T-T0},
\begin{align}
\psi_{T_i}(t)-\psi_{T_i}(T_i) = \mathfrak{c}|b_i|^{\frac{1}{2}}(t-T_i)^\frac{3}{2}+ O(|b_i||t-T_i|^3).
\end{align}
We then get a $C^{2}$ estimate as in \eqref{TY-neckError}. The rest of the proof is similar to the case of region $\mathfrak{R}_5$.
	
\subsubsection{The subcase \texorpdfstring{$c = 0\textup{:}$}{c=0} Essentially the same as the cusp}

If $c=0$, then $\psi_{T_i}(t_i)- \log |b_i| \rightarrow \infty$. By Proposition \ref{prop-psi-diff}, $\psi_{cusp}(t_i) - \log |b_i| \rightarrow \infty$. Thus
\begin{align}
	(-t_i)^{3}b_i \rightarrow 0,
	\end{align}
so for any finite constant $\varrho\in \mathbb{R}$, when $i$ is large, $\varrho t_i/ T_i$ is small.	By Proposition \ref{prop-T-c}, 
	\begin{align}
	\psi'_{T_i}(\varrho t_i) &= \psi_{cusp}' (\varrho t_i)  (1+O(|b_i|^\frac{1}{3}\tau_i) +O((\varrho t_i/ T_i)^{3} )),\\
\psi''_{T_i}(\varrho t_i) &=\psi_{cusp}''(\varrho t_i)  (1+O(|b_i|^\frac{1}{3}\tau_i) +O((\varrho t_i/ T_i)^{3} ))
	\end{align}
	as $i\rightarrow \infty$. So the pointed limit behavior of the (unrescaled) spaces $(\mathcal{X}_{\sigma_i},g_{glue,\sigma_i},x_i)$ is exactly the same as in the case of region $\mathfrak{R}_3$, and the rest of the discussion is then also the same.
	
\subsubsection{The subcase \texorpdfstring{$0 < c < 1\textup{:}$}{0<c<1} The genuine new neck, i.e., the obstructed case}
	
If $c \in (0,1)$, then by combining Proposition \ref{prop-psi-diff}, the convergence $F_i \rightarrow c$ and the fact that
\begin{align}
e^{-\psi_{cusp}(t_i)-a}|b_i| = e^{-a}|t_i^{3} b_i|,
\end{align}
we immediately obtain that after taking a subsequence,
\begin{align}
t_i/T_i \rightarrow \tilde c \in (0,1).
\end{align}
We set $\mu_i:=1$. With our usual reparametrization $s =1-t/T \in (0,1)$ we get
 \begin{align}\label{eq-psi-upbd}
\frac{w_{\sigma_i}}{w_{\sigma_i}(x_i)} &\to (1-\tilde{c})^\delta s^{-\delta},\\
	\mu_i^{-2}\frac{\tilde w_{\sigma_i}}{w_{\sigma_i}(x_i)} &\leq C \left(|b_i|^{\frac{1}{4}}(t - T_i)^{\frac{3}{4}}\right)^{-2} \frac{w_{\sigma_i}}{w_{\sigma_i}(x_i)} \leq C's^{-\frac{3}{2}-\delta},
 \end{align}
verifying \ref{ass:I} and \ref{ass:II}. Thus, as proved in Lemma \ref{l:modeloperators}, we get a smooth subsequential limit $\psi_\infty(s)$ solving the ODE $\tilde{L}_\infty\psi_\infty = L_\infty\psi_\infty= 0$ with $\psi_\infty(1-\tilde{c}) = 1$. Since $|\psi_\infty| \leq w_\infty = (1-\tilde{c})^\delta s^{-\delta}$, Lemma \ref{l:construct_uhat} tells us that $\psi_\infty = \lambda\hat{u}$ for some $\lambda \in \mathbb{R}$. We now use our assumption that $\phi_i \in \langle \hat{u}_{\sigma_i} \rangle^\perp$, which trivially implies $\psi_i \in \langle \hat{u}_{\sigma_i}\rangle^\perp$ and hence, by Proposition \ref{p:construct_obstruct}, $0 = \int_0^1 \psi_\infty\hat{u}\mu_\infty\,ds = \lambda \int_0^1 |\hat{u}|^2\mu_\infty\,ds$. Since $\mu_\infty$ is uniformly positive by Lemma \ref{l:volumeform}, we get $\lambda = 0$, verifying \ref{ass:III} and contradicting $\psi_\infty(1-\tilde{c}) = 1$. (This step is the \emph{only} reason why we need Proposition \ref{p:construct_obstruct}, which requires the hard work in Sections \ref{sss:neumann}--\ref{sss:final_boss}.)

\section{Obstruction theory and proof of the Main Theorem}\label{s:obstruction_theory}

We have now set up the pre-glued manifold $(\mathcal{X}_\sigma, \omega_{glue,\sigma})$ (Definition \ref{def:metric}) and the relevant complex Monge-Amp\`ere equation \eqref{mainMA}, estimated the right-hand side $f_\sigma$ (Theorem \ref{f:Riccipotential}), and proved a uniform weighted Hölder estimate for the inverse of the linearized operator $L_\sigma$ away from a certain obstruction space (Theorem \ref{t:uniform_Linfty}). It remains to explain the (not quite standard) inverse function theorem with obstructions which will be used to prove our Main Theorem. Throughout this section, it is helpful to refer to Table \ref{tab:weight} for the numerical behavior of the functions $\mathbf{r}_\sigma, w_\sigma, \tilde{w}_\sigma$ and $f_\sigma$.

\subsection{Fixed-point iteration on the orthogonal complement of the obstruction space}

This is a fairly standard argument using the weighted Hölder estimates of Theorem \ref{t:uniform_Linfty}, but it only solves the equation modulo some undetermined scalar multiple of the obstruction function.

We introduce the Monge-Amp\`ere operator
\begin{equation}
M_\sigma u := \log \biggl(\frac{(\omega_{glue,\sigma} + i\partial\overline\partial u)^2}{\omega_{glue,\sigma}^2}\biggr) - u,
\end{equation}
and we decompose $M_\sigma$ into its linearization at $u=0$ and a remainder as follows:
\begin{align}
L_\sigma = \Delta_{\omega_{glue,\sigma}}-{\rm Id},\quad Q_\sigma := M_\sigma - L_\sigma.
\end{align}
Recall the obstruction function $\hat{u}_\sigma = L_\sigma \hat{v}_\sigma \in C^\infty(\mathcal{X}_\sigma)$ introduced in Definition \ref{def:obstr}, and recall the result of Theorem \ref{t:uniform_Linfty}: the invertible operator 
\begin{equation}\label{eq-perp}
L_\sigma^\perp: \langle \hat{u}_\sigma \rangle^\perp \to \langle \hat v_\sigma \rangle^\perp,
\end{equation}
the restriction of $L_\sigma$ to the $L^2(\mathcal{X}_\sigma,\omega_{glue,\sigma})$-complements inside $C^{2,\bar\alpha}(\mathcal{X}_\sigma)$ resp. $C^{0,\bar\alpha}(\mathcal{X}_\sigma)$, satisfies
\begin{equation}
\|\phi\|_{C^{2,\bar\alpha}_w} \leq C\|L_\sigma^\perp \phi\|_{C^{0,\bar\alpha}_{\tilde{w}}}
\end{equation}
for all functions $\phi$ in its domain, with $C$ independent of $\sigma$.

We now introduce the following standard fixed-point iteration:
\begin{equation}\label{eq-def-iter}
\tilde{u}_{\sigma,0} := 0,\;\, \tilde{u}_{\sigma,i+1} := (L_\sigma^\perp)^{-1}[(f_\sigma - Q_\sigma(\tilde{u}_{\sigma,i}))^\perp].
\end{equation}
Here $f_\sigma$ is a fixed Ricci potential of $\omega_{glue,\sigma}$ (recall that such a potential is unique only up to constants) and $(f_\sigma - Q_\sigma(\tilde{u}_{\sigma,i}))^\perp$ denotes the $L^2(\mathcal{X}_\sigma,\omega_{glue,\sigma})$-orthogonal projection of $f_\sigma - Q_\sigma(\tilde{u}_{\sigma,i})$ onto $\langle \hat{v}_\sigma \rangle^\perp$.

\begin{lemma}\label{l:est-fperp}
Let $f_\sigma$ be the particular choice of Ricci potential in \eqref{riccipotential}. Then for all $\alpha,\bar\alpha \in (0,1)$ and for all $0<\delta \ll 1$ there exists a constant $C$ such that for all $0 < |\sigma| \ll 1$ we have that
\begin{align}\label{eq-f-perp-Ca}
\|f_\sigma^\perp\|_{C^{0,\bar\alpha}_{\tilde{w}}} \leq C|b|^{\frac{5}{6}(1-\alpha)-\frac{\delta}{3}}.    
\end{align}
\end{lemma}

\begin{proof}
This will follow from Theorem \ref{f:Riccipotential}. By definition,
\begin{align}\label{eq-def-fperp}
f^\perp_\sigma = f_\sigma - \frac{\langle f_\sigma, \hat v_\sigma\rangle_{L^2(\mathcal{X}_\sigma,\omega_{glue,\sigma})}}{\langle\hat v_\sigma, \hat v_\sigma\rangle_{L^2(\mathcal{X}_\sigma,\omega_{glue,\sigma})}}\hat v_\sigma.
\end{align}
Thus, we need to estimate four pieces: $f_\sigma$, the denominator and numerator of the fraction, and $\hat{v}_\sigma$.

(1) We claim that
\begin{align}\label{eq:fs_good}
\|f_\sigma\|_{C^{0,\bar\alpha}_{\tilde{w}}} \leq C|b|^{1-\alpha-\frac{\delta}{3}}. 
\end{align}
To prove this, note that since $f_\sigma$ is exponentially small in terms of $b$ everywhere else, we only need to estimate $f_\sigma$ on $\mathfrak{R}_3$ and on $\mathfrak{R}_5 \cup \mathfrak{R}_6 \cup \mathfrak{R}_7.$ 
On $\mathfrak{R}_3$ we have by Theorem \ref{f:Riccipotential} that
\begin{align}\label{est-0-1}
    \|f_\sigma\|_{C^0_{\tilde{w}}(\mathfrak{R}_3)}=O(|b||\tau|^3) \cdot |b|^\frac{1}{2} |T|^\frac{3}{2} \cdot |T|^\delta = O(|b|^{1-\frac{\delta}{3}}|\tau|^3).
\end{align}
For the $C^{0,\bar\alpha}_{\tilde{w}}$ seminorm, if $p, q\in \mathfrak{R}_3$ and $d_{\omega_{glue,\sigma}}(p,q)<\mathbf{r}_\sigma(p),$ then
\begin{align}\label{est-1}
\frac{\mathbf{r}_\sigma(p)^{\bar\alpha}}{\tilde{w}_\sigma(p)}
\frac{|f_\sigma(p)-f_\sigma(q)|}{d_{\omega_{glue,\sigma}}(p,q)^{\bar\alpha}} &\leq
C\frac{\mathbf{r}_\sigma(p)^{\bar\alpha}}{\tilde{w}_\sigma(p)} \left(\sup\nolimits_{B_{\omega_{glue,\sigma}}(p,\mathbf{r}_\sigma(p))} |\nabla_{\omega_{glue,\sigma}} f_\sigma|_{\omega_{glue,\sigma}}\right) \mathbf{r}_\sigma(p)^{1-\bar\alpha},
\end{align}
which has the same upper bound as \eqref{est-0-1} thanks to the gradient estimate in Theorem \ref{f:Riccipotential}. Lastly, 
on $\mathfrak{R}_5 \cup \mathfrak{R}_6 \cup \mathfrak{R}_7,$ we can similarly estimate
\begin{align}\label{est-2}
\|f_\sigma\|_{C^{0,\bar\alpha}_{\tilde{w}}(\mathfrak{R}_5 \cup \mathfrak{R}_6 \cup \mathfrak{R}_7)} \leq C|b|^{1-\alpha-\frac{\alpha\delta}{3}}.    
\end{align}
Combining \eqref{est-0-1}--\eqref{est-2}, we obtain the claim.

(2) It follows from Lemma \ref{l:volumeform} and \eqref{eq:sol_inhomog_ode} that for some $C>0$ independent of $\sigma$,  
\begin{equation}\label{eq-Phi-1}\langle\hat v_\sigma, \hat v_\sigma\rangle_{L^2(\mathcal{X}_\sigma,\omega_{glue,\sigma})}\geq C^{-1}|T|^{-2} \geq C^{-1}|b|^{\frac{2}{3}}.\end{equation}

(3) Since $f_\sigma$ is exponentially small in terms of $b$ everywhere else, we need to estimate $\langle f_\sigma, \hat v_\sigma\rangle_{L^2}$ only on $\mathfrak{R}_3$ and on $\mathfrak{R}_5 \cup \mathfrak{R}_6 \cup \mathfrak{R}_7$. On $\mathfrak{R}_3$, we have by Theorem \ref{f:Riccipotential}, \eqref{eq:guess_who} and Lemma \ref{l:volumeform} that
\begin{align}\label{eq:nothing_new}
\langle f_\sigma, \hat v_\sigma\rangle_{L^2(\mathfrak{R}_3,\omega_{glue,\sigma})} =O(|b||\tau|^3) \cdot |T|^{-2} \cdot \left|\frac{\tau}{T}\right| =O(|b|^{2} |\tau|^4).
\end{align}
On $\mathfrak{R}_5 \cup \mathfrak{R}_6 \cup \mathfrak{R}_7$, we can similarly estimate
\begin{align}\label{eq-f-v-IP}
\langle f_\sigma, \hat v_\sigma \rangle_{L^2(\mathfrak{R}_5 \cup \mathfrak{R}_6 \cup \mathfrak{R}_7,\omega_{glue,\sigma})} = O(|b|^\frac{1}{2} T_0^\frac{3}{2}) \cdot |b| \cdot T_0 = O(|b|^{\frac{3}{2} - \frac{5}{6}\alpha}).    
\end{align}

(4) We first estimate the weighted $C^0$ norm of $\hat{v}_\sigma$:
\begin{align}\label{eq:hatv-C0w}
\|\hat v_\sigma\|_{C^0_{\tilde w}(\mathcal{X}_\sigma)} \leq C\|\tilde w_\sigma^{-1}\|_{L^\infty(\mathcal{X}_\sigma)}\leq C|T|^\delta \leq C|b|^{-\frac{\delta}{3}}.
\end{align}
Now we estimate the weighted $C^{0,\bar\alpha}$ seminorm of $\hat v_\sigma$. We claim that, in fact,
\begin{align}\label{eq:hatv-Caw}
[\hat v_\sigma]_{C^{0,\bar\alpha}_{\tilde{w}}}
=\sup\left\{\frac{\mathbf{r}_\sigma(p)^{\bar\alpha}}{\tilde{w}_\sigma(p)}
\frac{|\hat{v}_\sigma(p)-\hat{v}_\sigma(q)|}{d_{\omega_{glue,\sigma}}(p,q)^{\bar\alpha}}: 0 < d_{\omega_{glue,\sigma}}(p,q) < \mathbf{r}_\sigma(p) \right\} \leq C|b|^{-\frac{\delta}{3}}
\end{align}
as well. This is an obvious consequence of the stronger claim
\begin{align}\label{eq:whatireallyreallywant}
\mathbf{r}_\sigma |\nabla_{\omega_{glue,\sigma}}\hat{v}_\sigma|_{\omega_{glue,\sigma}} \leq C.
\end{align}
To prove \eqref{eq:whatireallyreallywant}, hence the lemma, we fix a universal $0 < \varepsilon \ll 1$ and distinguish three cases.

(4a) $s=1-\eta\in[\epsilon,1-\epsilon]$. \eqref{eq:whatireallyreallywant} is clear thanks to the quasi-coordinates from Convention \ref{conv}.

(4b) $\eta\in (0,\epsilon]$. In this case, $\hat v_\sigma=\hat v\chi_\sigma(1-\cdot)$, where $\chi_\sigma$ is as in \eqref{eq:define_chi}. Then
\begin{align}
    |\nabla_{\omega_{glue,\sigma}}\hat v_\sigma|_{\omega_{glue,\sigma}} \leq C \eta |\hat{v}_\sigma'(1-\eta)| \leq C\eta^3,
\end{align}
using Proposition \ref{prop-T-c} to compare the $\omega_{glue,\sigma}$-gradient to the $\omega_{cusp}$-gradient and using Lemma \ref{l:construct_vhat} and \eqref{eq:define_chi} to estimate $\hat{v}$ and $\chi_\sigma$, respectively.

(4c) $s\in (0,\epsilon]$. In this case, $\hat v_\sigma=\chi_\sigma(\hat v-\hat v(0))+\hat v(0)$. From the proof of Proposition \ref{prop-L-s},
\begin{align}\label{eq:lastone!!}
    \mathbf{r}_\sigma |\nabla_{\omega_{glue,\sigma}}\hat v_\sigma|_{\omega_{glue,\sigma}} \leq C\mathbf{r}_\sigma s^{\frac{1}{4}}|\hat v_\sigma'(s)|\leq C s |\hat{v}_\sigma'(s)|.
\end{align}
Lemma \ref{l:construct_vhat} and \eqref{eq:define_chi} again tell us that this is bounded by $C s^{3/2}$. However, even the weaker version of Lemma \ref{l:construct_vhat} that allows for an $s^{1/2}$ term in the expansion of $\hat{v}$ would be sufficient to bound \eqref{eq:lastone!!} by $Cs^{1/2} = O(1)$ as $s \to 0$, which is still enough for the current lemma.
\end{proof}

We can now state the main result of this subsection.

\begin{proposition}\label{thm:solv_mod_obstr}
Let $f_\sigma$ be the particular choice of Ricci potential in \eqref{riccipotential}. Then for all $\alpha \in (0,\frac{1}{5})$, $\bar\alpha \in (0,1)$ and $0 < \delta \ll 1$, there exist constants $C,C'$ such that for all $0 < |\sigma| \ll 1$ the sequence $\tilde{u}_{\sigma,i}$ defined in \eqref{eq-def-iter} converges in the $C^{2,\bar\alpha}_w(\mathcal{X}_\sigma,\omega_{glue,\sigma})$ norm, its limit $\tilde{u}_\sigma \in \langle \hat{u}_\sigma \rangle^\perp$ satisfies
\begin{align}\label{weightedf}
\|\tilde{u}_\sigma\|_{C^{2,\bar\alpha}_w} \leq C\|f^\perp_\sigma\|_{C^{0,\bar\alpha}_{\tilde{w}}} \leq C'|b|^{\frac{5}{6}(1-\alpha)-\frac{\delta}{3}},
\end{align}
and $\tilde{u}_\sigma$ is a solution to the equation
\begin{align}\label{eq:what-the-theorem-says}
    M_\sigma \tilde{u}_\sigma = f_\sigma + \lambda_\sigma \hat{v}_\sigma
\end{align}
for some unknown $\lambda_\sigma \in \mathbb{R}$.
\end{proposition}

\begin{proof}
We first explain the structure of the argument, which is of course standard. In order to prove that the sequence $\tilde{u}_{\sigma,i}$ converges in the $C^{2,\bar\alpha}_w(\mathcal{X}_\sigma,\omega_{glue,\sigma})$ norm and to bound the norm of its limit $\tilde{u}_\sigma$, we prove that it is Cauchy by comparing it to a geometric sequence, as follows:
\begin{align}
\|\tilde{u}_{\sigma,i+1}-\tilde{u}_{\sigma,i}\|_{C^{2,\bar\alpha}_w} &\leq C\|(Q_\sigma(\tilde{u}_{\sigma,i}) - Q_\sigma(\tilde{u}_{\sigma,i-1}))^\perp\|_{C^{0,\bar\alpha}_{\tilde{w}}}\label{eq-clear}\\
&\leq C|b|^{-\frac{1}{6}-\frac{\delta}{3}}\|Q_\sigma(\tilde{u}_{\sigma,i}) - Q_\sigma(\tilde{u}_{\sigma,i-1})\|_{C^{0,\bar\alpha}_{\tilde{w}}} \label{eq-Ca-perp}\\
&\leq C|b|^{-\frac{2}{3}-\frac{\delta}{3}}(\|\tilde{u}_{\sigma,i}\|_{C^{2,\bar\alpha}_w}+\|\tilde{u}_{\sigma,i-1}\|_{C^{2,\bar\alpha}_w})\|\tilde{u}_{\sigma,i}-\tilde{u}_{\sigma,i-1}\|_{C^{2,\bar\alpha}_w}.\label{eq-clear-too}
\end{align}
Here, \eqref{eq-clear} follows directly from Theorem \ref{t:uniform_Linfty}, \eqref{eq-Ca-perp} is a little technical and is deferred to Claim \ref{claim:perp} below, and \eqref{eq-clear-too} then follows from elementary inequalities and from the fact that
\begin{align}
\frac{1}{\tilde{w}_\sigma}\biggl(\frac{w_\sigma}{\mathbf{r}_\sigma^2}\biggr)^2 = \tilde{w}_\sigma \leq |b|^{-\frac{1}{2}}.
\end{align}
Given \eqref{eq-clear-too}, we can finish the proof as follows: We have $\tilde{u}_{\sigma,0} = 0$ and $\tilde{u}_{\sigma,1} = (L_\sigma^\perp)^{-1}(f_\sigma^\perp)$, so
\begin{align}
\|\tilde{u}_{\sigma,1}\|_{C^{2,\bar\alpha}_w} &=  \|(L_\sigma^\perp)^{-1}(f_\sigma^\perp)\|_{C^{2,\bar\alpha}_w}\leq C \|f^\perp_\sigma\|_{C^{0,\bar\alpha}_{\tilde w}} \leq C|b|^{\frac{5}{6} - \frac{5}{6}\alpha -\frac{\delta}{3}}
\end{align}
by Lemma \ref{l:est-fperp}. Thus, we can aim to prove inductively that
\begin{align}\label{eq:7rc}
\|\tilde{u}_{\sigma,i}\|_{C^{2,\bar\alpha}_w}\leq C|b|^{\frac{5}{6} - \frac{5}{6}\alpha -\frac{\delta}{3}}.
\end{align}
This can be combined with \eqref{eq-clear-too}, resulting in the estimate 
\begin{align}
\|\tilde{u}_{\sigma,i+1}-\tilde{u}_{\sigma,i}\|_{C^{2,\bar\alpha}_w}
\leq C|b|^{\frac{1}{6} -\frac{5}{6}\alpha - \frac{2}{3}\delta}\|\tilde{u}_{\sigma,i}-\tilde{u}_{\sigma,i-1}\|_{C^{2,\bar\alpha}_w}.
\end{align}
Thus, if $\alpha<\frac{1}{5}$, then our sequence is Cauchy and we can also complete the inductive step for \eqref{eq:7rc}. Then, passing to the limit $i \to \infty$ in \eqref{eq-def-iter}, it is clear that $\tilde{u}_\sigma \in \langle \hat{u}_\sigma\rangle^\perp$ and
\begin{align}(M_\sigma \tilde{u}_\sigma)^\perp = L_\sigma^\perp \tilde{u}_\sigma + (Q_\sigma \tilde{u}_\sigma)^\perp = f_\sigma^\perp,\end{align}
which is equivalent to the claimed property \eqref{eq:what-the-theorem-says}.

It remains to prove the following claim, which directly implies \eqref{eq-Ca-perp}.

\begin{claim}\label{claim:perp}
For all $\Phi \in C^{0,\bar\alpha}(\mathcal{X}_\sigma,\omega_{glue,\sigma})$ we have that
\begin{align}\label{eq:worst_possible}
\|\Phi^\perp\|_{C^{0,\bar\alpha}_{\tilde{w}}} \leq C|b|^{-\frac{1}{6}-\frac{\delta}{3}}\|\Phi\|_{C^{0,\bar\alpha}_{\tilde{w}}}.
\end{align}
\end{claim}

\noindent \emph{Proof of Claim \ref{claim:perp}}. Most of the necessary steps were already done in the proof of Lemma \ref{l:est-fperp} in the special case $\Phi = f_\sigma$. In fact, the only part that needs to be redone is the estimate of $\langle \Phi, \hat{v}_\sigma\rangle_{L^2}$:
\begin{equation}\label{eq-Phi-2}
 \big|\langle\Phi, \hat v_\sigma\rangle_{L^2(\mathcal{X}_\sigma,\omega_{glue,\sigma})}\big|\leq C\|\Phi\|_{C_{\tilde w}^0}
 \|\hat v_\sigma \tilde w_\sigma\|_{L^1(\mathcal{X}_\sigma, \omega_{glue,\sigma})}
 \leq 
C\|\Phi\|_{C^0_{\tilde w}} |b|^{\frac{1}{2}},
\end{equation}
where the estimate of $ \|\hat v_\sigma \tilde w_\sigma\|_{L^1}$ is proved as follows. Recall that
\begin{align}
    \tilde{w}_\sigma = \mathbf{r}_\sigma^{-2}w_\sigma = \begin{cases}
|T|^{-\delta} &\text{on}\;\,\mathfrak{R}_1,\\
|b|^{-\frac{1}{2}}\Psi_\sigma^*(\Phi_\sigma^{-1})^*((t-T)^{-\frac{3}{2}-\delta})&\text{on}\;\,\mathfrak{R}_2 \cup \cdots \cup \mathfrak{R}_6,\\
|b|^{-\frac{1}{2}} &\text{on}\;\,\mathfrak{R}_7.
\end{cases}
\end{align}
Thus, using Lemma \ref{l:volumeform} to bound the volume form on the neck and Lemma \ref{l:construct_vhat} to bound $\hat{v}_\sigma$,
\begin{align}
\begin{split}
\int_{\mathcal{X}_\sigma}|\hat v_\sigma| \tilde w_\sigma \,\omega_{glue,\sigma}^2 &\leq C\biggl(|b| \cdot |b|^{-\frac{1}{2}} + 
|b| \cdot \int_1^{2|T|^\alpha} |b|^{-\frac{1}{2}}y^{-\frac{3}{2}-\delta} \, dy \\
&+ |T|^{-2}\int_{2|T|^{\alpha-1}}^{1-\frac{\tau}{T}} |b|^{-\frac{1}{2}}(|T|s)^{-\frac{3}{2}-\delta} \,ds \biggr) \leq C|b|^{\frac{1}{2}}.
\end{split}
\end{align}
Note that the decay $|\hat{v}(1-\eta)| = O(\eta^3)$ as $\eta \to 0$ compensates the blowup $\mu_T(1-\eta) = O(\eta^{-3})$ of the radial volume density from Lemma \ref{l:volumeform}. This observation will also be used several times below.

Then \eqref{eq-Phi-2} is proved, and Claim \ref{claim:perp} follows from this together with \eqref{eq-Phi-1} and \eqref{eq:hatv-C0w}--\eqref{eq:hatv-Caw}.
\end{proof}

\subsection{Killing the obstruction by varying the Ricci potential}\label{sec:kill}

Here we carry out the last step of the argument. Recall that for a fixed Ricci potential $f_\sigma$ we have solved the Monge-Amp\`ere equation modulo obstructions in Proposition \ref{thm:solv_mod_obstr}. Our final goal is to kill the obstruction by using the freedom of adding a constant to the Ricci potential. At first this seems contradictory because if we add a constant to $f_\sigma$, then the solution $u_\sigma$ to the Monge-Amp\`ere equation changes by the same constant. However, $\tilde{u}_\sigma$, the solution ``modulo obstructions,'' does not change in such an obvious way.

To be slightly more precise, we change $f_\sigma$ by a constant $s$ with $|s|$ not much bigger than the $C^{0,\bar\alpha}_{\tilde w}$ bound of $f_\sigma^\perp$ in \eqref{weightedf}. By running the same iteration as in the proof of Proposition \ref{thm:solv_mod_obstr}, we obtain $\tilde u_{\sigma,s} \in \langle \hat{u}_\sigma \rangle^\perp$ solving the Monge-Amp\`ere equation with right-hand side $f_\sigma + s$ modulo obstructions. Moreover, $\tilde{u}_{\sigma,s}$ satisfies a $C^{2,\bar\alpha}_{w}$ bound like \eqref{weightedf} independently of $s$. The obstruction coefficient $\lambda_{\sigma,s}$ depends continuously on $s$. It turns out that we can calculate $\lambda_{\sigma,s}$ up to errors that are negligible at the boundary of the allowed range for $s$, and this ``leading term'' of $\lambda_{\sigma,s}$ is proportional to $s$. So by the intermediate value theorem, $\tilde{u}_{\sigma,s}$ is the true solution $u_\sigma$ for some $s$ in the allowed range.

The precise statement is as follows. After the proof we will deduce our Main Theorem from this.

\begin{theorem}\label{thm:kill_obstr}
Let $f_\sigma$ be the choice of Ricci potential in \eqref{riccipotential}. For all $\alpha \in (0,\frac{1}{5})$, $\mathbf{a} \in (\alpha,\frac{1}{5})$, $\bar\alpha \in (0,1)$ and $0 < \delta \ll 1$, there exists a constant $C$ such that for all $0 < |\sigma| \ll 1$ the following holds. Set
\begin{equation}\label{eq:def_A}
A:= |b|^{\frac{5}{6}(1 - \mathbf{a})}.
\end{equation}
Then for all $s \in [-A,A]$ there exists a
$\tilde u_{\sigma,s}\in \langle \hat{u}_\sigma\rangle^\perp \subset C^{2,\bar\alpha}_w(\mathcal{X}_\sigma)$ solving the equation
\begin{align}
    M_\sigma \tilde u_{\sigma,s}=f_\sigma+s+\lambda_{\sigma,s}\hat v_\sigma
\end{align}
for some $\lambda_{\sigma,s} \in \mathbb{R}$ and satisfying the estimate
\begin{align}\label{eq:estimate_of_soln}
\|\tilde{u}_{\sigma,s}\|_{C^{2,\bar\alpha}_w} \leq C|b|^{\frac{5}{6}(1-\mathbf{a})-\frac{\delta}{3}}.
\end{align}
Moreover, there exists an $s_\sigma\in [-A,A]$ such that $\lambda_{\sigma,s_\sigma}=0$.
\end{theorem}

\begin{proof}
For all $s \in [-A,A]$ the proof of Proposition \ref{thm:solv_mod_obstr} yields a $\tilde u_{\sigma,s}\in \langle u_\sigma\rangle^\perp$ such that
\begin{align}
M_\sigma \tilde u_{\sigma,s}=f_\sigma+s+\lambda_{\sigma,s}\hat v_\sigma
\end{align}
and such that \eqref{eq:estimate_of_soln} holds. To see this carefully, we only need to redo the estimate of Lemma \ref{l:est-fperp}:
\begin{align}\label{eq:redo}
\begin{split}
\|(f_\sigma + s)^\perp\|_{C^{0,\bar\alpha}_{\tilde{w}}} &\leq \|f_\sigma ^\perp\|_{C^{0,\bar\alpha}_{\tilde{w}}} + |s| \left\|1 - \frac{\langle 1, \hat{v}_\sigma\rangle_{L^2}}{\langle \hat{v}_\sigma,\hat{v}_\sigma\rangle_{L^2}}\hat{v}_\sigma\right\|_{C^{0,\bar\alpha}_{\tilde{w}}}\\
&\leq C|b|^{\frac{5}{6}(1-\alpha)-\frac{\delta}{3}} + |b|^{\frac{5}{6}(1-\mathbf{a})}\left(\|\tilde{w}_\sigma^{-1}\|_{L^\infty}+
\frac{|\langle 1, \hat{v}_\sigma\rangle_{L^2}|}{\langle \hat{v}_\sigma,\hat{v}_\sigma\rangle_{L^2}}\|\hat{v}_\sigma\|_{C^{0,\bar\alpha}_{\tilde{w}}}\right)\\
&\leq C|b|^{\frac{5}{6}(1-\mathbf{a})-\frac{\delta}{3}}
\end{split}
\end{align}
thanks to \eqref{eq-f-perp-Ca}, \eqref{eq:hatv-C0w}--\eqref{eq:hatv-Caw}, \eqref{eq-Phi-1}, and the following upper bound of $|\langle 1, \hat{v}_\sigma\rangle_{L^2}|$: for $\alpha \leq \frac{1}{3}$,
\begin{align}\label{eq:hatvs-upper}
\begin{split}
|\langle 1, \hat{v}_\sigma\rangle_{L^2}|\leq C\left(|b| \cdot |T|^\alpha + |T|^{-2}\int_{2|T|^{\alpha-1}}^{1-\frac{\tau}{T}} ds \right)\leq C|b|^{\frac{2}{3}} \leq C \langle \hat{v}_\sigma,\hat{v}_\sigma\rangle_{L^2}.
\end{split}
\end{align}
Here we have used Lemmas \ref{l:volumeform} and \ref{l:construct_vhat} to bound the volume form on the middle neck and $\hat{v}_\sigma$. This proves \eqref{eq:redo}, hence, as in the proof of Proposition \ref{thm:solv_mod_obstr}, the existence of $\tilde{u}_{\sigma,s}$ and the estimate \eqref{eq:estimate_of_soln}. Note that for $s = \pm A$ the operator $^\perp$ did \emph{not} change our estimate of the $C^{0,\bar\alpha}_{\tilde{w}}$ norm of $f_\sigma + s$ except for an arbitrarily small power of $|b|$. This is a key improvement over the case $s = 0$, where we almost lost a factor of $|b|^{-1/6}$ (compare \eqref{eq-f-perp-Ca} to \eqref{eq:fs_good}), which is close to the worst possible loss \eqref{eq:worst_possible}.

By taking the $L^2(\mathcal{X}_\sigma, \omega_{glue,\sigma})$-inner product with $\hat{v}_\sigma$,
\begin{align}\label{lambdasigmas}
   \lambda_{\sigma,s} = \frac{\langle L_\sigma\tilde{u}_{\sigma,s} + Q_\sigma\tilde{u}_{\sigma,s} - f_\sigma-s, \hat{v}_\sigma\rangle_{L^2}}{\|\hat{v}_\sigma\|_{L^2}^2} = \frac{\langle Q_\sigma\tilde{u}_{\sigma,s} - f_\sigma-s, \hat{v}_\sigma\rangle_{L^2}}{\|\hat{v}_\sigma\|_{L^2}^2},
\end{align}
where we have used $\langle L_\sigma\tilde u_{\sigma,s},\hat v_\sigma\rangle_{L^2}=\langle \tilde u_{\sigma,s}, L_\sigma\hat v_\sigma\rangle_{L^2}=\langle \tilde u_{\sigma,s},\hat{u}_\sigma\rangle_{L^2}=0$. Our goal is to show that as $s$ goes from $-A$ to $A$, the numerator of $\lambda_{\sigma,s}$ will have a sign change. One can prove directly from \eqref{eq-def-iter} that, for a fixed $\sigma$, the solution $\tilde{u}_{\sigma,s}$ depends continuously on $s$ in the $C^{2,\bar\alpha}$ topology on $\mathcal{X}_\sigma$. Thus, if there is indeed a sign change, then $\lambda_{\sigma,s}$ must vanish for some $s \in [-A,A]$, as desired.

To prove that the numerator of $\lambda_{\sigma,s}$ changes sign, we bound $\langle s,\hat v_\sigma\rangle_{L^2}$, $\langle f_\sigma,\hat v_\sigma\rangle_{L^2}$ and $\langle Q_\sigma \tilde u_{\sigma,s}, \hat v_\sigma\rangle_{L^2}$ separately. This is done in the following three steps.

(1) From Lemmas \ref{l:volumeform} and \ref{l:construct_vhat},
\begin{equation}\label{1dotv}
|\langle 1, \hat{v}_\sigma\rangle_{L^2}|\geq C^{-1}|T|^{-2}\geq C^{-1}|b|^{\frac{2}{3}},
\end{equation}
i.e., we have a lower bound matching the upper bound \eqref{eq:hatvs-upper} up to a constant.  This is the good term that will enforce a sign change. Here we crucially use the fact that $\hat{v}$ has a sign, so that it suffices to integrate over a small neighborhood of the left endpoint $s = 0$, and that $\hat{v}(0) = C_0 \neq 0$.

(2) We have already proved in \eqref{eq:nothing_new}--\eqref{eq-f-v-IP} that
\begin{align}
|\langle f_\sigma, \hat{v}_\sigma\rangle_{L^2}| \leq C|b|^{\frac{3}{2}-\frac{5}{6}\alpha}.\label{fdotv2}
\end{align}

(3) Lastly, we claim that
\begin{align}\label{udotv}
|\langle Q_\sigma\tilde{u}_{\sigma,s},\hat{v}_\sigma\rangle_{L^2} | 
\leq C|b|^{\frac{5}{3}(1-\mathbf{a})-\frac{2}{3}\delta}.
\end{align}
To prove this, we first make a pointwise estimate using \eqref{eq:estimate_of_soln}:
\begin{align}\label{eq:est_of_hess_utilde}
\begin{split}
\left|\nabla^2_{\omega_{glue,\sigma}}\tilde{u}_{\sigma,s}\right|_{\omega_{glue,\sigma}} &\leq  \|\tilde{u}_{\sigma,s}\|_{C^{2,\bar{\alpha}}_w} \cdot \mathbf{r}_\sigma^{-2} w_\sigma\\
&\leq C|b|^{\frac{5}{6}(1-\mathbf{a})-\frac{\delta}{3}}
\cdot 
\begin{cases}
|T|^{-\delta} &\text{on}\;\,\mathfrak{R}_1,\\
|b|^{-\frac{1}{2}}\Psi_\sigma^*(\Phi_\sigma^{-1})^*((t-T)^{-\frac{3}{2}-\delta})&\text{on}\;\,\mathfrak{R}_2 \cup \cdots \cup \mathfrak{R}_6,\\
|b|^{-\frac{1}{2}} &\text{on}\;\,\mathfrak{R}_7.
\end{cases}
\end{split}
\end{align}
Thus, using again Lemmas \ref{l:volumeform} and \ref{l:construct_vhat}, 
\begin{align}\begin{split}
|\langle Q_\sigma\tilde{u}_{\sigma,s},\hat{v}_\sigma\rangle_{L^2} | 
&\leq C|b|^{\frac{5}{3}(1-\mathbf{a})-\frac{2}{3}\delta} \cdot \biggl( |b| \cdot |b|^{-1} + |b| \cdot \int_1^{2|T|^\alpha} |b|^{-1} y^{-3-2\delta} \,dy\\  
&+ |T|^{-2} \int_{2|T|^{\alpha-1}}^{1-\frac{\tau}{T}} |b|^{-1}(|T|s)^{-3-2\delta}\,ds\biggr).
 \end{split}\end{align}
All terms in the big parenthesis are uniformly bounded as $\sigma \to 0$, which proves \eqref{udotv}.

Finally, let us prove the theorem using these estimates. Choosing $\alpha<\mathbf{a}$ we get
\begin{align}
    \frac{5}{6}(1-\mathbf{a}) +\frac{2}{3} < \frac{3}{2}-\frac{5}{6}\alpha.
\end{align}
Thus, from \eqref{1dotv} and \eqref{fdotv2} we have that
\begin{equation}|\langle A,\hat v_\sigma\rangle_{L^2}| \gg |\langle f_\sigma,\hat{v}_\sigma\rangle_{L^2}|.\end{equation}
Similarly, if $\mathbf{a}<\frac{1}{5}$ and $\delta \ll 1$, then
\begin{align}
    \frac{5}{6}(1-\mathbf{a}) +\frac{2}{3} <\frac{5}{3}(1-\mathbf{a}) -\frac{2}{3}\delta,
\end{align}
so from \eqref{1dotv} and \eqref{udotv},
\begin{equation}|\langle A,\hat v_\sigma\rangle_{L^2}| \gg |\langle Q_\sigma\tilde{u}_{\sigma,s},\hat{v}_\sigma\rangle_{L^2}|.\end{equation}
It is now clear from \eqref{lambdasigmas} that $\lambda_{\sigma,s}$ changes sign at the boundary of the interval $s \in [-A,A]$.
\end{proof}

We now deduce the statements of the Main Theorem.

We begin by applying Theorem \ref{thm:kill_obstr} with $\alpha < \mathbf{a}$ both very small and with $\bar\alpha \in (0,1)$ arbitrary. This yields that for all $0 < |\sigma| \ll 1$ there exists an $s_\sigma \in [-A,A]$ such that $\lambda_{\sigma,s_\sigma} = 0$, so $\tilde{u}_{\sigma,s_\sigma}$ coincides with the (unique) solution $u_\sigma$ to our Monge-Amp\`ere equation. Thus, $u_\sigma$ satisfies estimate \eqref{eq:estimate_of_soln}, i.e., 
\begin{align}\label{eq:estimate_u_part2}
    \|u_\sigma\|_{C^{2,\bar\alpha}_w} = O_\varepsilon(|b|^{\frac{5}{6}-\varepsilon})
\end{align}
for any $\varepsilon>0$. The first part of the Main Theorem concerns the behavior of $\omega_{KE,\sigma} = \omega_{glue,\sigma}+\ddbar u_\sigma$ on the Tian-Yau cap, region $\mathfrak{R}_7$. The definition of $\mathfrak{R}_7$ depends on a parameter $R$ via the definition of $\Phi_\sigma$, \eqref{eq-Phis}. Up to replacing the $R$ of the Main Theorem by $R/C$ and assuming that this is $> C$ for some universal constant $C \gg 1$, the $R$ of the Main Theorem can be identified with the $R$ of \eqref{eq-Phis}. Fixing $R$, the weight function $w_\sigma$ equals $1$ and the regularity scale $\mathbf{r}_\sigma$ equals $|b|^{1/4}$ on $\mathfrak{R}_7$. So it follows from \eqref{eq:estimate_u_part2} and from the definition of the weighted norm, \eqref{eq:def_wtd_norm}, that 
\begin{align}\label{eq:estimate_u_part3}
    \sup\nolimits_{\mathfrak{R}_7} |\ddbar u_\sigma|_{\omega_{glue,\sigma}} = O(|b|^{-\frac{1}{2}}) \cdot O_\varepsilon(|b|^{\frac{5}{6}-\varepsilon}) = O_\varepsilon(|{\log |\sigma|}|^{-1+\varepsilon}).
\end{align}
The left-hand side of \eqref{eq:estimate_u_part3} is invariant under diffeomorphisms and rescalings applied simultaneously to $\omega_{KE,\sigma}$ and $\omega_{glue,\sigma}$. This implies the first statement of the Main Theorem.

\begin{remark}
Given \eqref{eq:estimate_u_part3} and the estimate $\sup_{\mathfrak{R}_7} |u_\sigma| = O_{\varepsilon}(|b|^{(5/6)-\varepsilon})$ that also follows from \eqref{eq:estimate_u_part2}, the standard theory of the Monge-Amp\`ere equation tells us that in fact, the $C^{1,\beta}(\mathfrak{R}_7,\omega_{glue,\sigma})$ norm of $\omega_{KE,\sigma} - \omega_{glue,\sigma}$ goes to zero for $\beta < \frac{1}{3}$. Clearly, the optimal result would be to have this for all $\beta < 1$ because $\omega_{glue,\sigma}$ is Ricci-flat on $\mathfrak{R}_7$ whereas $\omega_{KE,\sigma}$ has Ricci $= -1$. This suggests that the best possible exponent in \eqref{eq:estimate_u_part2} is $1$ rather than $\frac{5}{6}$ (and $-\frac{3}{2}$ rather than $-1$ in \eqref{eq:estimate_u_part3} and in the Main Theorem).
\end{remark}

\begin{remark}
Note the following subtlety related to the asymptotic expansion of $\omega_{KE,0}$ in \eqref{eq:KE-CUSP-2}. As explained in Remark \ref{scalelambda}, we need to replace $\Psi_\sigma$ by $scale_{\lambda^{-1}}\circ \Psi_\sigma$ to be able to carry out our gluing construction, where $\lambda = e^{\mathfrak{s}}$ for the constant $\mathfrak{s}$ from \eqref{eq:KE-CUSP-2}, which is uniquely determined by the global geometry of $(\mathcal{X}_0^{reg}, \omega_{KE,0})$ but is not known explicitly. We have assumed that $\mathfrak{s} = 0$ for simplicity, but if $\mathfrak{s} \neq 0$ then $\mathfrak{s}$ or $\lambda$ could in principle appear in the statement of the Main Theorem. However, in the Main Theorem we consider the pullback $(m_\sigma \circ \Psi_\sigma)^{-*}(\omega_{KE,\sigma})$, and $\omega_{KE,\sigma}$ is approximated by
\begin{equation}
    (scale_{\lambda^{-1}} \circ \Psi_\sigma)^*[\sqrt{2|b|}\, (m_{\sigma \lambda^{-3}})^*(\omega_{TY_1})]
\end{equation}
in the bubbling region. One easily checks that the unknown constant $\lambda$ cancels out.
\end{remark}

As for the rest of the Main Theorem, the Euler numbers of $\mathcal{X}_0^{reg}$ and of $\mathcal{X}_\sigma$ ($0 < |\sigma| \ll 1$) obviously differ by the Euler number of the Tian-Yau space. The Euler number of a smooth sextic in $\mathbb{CP}^3$ is $108$ by a standard computation. As a smoothing of an isolated cubic cone singularity, the Tian-Yau space is homotopy equivalent to $\bigvee_{i=1}^8 S^2$ and so has Euler number $9$ \cite[Thm 1]{MO}. The $L^2$-curvature identity follows from this by using the Chern-Gauß-Bonnet theorem for Einstein $4$-manifolds and the standard fact that the Chern-Gauß-Bonnet formula holds without boundary terms both on $(TY_1, \omega_{TY_1})$ and on $(\mathcal{X}_0^{reg},\omega_{KE,0})$. Because there is no loss of $L^2$-curvature, there cannot be any other bubbles.


\begin{thebibliography}{99}
\bibitem{And}{M. Anderson, \emph{Dehn filling and Einstein metrics in higher dimensions}, J. Differ. Geom. {\bf 73} (2006), 219--261.}
\bibitem{Aubin}{T. Aubin, \emph{Équations du type Monge-Ampère sur les variétés kählériennes compactes}, Bull. Sci. Math. {\bf 102} (1978), 63--95.}
\bibitem{Bamler}{R. Bamler, \emph{Construction of Einstein metrics by generalized Dehn filling}, J. Eur. Math. Soc. {\bf 14} (2012), 887--909.}
\bibitem{BG}{R. Berman, H. Guenancia, \emph{Kähler-Einstein metrics on stable varieties and log canonical pairs}, Geom. Funct. Anal. {\bf 24} (2014), 1683--1730.}
\bibitem{biqguen}{O. Biquard, H. Guenancia, \emph{Degenerating Kähler-Einstein cones, locally symmetric cusps, and the Tian-Yau metric}, Invent. Math. {\bf 230} (2022), 1101--1163.}
\bibitem{biqroll}{O. Biquard, Y. Rollin, \emph{Smoothing singular constant scalar curvature Kähler surfaces and minimal Lagrangians}, Adv. Math. {\bf 285} (2015), 980--1024.}
\bibitem{CY}{S.-Y. Cheng, S.-T. Yau, \emph{Differential equations on Riemannian manifolds and their geometric applications}, Comm. Pure Appl. Math. {\bf 28} (1975), 333--354.}
\bibitem{CYcusp} S.-Y. Cheng, S.-T. Yau, \emph{Inequality between Chern numbers of singular K\"ahler surfaces and characterization of orbit space of discrete group of $SU(2,1)$}, Complex differential geometry and nonlinear differential equations (Brunswick, ME, 1984), 31--44, Contemp. Math., 49, Amer. Math. Soc., Providence, RI, 1986.
\bibitem{CH}{R. Conlon, H.-J. Hein, \emph{Asymptotically conical Calabi-Yau manifolds, I}, Duke Math. J. {\bf 162} (2013), 2855--2902.}
\bibitem{DFS}{V. Datar, X. Fu, J. Song, \emph {K\"ahler-Einstein metric near an isolated log canonical singularity}, J. Reine Angew. Math. {\bf 797} (2023), 79--116.}
\bibitem{DiNeGueGue}{E. Di Nezza, V. Guedj, H. Guenancia, \emph{Families of singular {K}\"{a}hler-{E}instein metrics}, J. Eur. Math. Soc. {\bf 25} (2023), 2697--2762.}
\bibitem{DS2}{S. Donaldson, S. Sun, \emph{Gromov-Hausdorff limits of Kähler manifolds and algebraic geometry, II}, J. Differ. Geom. {\bf 107} (2017), 327--371.}
\bibitem{FHJ}{X. Fu, H.-J. Hein, X. Jiang, \emph{Asymptotics of Kähler-Einstein metrics on complex hyperbolic cusps}, Calc. Var. PDE {\bf 63} (2024), Paper No. 6.}
\bibitem{GM}{M. Giaquinta, L. Martinazzi, \emph{An introduction to the regularity theory for elliptic systems, harmonic maps and minimal graphs}, Edizioni della Normale, Pisa, 2005.}
\bibitem{HaJa}{U. Hamenstädt, F. Jäckel, \emph{Stability of Einstein metrics and effective hyperbolization in large Hempel distance}, preprint (2022), arXiv:2206.10438.}
\bibitem{Hein}{H.-J. Hein, \emph{Gravitational instantons from rational elliptic surfaces}, J. Amer. Math. Soc. {\bf 25} (2012), 355--393.}
\bibitem{HS}{H.-J. Hein, S. Sun, \emph{Calabi-Yau manifolds with isolated conical singularities}, Publ. Math. IHÉS {\bf 126} (2017), 73--130.}
\bibitem{HSVZ}{H.-J. Hein, S. Sun,  J. Viaclovsky, R. Zhang, \emph{Nilpotent structures and collapsing Ricci-flat metrics on the K3 surface}, J. Amer. Math. Soc. {\bf 35} (2022), 123--209.}
\bibitem{HSVZ2}{H.-J. Hein, S. Sun,  J. Viaclovsky, R. Zhang, \emph{Asymptotically Calabi metrics and weak Fano manifolds}, preprint (2021), arXiv:2111.09287, to appear in Geom. Topol.}
\bibitem{HT20}{H.-J. Hein, V. Tosatti, \emph{Higher-order estimates for collapsing Calabi-Yau metrics}, Camb. J. Math. {\bf 8} (2020), 683--773.}
\bibitem{Hummel}{C. Hummel, V. Schroeder, \emph{Cusp closing in rank one symmetric spaces}, Invent. Math. {\bf 123} (1996), 283--307.}
\bibitem{KS}{A. Kas, M. Schlessinger, \emph{On the versal deformation of a complex space with an isolated singularity}, Math. Ann. {\bf 196} (1972), 23--29.}
\bibitem{rkoba}{R. Kobayashi, \emph{Einstein-Kaehler metrics on open algebraic surfaces of general type}, Tohoku Math. J. {\bf 37} (1985), 43--77.}
\bibitem{KSB}{J. Kollár, N. Shepherd-Barron, \emph{Threefolds and deformations of surface singularities}, Invent. Math. {\bf 91} (1988), 299--338.}
\bibitem{Kova}{S. Kovács, \emph{Answer to MathOverflow Question 56019}, \href{https://mathoverflow.net/questions/56019/degenerations-of-smooth-projective-varieties}{https://mathoverflow.net/questions/56019/degenerations-of-smooth-projective-varieties}, 2011.}
\bibitem{Mandel}{H. Mandel, \emph{Degenerations of negative Kähler-Einstein surfaces}, J. Lond. Math. Soc. {\bf 109} (2024), Paper e12818.}
\bibitem{MO}{J. Milnor, P. Orlik, \emph{Isolated singularities defined by weighted homogeneous polynomials}, Topology {\bf 9} (1970), 385--393.}
\bibitem{OzuchThesis}{T. Ozuch, \emph{Noncollapsed degeneration of Einstein $4$-manifolds, II}, Geom. Topol. {\bf 26} (2022), 1529--1634.}
\bibitem{Pinkham}{H. Pinkham, \emph{Deformations of algebraic varieties with $G_m$ action}, Astérisque, vol. 20, Soc. Math. de France, Paris, 1974.}
\bibitem{Ruppenthal}{J. Ruppenthal, \emph{A $\overline{\partial}$-theoretical proof of Hartogs’ extension theorem on Stein spaces with isolated singularities}, J. Geom. Anal. {\bf 18} (2008), 1127--1132.}
\bibitem{Saito}{K. Saito, \emph{Einfach-elliptische Singularitäten}, Invent. Math. {\bf 23} (1974), 289--325.}
\bibitem{Savin}{O. Savin, \emph{Small perturbation solutions for elliptic equations}, Comm. Part. Differ. Eq. {\bf 32} (2007), 557--578.}
\bibitem{Song}{J. Song, \emph{Degeneration of Kähler-Einstein manifolds of negative scalar curvature}, preprint (2017), arXiv:1706.01518.}
\bibitem{SSW2}{J. Song, J. Sturm, X. Wang, \emph{Continuity of the Weil-Petersson potential}, preprint (2020), arXiv:2008.11215.}
\bibitem{spotti}{C. Spotti, \emph{Deformations of nodal Kähler-Einstein del Pezzo surfaces with discrete automorphism groups}, J. Lond. Math. Soc. {\bf 89} (2014), 539--558.}
\bibitem{Thu}{W. Thurston, \emph{Three-dimensional manifolds, Kleinian groups and hyperbolic geometry}, Bull. Amer. Math. Soc. (N.S.) {\bf 6} (1982), 357--381.}
\bibitem{TY0} G. Tian, S.-T. Yau, \emph{Existence of K\"ahler-Einstein metrics on complete K\"ahler manifolds and their applications to algebraic geometry}, Mathematical aspects of string theory (San Diego, CA, 1986), 574--628, 
Adv. Ser. Math. Phys., 1, World Sci. Publishing, Singapore, 1987.
\bibitem{TY}{G. Tian, S.-T. Yau, \emph{Complete K\"ahler manifolds with zero Ricci curvature, I}, J. Amer. Math. Soc. {\bf 3} (1990), 579--609.}
\bibitem{CY2}{S.-T. Yau, \emph{Harmonic functions on complete Riemannian manifolds}, Comm. Pure Appl. Math. {\bf 28} (1975), 201--228.}
\bibitem{Yau}{S.-T. Yau, \emph{On the Ricci curvature of a compact K\"ahler manifold and the complex Monge-Amp\`ere equation, I},
Comm. Pure Appl. Math. {\bf 31} (1978), 339--411.}
\end{thebibliography}
\end{document}